\documentclass[11pt,reqno]{amsart}
\usepackage[usenames,dvipsnames,svgnames,table]{xcolor}
\usepackage{amssymb,amsfonts,amsmath,amscd}
\usepackage{verbatim}
\usepackage[all]{xy}
\usepackage{hyperref}
\usepackage{fancyhdr}
\usepackage{epsfig}
\usepackage{helvet} 
\newtheorem{theorem}{Theorem}[section] 
\newtheorem{lemma}[theorem]{Lemma}     
\newtheorem{corollary}[theorem]{Corollary}
\newtheorem{proposition}[theorem]{Proposition}

\theoremstyle{definition}

\newtheorem{definition}[theorem]{Definition}
\newtheorem{definitions}[theorem]{Definitions}

\newtheorem{remark}[theorem]{Remark}

\newtheorem{purpose}[theorem]{Purpose}

\newtheorem{example}[theorem]{Example}
\newtheorem{notation}[theorem]{Notation} 
\newtheorem{assumption}[theorem]{Assumption}
\newtheorem{discussion}[theorem]{Discussion}

\newcommand{\mbn}{\mbox{$\mathbb{N}$}}
\newcommand{\mbz}{\mbox{$\mathbb{Z}$}}
\newcommand{\mbq}{\mbox{$\mathbb{Q}$}}
\newcommand{\mbr}{\mbox{$\mathbb{R}$}}
\newcommand{\mbc}{\mbox{$\mathbb{C}$}}
\newcommand{\mbk}{\mbox{$\mathbb{K}$}}
\newcommand{\mba}{\mbox{$\mathbb{A}$}}

\newcommand{\mbg}{\mbox{$\mathbb{G}$}}
\newcommand{\mbv}{\mbox{$\mathbb{V}$}}
\newcommand{\mbe}{\mbox{$\mathbb{E}$}}
\newcommand{\mbone}{\mbox{$1\mkern-4.6mu {\rm l}$}}
\newcommand{\lc}{\mbox{${\bf L}{\rm c}$}}
\newcommand{\lm}{\mbox{${\bf L}{\rm m}$}}
\newcommand{\lmod}{\mbox{${\bf L}$}}
\newcommand{\lt}{\mbox{${\bf L}{\rm t}$}}
\newcommand{\mcb}{\mbox{$\mathcal{B}$}}
\newcommand{\mcc}{\mbox{$\mathcal{C}$}}

\newcommand{\mcl}{\mbox{$\mathcal{L}$}}
\newcommand{\mcm}{\mbox{$\mathcal{M}$}}

\newcommand{\icb}{\mbox{$\mathrm{ICB}$}}
\newcommand{\pcb}{\mbox{$\mathrm{PCB}$}}
\newcommand{\cb}{\mbox{$\mathrm{CB}$}}

\newcommand{\mfm}{\mbox{$\mathfrak{m}$}}

\newcommand{\mbfh}{\mbox{${\bf H}$}}

\newcommand{\ud}{\mbox{${\rm d}^\vee$}}
\newcommand{\dif}{\mbox{$\partial$}}
\newcommand{\sdif}{\mbox{\tiny $\partial$}}

\newcommand{\I}{\mbox{${\rm I}$}}
\newcommand{\syz}{\mbox{${\rm Syz}$}}
\newcommand{\Lcm}{\mbox{${\rm LCM}$}}

\newcommand{\rad}{\mbox{${\rm rad}$}}
\newcommand{\pd}{\mbox{${\rm projdim}$}}

\newcommand{\grade}{\mbox{${\rm grade}$}}

\newcommand{\rank}{\mbox{${\rm rank}$}}

\newcommand{\nulls}{\mbox{${\rm Nullspace}$}}
\newcommand{\adj}{\mbox{${\rm adj}$}}
\newcommand{\cyc}{\mbox{${\rm Cyc}$}}
\newcommand{\hull}{\mbox{${\rm Hull}$}}

\newcommand{\srle}{\mbox{${\rm srle}$}}
\newcommand{\wrlo}{\mbox{${\rm wrlo}$}}

\title{Minimal free resolutions of lattice ideals of digraphs}

\author{Liam O'Carroll and Francesc Planas-Vilanova}


\begin{document}

\begin{abstract}
Based upon a previous work of Manjunath and Sturmfels for a finite,
complete, undirected graph, and a refined algorithm by Er\"ocal,
Motsak, Schreyer and Steenpa{\ss} for computing syzygies, we display a
free resolution of the lattice ideal associated to a finite, strongly
connected, weighted, directed graph. Moreover, the resolution is
minimal precisely when the digraph is strongly complete.
\end{abstract}

\maketitle 

\section{Introduction}

The aim of this paper is to present a free resolution of the lattice
ideal $I(\mcl)$ associated to the lattice $\mcl$ spanned by the
columns of the Laplacian matrix $L$ of a finite, strongly connected,
weighted, directed graph $\mbg$, and to show that this resolution is
minimal if and only if the graph $\mbg$ is strongly complete. Our
techniques are a novel mixture of a mild generalization of previous
work by Manjunath and Sturmfels for finite, complete, weighted,
undirected graphs and a recent refined algorithm for computing
syzygies due to Er\"ocal, Motsak, Schreyer and Steenpa{\ss}. (The
meaning of the terminology we use is presented in
Section~\ref{preliminaries}.)

First we present some background.

The Abelian Sandpile Model (ASM) is a game played on a finite,
weighted, connected, undirected graph $\mathbb{G}$ with $n$ vertices,
that realizes the dynamics implicit in the discrete Laplacian matrix
$L$ of the graph, this matrix being an integer matrix that is
symmetric. Each configuration is a mapping from the vertices of the
graph into the set of non-negative integers. The value of the mapping
at a vertex may be considered as the number of grains of sand on a
sandpile placed at the vertex. The game's evolution is given by a
`toppling' rule: each vertex containing at least as many grains as it
has neighbours distributes one grain to each of them.  (This process
has also been called `sand-firing' or `chip-firing'.)

The ASM was introduced by Bak, Tang and Wiesenfeld \cite{btw} in the
context of self-organized critical phenomena in statistical physics
and has been studied extensively since. In the seminal paper
\cite{crs}, Cori, Rossin and Salvy, enumerated the vertices of the
undirected graph $\mbg$ using a natural metric, considered the lattice
$\mcl$ spanned over $\mbz$ by the rows of the symmetric Laplacian
matrix $L$ and introduced a so-called toppling (or lattice) ideal
$I(\mcl)$ in the polynomial ring $\mbq[x]:=\mbq[x_1,\ldots
  ,x_n]$. This ideal encodes configurations with monomials and
topplings with binomials. In particular, they showed that the set of
toppling binomials constituted a Gr\"obner basis for $I(\mcl)$ with
respect to a natural monomial ordering in $\mbq[x]$, which may be
taken to be the graded reverse lexicographic ordering.

In their recent survey paper \cite{ppw}, Perkinson, Perlman and Wilmes
associated to the Laplacian matrix $L$ of a directed graph $\mbg$ a
collection of binomial equations, and noted that these binomials span
the lattice ideal $I(\mcl)$ in $\mbc[x_1,\ldots ,x_n]$, where $\mcl$
again is the lattice spanned over $\mbz$ by the rows of the (now
non-necessarily symmetric) Laplacian matrix $L$, using the theory of
lattice ideals extensively in their discussion.

More recently, and for complete undirected graphs, i.e., any
two distinct vertices are connected by an edge, Manjunath and
Sturmfels \cite{ms} gave a minimal free resolution of
$\mbk[x]/I(\mcl)$, $\mbk$ an arbitrary field, showing also
that the binomials in question formed a minimal generating set and a
minimal Gr\"obner basis of $I(\mcl)$ in the graded reverse
lexicographic ordering. Subsequently, Manjunath, Schreyer and Wilmes
\cite{msw}, and Mohammadi and Shokrieh \cite{mos}, gave a minimal free
resolution of $\mbk[x]/I(\mcl)$ in the case where the underlying
undirected graph was, more generally, connected. We remark here that
the techniques used in these papers are a mix of results from
combinatorial graph theory and the theory of Gr\"obner bases (e.g., the
Schreyer algorithm as a tool to proving the exactness of a complex).

The transition from undirected to directed graphs introduces extra
technicalities, and means that we have to rely heavily on techniques
from the theory of Gr\"obner bases and homogeneous Commutative
Algebra. Our general setting is the following: $\mbg$ is a finite,
strongly connected, weighted, directed graph, $L$ stands for its
Laplacian matrix, $\mcl$ is the lattice spanned by the columns of $L$
and $I(\mcl)$ is the lattice ideal associated to $\mcl$.

First of all, we see that, having again enumerated the vertices of the
graph $\mbg$ using a generalized natural metric, $I(\mcl)$ is
homogeneous in a weighted reverse lexicographic order on $\mbk[x]$,
$\mbk$ an arbitrary field, the weighting coming from the structure of
$\adj(L)$ (see Remark~\ref{ILhomogeneous}). We then use basic aspects
of homogeneous Commutative Algebra and previous work of ours to show
that $I(\mcl)$ is a Cohen-Macaulay ideal of dimension $1$, so that
$\pd_{\mathbb{K}[x]}(\mbk[x]/I(\mcl))=n-1$: see
Corollary~\ref{projdim}.

This earlier work shows up as follows. We find (see
Sections~\ref{pcbcb}, \ref{ibm}) that the matrices $L$ we deal with
are what we called Critical Binomial Matrices ($\cb$ matrices, in
brief) in previous work, and that $\mbg$ is, further, strongly
complete precisely when the matrices $L$ are what we previously called
Positive Critical Binomial Matrices ($\pcb$ matrices, in brief): see
\cite{op}, \cite{opv}. In fact, as we shall see in
Section~\ref{preliminaries}, since we treat the case where $\mbg$ is
strongly connected, the $\cb$ matrix $L$ has a further property which
is precisely Irreducibility, hence the matrices $L$ are Irreducible
Critical Binomial Matrices ($\icb$ matrices, in brief), whose
properties are developed in the paper: see, for example,
Sections~\ref{ibm-adj}, \ref{hom-scd}, \ref{block-icb}. Wherever
possible, we give new proofs and develop new approaches to this
previous work: see, for example, Propositions~\ref{adj(L)rowconstant},
\ref{radical}.

There are two main ingredients in the proof of our main theorem
(Theorem~\ref{main}). The first is the fact that the resolving complex
$\cyc_n$ of \cite{ms} (treated with a little care in our more general
context, as the Laplacian matrix $L$ need no longer be symmetric)
continues to be a complex and, moreover, is exact in degree zero: see
Sections~\ref{cyc-complex-cb}, \ref{exact-degree0}.  The second
ingredient is a refined choice of Schreyer syzygies: see
Section~\ref{schreyer-alg}. We find, via an inductive proof that keeps
track of the leading terms of the Manjunath-Sturmfels boundaries, that
in a very satisfying yet surprising manner, these boundaries pair off
with appropriate Schreyer syzygies, enabling one to deduce in
Theorem~\ref{main} that $\cyc_n$ is in fact acyclic.  It then follows
easily that $\cyc_n$ is minimal if and only if $\mbg$ is, further,
strongly complete (see Corollary~\ref{min-resolution}).

We finish this introduction by emphasizing that, usually, and
alternatively to the works of, e.g., Perkinson, Perlman and Wilmes in
\cite{ppw}, or Backman and Manjunath in \cite{bm}, we have taken
$\mcl$ to be the lattice spanned by the columns of the Laplacian
matrix $L$ and not the rows (see Sections~\ref{general} and
\ref{dictionary}). Note that, for the undirected graph case this is an
irrelevant matter, because $L$ being a symmetric matrix, then the
lattice spanned by the columns or the lattice spanned by the rows do
coincide. It is worth to mention here that studying the case of the
lattice spanned by the columns can be related to a column chip-firing
game. This is done by Asadi and Backman in \cite[Introduction and
  Section~3.2]{ab}. An interesting future research would be to pass
from columns to rows. As a possible approach one could consider the
lattice $\mcl^{\top}$, spanned by the columns of the transpose of the
Laplacian matrix $L^{\top}$, and then proceed in a similar vein as in
\cite{opv}. See also the connection with Eulerian digraphs (the
indegree and outdegree coincide for every vertex) in the
aforementioned work of Asadi and Backman (\cite[Theorem~3.18]{ab}).

Throughout, we use the explicit case $n=4$ as a running illustrative
example.

\bigskip 

{\em Liam O'Carroll died on October 2017, while this manuscript was under
revision. Liam and I started to work together several years ago. To
me, all these years of collaboration with him have been, above all,
years of enjoying our friendship. I will greatly miss him.}

\section{Notations and Preliminaries}\label{preliminaries}

\subsection{General setting}\label{general}
$\phantom{+}$\bigskip 

The following notations will remain in force throughout the paper.  We
denote the set of non-negative integers by $\mbn$ and the set of
positive integers by $\mbn_+$. Let $n$ be a positive integer with
$n\geq 3$, to avoid trivialities, and let $[n]=\{1,\ldots ,n\}$. Any
vector $\alpha\in\mbz^n$ can be written uniquely as
$\alpha=\alpha^+-\alpha^-$, where $\alpha^+=\max(\alpha,0)\in\mbn^n$
and $\alpha^-=-\min(\alpha,0)\in\mbn^n$ have disjoint support. Set
$\mbone=(1,\ldots ,1)$. 

Let $\mbk$ be a field and let $\mba=\mbk[x_1,\ldots ,x_n]=\mbk[x]$ be the
polynomial ring over $\mbk$ in indeterminates $x_1,\ldots ,x_n$. The
maximal ideal generated by $x_1,\ldots ,x_n$ will be denoted
$\mfm=(x_1,\ldots ,x_n)$.  For any $\alpha=(\alpha_1,\ldots
,\alpha_n)\in\mbn^n$, we write $x^\alpha$ as shorthand for
$x_1^{\alpha_1}\cdots x_n^{\alpha_n}$. By a binomial in $\mba$, we
understand a polynomial of $\mba$ with at most two terms, say $\lambda
x^\alpha-\mu x^\beta$, where $\lambda, \mu\in\mbk$ and
$\alpha,\beta\in\mbn^{n}$. A binomial ideal of $\mba$ is an ideal
generated by binomials.

Let $M$ be an $n\times s$ integer matrix, $M=(m_{i,j})$; that is,
$m_{i,j}\in\mbz$, for all $i=1,\ldots ,n$, $j=1,\ldots ,s$. We will
say that $M$ is non-negative, and write $M\geq 0$, if all the entries
in $M$ are non-negative, i.e., $m_{i,j}\geq 0$, for all
$i,j$. Similarly, $M$ is said to be positive, $M>0$, if $m_{i,j}>0$,
for all $i,j$.

We will denote by $m_{i,*}$ and $m_{*,j}$ the $i$-th row and $j$-th
column, repectively, of $M$. We will denote by $\mcm\subset \mbz^n$
the additive subgroup of $\mbz^n$ spanned by the columns of $M$,
commonly called the lattice of $\mbz^{n}$ defined by the columns of
$M$.

Furthermore, $f_{m_{*,j}}=x^{(m_{*,j})^+}-x^{(m_{*,j})^-}$ will denote
the binomial of $\mba$ defined by the $j$-th column of $M$. The ideal
$I(M)=(f_{m_{*,j}}\mid j=1,\ldots ,s)$, generated by the binomials
$f_{m_{*,j}}$, is called the binomial ideal associated to the matrix
$M$. The binomial ideal $I(\mcm)=(x^{m^+}-x^{m^-}\mid m\in\mcm)$ is
called the lattice ideal associated to $\mcm$. Clearly $I(M)\subseteq
I(\mcm)$.

By a permutation of a square $n\times n$ matrix $M$, we understand a
permutation of the rows of $M$ together the same permutation of the
columns. Equivalently, a permutation of $M$ is a square matrix
$P^{\top}MP$, where $P$ is a permutation matrix resulting from a
permutation of the columns of the $n\times n$ identity matrix $\I$.

\subsection{(Positive) critical binomial matrices and ideals}
\label{pcbcb}$\phantom{+}$\bigskip

\noindent Recall that a critical binomial matrix, $\cb$ matrix for
short, is an $n\times n$ integer matrix defined as follows:

\begin{eqnarray}\label{cbmatrix}
L=\left(\begin{array}{rrrr}
a_{1,1}&-a_{1,2}&\ldots&-a_{1,n}\\
-a_{2,1}&a_{2,2}&\ldots&-a_{2,n}\\
\vdots\phantom{+}&\vdots\phantom{+}&&\vdots\phantom{+}\\
-a_{n,1}&-a_{n,2}&\ldots&a_{n,n}
\end{array}\right),
\end{eqnarray}
where $a_{i,i}>0$, $a_{i,i}=\sum_{j\neq i}a_{i,j}$, $a_{i,j}\geq 0$,
for all $i,j=1,\ldots ,n$, and, for each column of $L$, at least one
off-diagonal entry is nonzero (see \cite{opv}). Alternatively, an
$n\times n$ integer matrix $L$ is a $\cb$ matrix whenever $L=D-B$,
where $D$ is a positive, diagonal $n\times n$ integer matrix and $B$
is a non-negative $n\times n$ integer matrix with zero entries in the
diagonal, nonzero columns $b_{*,1},\ldots ,b_{*,n}$, and such that
$(D^{-1}B)\mbone^{\top}=\mbone^{\top}$. In particular, $D^{-1}B$ is an
stochastic matrix. The expression $L=D-B$, which is clearly unique,
will be called the ``$(D,B)$ form of $L$''.

Observe that a permutation $P^{\top}LP$ of a $\cb$ matrix $L$ is again
a $\cb$ matrix. Indeed, writing $L=D-B$ in its $(D,B)$ form, then
$P^{\top}LP=P^{\top}DP-P^{\top}BP$, and $P^{\top}DP>0$ is diagonal and
$P^{\top}BP\geq 0$, with zero entries in the diagonal, nonzero columns
and $(P^{\top}DP)^{-1}(P^{\top}BP)\mbone^{\top}=\mbone^{\top}$.

Since $a_{i,i}>0$ and $b_{*,1},\ldots ,b_{*,n}$ are nonzero, for each
column $l_{*,1},\ldots ,l_{*,n}$ of an $n\times n$ $\cb$ matrix $L$,
the binomial
\begin{eqnarray*}
f_{l_{*,j}}=x^{(l_{*,j})^+}-x^{(l_{*,j})^-}=x_j^{a_{j,j}}-x_1^{a_{1,j}}\cdots
x_{j-1}^{a_{j-1},j}x_{j+1}^{a_{j+1,j}}\cdots x_n^{a_{n,j}},
\end{eqnarray*}
has two non-constant terms. The binomial ideal $I(L)$, generated by
the $f_{l_{*,j}}$, is called the critical binomial ideal, $\cb$ ideal
for short, associated to $L$ (see \cite{opv}).

A special class of $\cb$ matrices is the following. When all the
coefficients $a_{i,j}$ are positive, the matrix $L$ is said to be a
positive critical binomial matrix ($\pcb$ matrix, for short). The
ideal $I(L)$ is called the positive critical binomial ideal ($\pcb$
ideal, for short) associated to $L$. These ideals were extensively
studied in \cite{op} (see also \cite{opv}).

\subsection{Irreducible binomial matrices}\label{ibm}
$\phantom{+}$\bigskip

\noindent A square $n\times n$ matrix $M=(m_{i,j})$ is called
reducible if the set of indices $[n]=\{1,\ldots ,n\}$ can be split
into two non-empty disjoint sets $I,J$ such that $m_{i,j}=0$, for all
$i\in I$ and $j\in J$. The matrix $M$ is called irreducible if it is
not reducible. Equivalently, $M$ is reducible if and only if there
exists a permutation matrix $P$ such that
\begin{eqnarray}\label{reducedform}
\newcommand\y{\cellcolor{gray!40}}
P^{\top}MP=\left(\begin{array}{cc}\y M_{I,I}&0\\
\y M_{J,I}&\y M_{J,J}\end{array}\right),
\end{eqnarray}
where $M_{I,I}$ and $M_{J,J}$ are square matrices (see, e.g.,
\cite[Chapter~XIII, p.50]{gantmacher}; see also \cite[\S~3.2]{br}).
Here we understand $M_{I,J}$ as the submatrix of the matrix $M$
defined by the set of rows $I$ and the set of columns $J$.  Of course,
this is equivalent to saying that there exists a permutation matrix
$Q$ that permutes columns and then corresponding rows, such that
\begin{eqnarray*}
\newcommand\y{\cellcolor{gray!40}}
Q^{\top}MQ=\left(\begin{array}{cc}
\y M_{J,J}&\y M_{J,I}\\0&\y M_{I,I}\end{array}\right).
\end{eqnarray*}

Clearly, adding a diagonal matrix, multiplying by a diagonal matrix,
or transposing, are three operations that preserve the condition of
irreducibility on a matrix.

We will say that $L$ is an irreducible critical binomial matrix
($\icb$ matrix, for short), if $L$ is a $\cb$ matrix which is
irreducible. Write $L=D-B$ in its $(D,B)$ form. If $L$ is irreducible,
then $(D^{-1}B)^{\top}$ is also irreducible, and conversely.

Clearly, we have the chain of implications linking these three
properties: $\pcb\Rightarrow\icb\Rightarrow\cb$. (When $n=2$, the
three classes coincide.) If $n=3$, then the condition $\icb$ is
equivalent to the condition $\cb$, but there are $\icb$ matrices which
are not $\pcb$. If $n=4$, it is easy to see that the three classes are
distinct.

\begin{remark}\label{remarkreducedform}
Let $L$ be a reducible $\cb$ matrix, $n\geq 4$. After a permutation,
we can suppose that $L$ can be written as
\begin{eqnarray}\label{eqreducedform}
\newcommand\y{\cellcolor{gray!40}}
L=\left(\begin{array}{cc} 
\y L_{I,I}&0\\
\y L_{J,I}&\y L_{J,J}
\end{array}\right),
\end{eqnarray}
where $L_{I,I}$ is a square $r\times r$ integer matrix, $L_{J,J}$ is a
square $s\times s$ integer matrix, with $r+s=n$ and $1\leq r,s\leq n$.
Note that $L_{I,I}\mbone^{\top}=0$. In particular, since the sum of
the entries of the first row is zero and $a_{1,1}>0$, then $r\geq
2$. Similarly, $s\geq 2$, otherwise the last column of $L$ would not
have an off-diagonal nonzero entry.
\end{remark}

\subsection{The digraphs we will deal with}
$\phantom{+}$\bigskip

\noindent All graphs $\mbg=(\mbv,\mbe,w)$ we will deal with are
assumed to be finite, weighted, directed graphs. That is,
$\mbv=\{v_1,\ldots ,v_n\}$ is the finite set of vertices; $\mbe$ is
the set of directed arcs of $\mbg$, which are ordered pairs
$(v_i,v_j)$, with $i\neq j$; and $w$ is a weight function that
associates to a every arc $e$ in $\mbe$ a weight
$w_e\in\mbn_+$. Moreover, we will always assume that $\mbg$ has no
loops, sources or sinks. That is, there are no arcs from $v_i$ to
$v_i$, for any $i$; and for every vertex $v_i$, there exists at least
a $j\in\{1,\ldots,n\}\setminus\{i\}$ and an arc $(v_j,v_i)\in \mbe$,
and there exists at least a $k\in\{1,\ldots,n\}\setminus\{i\}$ and an
arc $(v_i,v_k)\in \mbe$.  (See, e.g., \cite{br} or \cite{gr}, for all
unexplained notations on graph theory.)

\begin{assumption}\label{digraph}
Thus, from now on, a digraph $\mbg$ will mean a finite, weighted,
directed graph, without loops, sources or sinks.
\end{assumption}

A digraph $\mbg$ is said to be strongly connected if any two
(distinct) vertices can be joined by a (directed) path. A directed
path from vertex $v_i$ to a distinct vertex $v_j$ will be denoted by
$v_i\to\cdots\to v_j$. The {\em unweighted length} of the path
$v_i\to\cdots\to v_j$ is the number of arcs which constitute it,
without reference to their weights.  The {\em unweighted distance}
$\ud(v_i,v_j)$ between distinct vertices $v_i$ and $v_j$ is defined to
be the minimum unweighted length of a directed path connecting $v_i$
with $v_j$. Thus, $\ud(v_i,v_j)=1$ if and only if $(v_i,v_j)\in\mbe$.
It may happen that $\ud(v_i,v_j)\neq \ud(v_j,v_i)$.

A directed graph $\mbg$ is said to be strongly complete if any two
(distinct) vertices can be joined by an arc, that is, $(v_i,v_j)\in
\mbe$, for all $i,j=1,\ldots ,n$, $i\neq j$.

The {\em adjacency matrix} $A(\mbg)$ of $\mbg$ is the $n\times n$
integer matrix given by $A(\mbg)_{i,j}=w_e$ if $e=(v_i,v_j)\in \mbe$,
and $A(\mbg)_{i,j}=0$ otherwise. Since $\mbg$ has no loops,
$A(\mbg)_{i,i}=0$; since $\mbg$ has no sources, for each column of
$A(\mbg)$, at least one off-diagonal entry is nonzero; since $\mbg$
has no sinks, for each row of $A(\mbg)$, at least one off-diagonal
entry is nonzero. The {\em (out)degree matrix} $D(\mbg)$ of $\mbg$ is
the $n\times n$ integer diagonal matrix defined by
$D(\mbg)_{i,i}=\sum_{e\in O(v_i)}w_e$, the weighted (out)degree of the
vertex $v_i$, where $O(v_i)$ is the set of vertices $v_j$, $j\neq i$,
such that $(v_i,v_j)\in \mbe$. Since $\mbg$ has no sinks,
$D(\mbg)_{i,i}\neq 0$ for every $i=1,\ldots, n$. The {\em Laplacian
  matrix} $L(\mbg)$ of $\mbg$ is defined as $L(\mbg)=D(\mbg)-A(\mbg)$.

\subsection{Our dictionary: CB matrices - digraphs}\label{dictionary}
$\phantom{+}$\bigskip

\noindent Clearly, the Laplacian matrix $L(\mbg)$ of a digraph $\mbg$
is a $\cb$ matrix, sometimes also denoted $L_{\mathbb{G}}$ (recall
that we are using Assumption~\ref{digraph}). Conversely, given an
$n\times n$ $\cb$ matrix $L=D-B$, one can define a digraph $\mbg_L$ of
$n$ vertices $\mbv=\{v_1,\ldots ,v_n\}$ with adjacency matrix $B$ in
an obvious way, namely, there exists a directed edge $v_i\to v_j$
between distinct vertices $v_i$ and $v_j$ if and only if $a_{i,j}>0$,
in which case the edge $v_i\to v_j$ is assigned weight $a_{i,j}$.

Moreover, $L$ is an $\icb$ matrix if and only if $\mbg_L$ is strongly
connected (see \cite{cv}, for a recent related work). Indeed, first
remark that a permutation of $L$ is just a re-enumeration of the
vertex set of $\mbg$ and a corresponding relabelling of the directed
arcs. Now suppose that $L$ is reducible and let $I,J$ be the partition
on $[n]$ such that $a_{i,j}=0$, for all $i\in I$ and all $j\in
J$. Then it follows that there is no directed path from a vertex $v_i$
to a vertex $v_j$, with $i\in I$ and $j\in J$. Conversely, suppose
that $\mbg$ is not strongly connected and, after a re-enumeration of
the vertices and corresponding relabelling of the directed arcs,
suppose there is no directed path from $v_1$ to $v_n$. Let
$I=\{1\}\cup\{i\in [n]\mid \exists \; v_1\to\cdots\to v_i\}$ and let
$J=[n]\setminus I$. After a re-enumeration of the vertices and
corresponding relabelling of the directed arcs, we can suppose that
$I=\{1,2,\ldots ,r\}$ and $J=\{r+1,\ldots ,n\}$. Clearly, there is not
an arc from $v_i$, $i\in I$, to $v_j$, $j\in J$ (otherwise there would
be a directed path $v_1\to\cdots\to v_i\to v_j$ and $j$ would belong
to $I$, a contradiction). Therefore $L_{I,J}=0$ and $L$ is reducible
(see, e.g., \cite[Theorem~3.2.1]{br}, for more details).

Finally, $L$ is a $\pcb$ matrix if and only if $\mbg_L$ is strongly
complete.

Observe that if $\mbg$ is undirected, then $L_{\mathbb{G}}$ is
symmetric, and the discussion above can be specialized to this case in
an obvious manner (see, e.g., \cite{ms}).

\subsection{Irreducible critical binomial matrices and their adjugate 
matrices}\label{ibm-adj} $\phantom{+}$\bigskip

\noindent Given an $n\times n$ integer matrix $L$, let $L_{i,j}$
denote the $(n-1)\times (n-1)$ matrix obtained from $L$ by eliminating
the $i$-th row and the $j$-th column of $L$. Set $\adj(L)=
((-1)^{i+j}\lvert L_{j,i}\rvert)$ to be the adjugate matrix of $L$,
sometimes also called the adjoint matrix, where $\lvert L_{j,i}\rvert$
stands for the determinant $\det(L_{j,i})$. For convenience, we write
$\epsilon_{i,j}$ to denote $(-1)^{i+j}$.

In the next set of results, invertibility, ranks or linear subspaces
are thought of over $\mbq$. Observe that if $\rank(L)=n$, then
$\adj(L)$ is invertible (since $L\cdot \adj(L)=\rvert
L\rvert\cdot\I$). If $\rank(L)=n-1$, then, by definition, at least
one $(n-1)\times (n-1)$ minor of $L$ is nonzero, thus $\adj(L)\neq 0$;
moreover, $\rank(\adj(L))=1$ because $L\cdot \adj(L)=\lvert
L\rvert\cdot \I=0$ and so all the columns of $\adj(L)$ belong to the
nullspace of $L$, which is of dimension $1$. If $\rank(L)\leq n-2$,
then all the $(n-1)\times (n-1)$ minors of $L$ are zero, so
$\adj(L)=0$. In particular, $L$ is invertible if and only if $\adj(L)$
is invertible.

The next result can be deduced from \cite[Theorem~7.6]{opv}. We
present here a more direct proof, that does not involve a wider class
of matrices. 

\begin{proposition}\label{adj(L)rowconstant}
Let $L$ be an $n\times n$ integer matrix and let $\adj(L)$ be its
adjugate matrix.
\begin{itemize}
\item[$(a)$] If $L\mbone^{\top}=0$, then all the rows of $\adj(L)$ are
  equal.
\item[$(b)$] If $L$ is a $\cb$ matrix, then, further, $\adj(L)\geq 0$.
\item[$(c)$] If $L$ is an $\icb$ matrix, then, further, $\rank(L)=n-1$
  and $\adj(L)>0$.
\end{itemize}
\end{proposition}
\begin{proof}
If $\rank(L)\leq n-2$, then $\adj(L)=0$ and all the rows are
equal. Suppose that $\rank(L)\geq n-1$. Since $L\mbone^{\top}=0$, then
$\rank(L)=n-1$ and so $\rank(\adj(L))=1$. Moreover, $L\cdot\adj(L)=0$,
so every column of $\adj(L)$ is in the $\mbq$-linear subspace
$\nulls(L)$, which is generated by $\mbone^{\top}$. Hence there exist
$\mu_1,\ldots ,\mu_n\in\mbq$ such that the columns of $\adj(L)$ are
equal to $\mu_1\mbone^{\top}, \ldots,\mu_n\mbone^{\top}$:
\begin{eqnarray}\label{adjugate}
\adj(L)=\left(\begin{array}{rrrr} \epsilon_{1,1}\lvert
L_{1,1}\rvert&\epsilon_{1,2}\lvert L_{2,1}\rvert&\ldots&
\epsilon_{1,n}\lvert L_{n,1}\rvert\\ \epsilon_{2,1}\lvert
L_{1,2}\rvert&\epsilon_{2,2}\lvert L_{2,2}\rvert&\ldots&
\epsilon_{2,n}\lvert L_{n,2}\rvert\\ 
\vdots\phantom{+}&\vdots\phantom{+}&&\vdots\phantom{+}
\\ \epsilon_{n,1}\lvert L_{1,n}\rvert&\epsilon_{n,2}\lvert 
L_{2,n}\rvert&\ldots& \epsilon_{n,n}\lvert L_{n,n}\rvert
\end{array}\right)= 
\left(\begin{array}{cccc}
\mu_1&\mu_2&\ldots&\mu_n\\
\mu_1&\mu_2&\ldots&\mu_n\\
\vdots&\vdots&&\vdots\\
\mu_1&\mu_2&\ldots&\mu_n
\end{array}\right). 
\end{eqnarray}
Thus each row of $\adj(L)$ is equal to $\mu:=(\mu_1,\ldots ,\mu_n)$,
which proves $(a)$. Note that, since $\adj(L)$ is an integer matrix,
then necessarily $\mu_1,\ldots ,\mu_n\in\mbz$.

The proof of $(b)$ is an adaptation of \cite[Lemma~2.1]{op} (see also
\cite[Theorem~6.3,~$(c)$]{opv}). Fix $i\in\{1,\ldots ,n\}$. Let us
prove that $\mu_{i}=\lvert L_{i,i}\rvert$ is non-negative.  By the
Gershgorin Circle Theorem, every eigenvalue $\lambda$ of $L_{i,i}$
lies within at least one of the discs $\{z\in\mbc\mid \lvert
z-a_{j,j}\rvert\leq R_{j}\}$, $j\neq i$, where $R_{j}=\sum_{k\neq
  i,j}\lvert -a_{j,k}\rvert$. Observe that $R_{j}\leq a_{j,j}$ due to
the fact that $L$ is a $\cb$ matrix. If $\lambda\in\mbr$, then
$\lambda\geq 0$. If $\lambda\not\in \mbr$, then since $L_{i,i}$ is a
real matrix, the conjugate $\overline{\lambda}$ must also be an
eigenvalue of $L_{i,i}$. Since $\lvert L_{i,i}\rvert$ is the product
of the $n-1$ (possibly repeated) eigenvalues of $L_{i,i}$, it follows
that $\mu_i\geq 0$ and so $\adj(L)\geq 0$.

We now turn to the proof of $(c)$. Let us first see that $\adj(L)\neq
0$. Write $L=D-B$, with $D>0$ diagonal and $D^{-1}B$ stochastic. In
particular, the spectral radius $\rho(D^{-1}B)$ of $D^{-1}B$ is $1$
(\cite[Chapter~XIII,~p.~83]{gantmacher}). Since $L$ is irreducible,
$(D^{-1}B)^{\top}$ is an irreducible, non-negative, $n\times n$
integer matrix with $\rho((D^{-1}B)^{\top})=1$. By the
Perron-Frobenius Theorem (see, e.g.,
\cite[Chapter~XIII,~p.~53]{gantmacher},
\cite[Theorem~8.8.1,~p.~178]{gr}), $\rho((D^{-1}B)^{\top})$ is
realized by a simple eigenvalue; in particular,
$\nulls((D^{-1}B)^{\top}-\I)$ has dimension $1$. Since
$\nulls((D^{-1}B)^{\top}-\I)$ coincides with $\nulls(L^{\top})$, it
follows that $L^{\top}$ and $L$ have rank $n-1$. In particular,
$\adj(L)$ has rank $1$ and $\adj(L)\neq 0$.

Let us now prove that $\adj(L)>0$. Suppose not, i.e., there exists
some $\mu_j$ which is zero. Since $\adj(L)\neq 0$, by permuting $L$,
we can suppose that $\mu_1,\ldots ,\mu_r$ are positive and
$\mu_{r+1},\ldots ,\mu_n$ are zero, where $1\leq r<n$. Setting
$I=\{1,\ldots ,r\}$ and $J=\{r+1,\ldots ,n\}$, $L$ can be written as
\begin{eqnarray*}
L=\left(\begin{array}{cc}
L_{I,I}&L_{I,J}\\L_{J,I}&L_{J,J}
\end{array}\right),\mbox{ where }
L_{I,J}=-\left(\begin{array}{ccc}
a_{1,r+1}&\ldots &a_{1,n}\\\vdots&&\vdots\\a_{r,r+1}&\ldots &a_{r,n}
\end{array}\right),
\end{eqnarray*}
and $L_{I,I}$ and $L_{J,J}$ are two square matrices. Write
$\mu=(\mu_I,\mu_J)$, with $\mu_I=(\mu_1,\ldots ,\mu_r)$ and
$\mu_J=(\mu_{r+1},\ldots ,\mu_n)=0$. Since $\adj(L)\cdot L=0$, then
$\mu L=0$. It follows that $\mu_IL_{I,J}=0$. But $\mu_1,\ldots
,\mu_r>0$ and $a_{i,j}\geq 0$, for all $i=1,\ldots ,r$ and all
$j=r+1,\ldots, n$, so $L_{I,J}=0$. Thus $L$ is reducible, a
contradiction.
\end{proof}

\begin{notation}\label{vectornu}
Given $L$ an $\icb$ matrix, we will denote the last (and so any) row of
$\adj(L)$ by $\mu=\mu(L)=(\mu_1,\ldots ,\mu_n)\in\mbn^n_{+}$.  If
$d=\gcd(\mu)$, set $\nu(L):=\mu(L)/d$.
\end{notation}

\begin{corollary} \label{ICBvsCB}
Let $L$ be an $n\times n$ integer matrix. Then $L$ is an $\icb$ matrix
if and only if $L$ is a $\cb$ matrix and any set of $n-1$ rows of $L$
is linearly independent.
\end{corollary}
\begin{proof}
Let $F=\langle l_{1,*},\ldots,l_{n,*}\rangle$ be the $\mbq$-linear
subspace spanned by the rows of a $\cb$ matrix $L$. Writing $\adj(L)$
as in \eqref{adjugate}, and since $\adj(L)\cdot L=0$, it follows that
$\mu_1l_{1,*}+\cdots +\mu_nl_{n,*}=0$, with each $\mu_i\geq 0$.
Suppose first of all that $L$ is irreducible. Fix $i\in\{1,\ldots
,n\}$. By Proposition~\ref{adj(L)rowconstant}, $\rank(L)=n-1$ and
$\mu_i>0$, for all $i=1,\ldots ,n$. Therefore, $F$ has dimension $n-1$
and $\langle l_{1,*},\ldots \widehat{l}_{i,*},\ldots,l_{n,*}\rangle$
equals $F$. Hence the set of rows $\{l_{1,*},\ldots
\widehat{l}_{i,*},\ldots,l_{n,*}\}$ is linearly
independent. Conversely, if $n=3$, the condition $\cb$ implies the
condition $\icb$. Suppose $n\geq 4$ and that $L$ is reducible. After a
permutation, one can suppose that $L$ can be written as in
\eqref{eqreducedform}, where $L_{I,I}$ is an $r\times r$ matrix,
$2\leq r\leq n-2$, such that $L_{I,I}\mbone^{\top}=0$. In particular,
$\rank(L_{I,I}^{\top})=\rank(L_{I,I})\leq r-1$. Therefore the first
$r$ rows of $L$ are linearly dependent.
\end{proof}

\begin{remark}\label{NotICB}
Clearly, an easy adaptation of the same argument shows that if $L$ is
an $\icb$ matrix, then $L$ is a $\cb$ matrix and any set of $n-1$
columns of $L$ is linearly independent. However, the converse is not
true, as the example below shows:
\begin{eqnarray*}
\newcommand\y{\cellcolor{gray!40}}
L=\left(\begin{array}{rrrr}
\y1&\y-1&0&0\\
\y-1&\y1&0&0\\
\y-1&\y-1&\y3&\y-1\\
\y-1&\y-1&\y-1&\y3
\end{array}\right),\mbox{ corresponding to }\mbg:
\vcenter{\xymatrix{ & \bullet_2 \ar@/^/[dl] & \\ \bullet_1 \ar@/^/[ur]
    & & \bullet_3 \ar[ll] \ar@/^/[dl] \ar[ul]\\ & \bullet_4 \ar[uu]
    \ar[ul] \ar@/^/[ur]}}.
\end{eqnarray*}
\end{remark}

\subsection{Homogeneity of lattice ideals of strongly connected digraphs}
\label{hom-scd}\phantom{+}\bigskip

\begin{assumption}\label{grading}
Let $L$ be an $\icb$ matrix (or, equivalently, let $\mbg_L$ be a
strongly connected digraph). Let $\nu=\nu(L)$ be defined as is
Notation~\ref{vectornu}. From now on, we will always consider the
$\mbn$-grading on $\mba=\mbk[x_1,\ldots ,x_n]$ where each $x_i$ is
given weight $\nu_i$. Thus for a monomial $x^\alpha\in \mba$, we
define the degree of $x^\alpha$ as follows:
$\deg(x^\alpha)=\nu_1\alpha_1+\cdots+\nu_n\alpha_n$.
\end{assumption}

\begin{remark}\label{ILhomogeneous}
With this grading, the matrix ideal
$I(L)=(x^{l_{*,1}^+}-x^{l_{*,1}^-},\ldots
,x^{l_{*,n}^+}-x^{l_{*,n}^-})$ and the lattice ideal
$I(\mcl)=(x^{l^+}-x^{l^-}\mid l\in\mcl)$ are homogeneous ideals of
$\mba$ (cf. Definitions in Section~\ref{general}). Indeed, let
$l\in\mcl$. So there exists $b\in \mbz^n$ with $l^{\top}=Lb^{\top}$, a
linear combination of the columns of $L$.  Then $\nu l^{\top}=\nu
Lb^{\top}=0$, so $\nu (l^{+})^{\top}=\nu (l^{-})^{\top}$ and
$\deg(x^{l^{+}})=\nu (l^{+})^{\top}=\nu
(l^-)^{\top}=\deg(x^{l^-})$. Thus $f_l=x^{l^+}-x^{l^-}$ is
homogeneous. (See \cite[Definition~2.2]{op} or
\cite[Remark~5.6]{opv}.)
\end{remark}

Note that the hypothesis that $\mbg$ be strongly connected is
essential to define a $\mbn$-grading on $\mba=\mbk[x_1,\ldots ,x_n]$
and to ensure that $I(L)$ is homogeneous. See, for instance, the
example in Remark~\ref{NotICB}.

The following result can be deduced from
\cite[Proposition~7.6]{opv}. Here we give a shorter direct proof.

\begin{proposition}\label{radical}
Let $L$ be an $\icb$ matrix. Then $\rad(I(L),x_{i})=\mfm$, for all
$i=1,\ldots ,n$.
\end{proposition}
\begin{proof}
Fix $i\in\{1,\ldots ,n\}$. Let $\mbg_L$ be the corresponding strongly
connected digraph. Let $v_{i_1}\to\cdots\to v_{i_N}$ be a directed
path passing through all the vertices and begining in $v_{i_1}=v_i$
(there might be possibly repeated vertices, though adjacent vertices
will not be repeated). To simplify notations, write $f_i$ to denote
the binomial defined by the $i$-th column of $L$. Thus
$I(L)=(f_{i_N},\ldots ,f_{i_1})$. Write
$f_{i_1}=x_{i_1}^{a_{i_1,i_1}}-g_1$, where $g_1\in \mba$ is a monomial
divisible by at least one variable $x_j$ different from $x_{i_1}$.
Since there is a directed arc from the vertex $v_{i_1}$ to the vertex
$v_{i_2}$, then $a_{i_1,i_2}\neq 0$ and
$f_{i_2}=x_{i_2}^{a_{i_2,i_2}}-x_{i_1}^{a_{i_1,i_2}}g_2$, where
$g_2\in \mba$ denotes a monomial, possibly constant. Then
$\rad(f_{i_2},x_{i_1})=(x_{i_2},x_{i_1})$.  Similarly,
$\rad(f_{i_3},x_{i_2})=(x_{i_3},x_{i_2})$ and so on. Therefore
\begin{eqnarray*}
&&\rad(I(L),x_{i})=\rad(f_{i_N},\ldots
  ,f_{i_2},f_{i_1},x_{i_1})=\rad(f_{i_N},\ldots
  ,f_{i_2},x_{i_1},g_1)=\\ &&\rad(f_{i_N},\ldots
  ,f_{i_3},x_{i_2},x_{i_1},g_1)=\ldots =\rad(x_{i_N},\ldots
  ,x_{i_1},g_1)=\rad(\mfm,g_1)=\mfm.
\end{eqnarray*}
\end{proof}

We write $\hull(I)$ for the intersection of the isolated primary
components of $I$ (see, e.g., \cite{op} or \cite{opv}).

\begin{corollary}\label{projdim}
Let $L$ be an $\icb$ matrix. Then $I(L)$ and $I(\mcl)$ have height
$n-1$. Moreover, $I(\mcl)=\hull(I(L))$ is the intersection of the
isolated primary components of $I(L)$. In particular, $I(\mcl)$ is a
Cohen-Macaulay ideal of dimension $1$ and projective dimension $n-1$.
\end{corollary}
\begin{proof}
That $I(L)$ has height $n-1$ follows from Remark~\ref{ILhomogeneous},
Proposition~\ref{radical} and \cite[Lemma~5.1]{opv}. That $I(\mcl)$
has height $n-1$ and is the Hull of $I(L)$ follows from
\cite[Proposition~5.7]{opv}. In particular, $I(\mcl)$ is (homogeneous)
unmixed and $\mba/I(\mcl)$ is a (graded) Cohen-Macaulay ring of dimension
$1$. By \cite[Corollary~2.2.15]{bh}, $I(\mcl)$ is perfect and so the
projective dimension of $I(\mcl)$ is $n-1$.
\end{proof}

\section{The Cyc complex associated to a directed graph}

\subsection{The set of cyclically ordered partitions and an 
enumeration}\label{thesetcyc}
$\phantom{+}$\bigskip

\noindent Let us now recall some material from \cite[Section~2]{ms},
treating it however in a more general context that the one considered
there. Set $[n]=\{1,\ldots ,n\}$, take $1\leq k\leq n$, and let
$\cyc_{n,k}$ denote the set of cyclically ordered partitions of the
set $[n]$ into $k$ blocks, which we shall always take to be
non-empty. Each element of $\cyc_{n,k}$ has the form $(I_1,\ldots
,I_k)$, where $I_1\cup\ldots\cup I_k=[n]$ is a partition, and we
regard the $(I_1,\ldots ,I_k)$ as formal symbols subject to the
identifications
\begin{eqnarray*}
(I_1,I_2,\ldots ,I_k)=(I_2,I_3,\ldots
  ,I_k,I_1)=\ldots=(I_k,I_1,\ldots,I_{k-2},I_{k-1}).
\end{eqnarray*}
Each such symbol clearly specifies an equivalence class, and we pick a
unique representative of the corresponding equivalence class by
assuming, after relabelling of suffices, that $n\in I_k$.

The cardinality of the set $\cyc_{n,k}$ is $\lvert
\cyc_{n,k}\rvert=(k-1)!\cdot S_{n,k}$, where $S_{n,k}$ is the Stirling
number of the second kind. Note that $\lvert \cyc_{n,1}\rvert=1$ and
$\lvert \cyc_{n,n}\rvert=(n-1)!$. When the value of $n$ is understood
from the context, it will be convenient to denote
$\lvert\cyc_{n,k}\rvert$ by $r_{k-1}$ (see
Definition~\ref{cyccomplex}).

We take the opportunity to introduce a convenient enumeration of
$\cyc_{n,k}$, which will be used subsequently in the following
sections.

\begin{remark}\label{srle}
Given a subset $I$ of $[n]$ of cardinality $\lvert I\rvert$, let
$\chi_I$ denote the incidence vector of $I$, so that
$\chi_I^{\top}=(\chi_{I}(1),\ldots ,\chi_{I}(n))\in \{0,1\}^n\subset
\mbz^n$, with $\chi_I(i)=1$, if $i\in I$, and $\chi_I(i)=0$,
otherwise.  If $J$ and $I\subseteq [n]$ are two different subsets of
$[n]$, we will use the notation $J\prec I$ and say that ``$J$ precedes
$I$'' (or alternatively, ``$I$ succeeds $J$''), whenever $\lvert
J\rvert >\lvert I\rvert$, or whenever $\lvert J\rvert=\lvert I\rvert$
and the difference $\chi_J^{\top}-\chi_I^{\top}$ is $(*,\ldots
,*,1,0\ldots ,0)$; that is, the rightmost non-zero coefficient of
$\chi_J^{\top}-\chi_I^{\top}$ is $1$.  Given $(J_1,\ldots ,J_k)$ and
$(I_1,\ldots ,I_k)$, two different cyclically ordered partitions of
$[n]$ into $k$ blocks, with $n\in J_k,I_k$, we will use the notation
$(J_1,\ldots ,J_k)\prec (I_1,\ldots ,I_k)$ and say that ``$(J_1,\ldots
,J_k)$ precedes $(I_1,\ldots ,I_k)$'' (or alternatively,
``$(I_1,\ldots ,I_k)$ succeeds $(J_1,\ldots ,J_k)$''), whenever, for
some $s\in\{1,\ldots ,k-1\}$, $J_1=I_1,\ldots ,J_{s-1}=I_{s-1}$ and
$J_{s}$ precedes $I_s$. This is a ``size-reverse lexicographic
enumeration'', so we will refer to it as an $\srle$ (for a discussion
on this subject see, e.g., \cite{lhs}). Any reference made below
without qualification to an enumeration will refer to the enumeration
employing the $\srle$. (In this regard, note the word of caution
given in Remark~\ref{nonsrle}.)
\end{remark}

\begin{example}\label{srlen=4}
Set $n=4$. For the sake of brevity, write $(123,4)$ for
$(\{1,2,3\},\{4\})$.  Understanding that, in the display, elements on
the left precede elements on the right and employing the $\srle$,
$\cyc_{4,k}$ is enumerated as follows:
\begin{eqnarray*}
\cyc_{4,2}&=&\{(123,4),(23,14),(13,24),(12,34),(3,124),(2,134),(1,234)\};\\ 
\cyc_{4,3}&=&\{(23,1,4),(13,2,4),(12,3,4),(3,12,4),(3,2,14),(3,1,24),\\ 
&\mbox{ }&(2,13,4),(2,3,14),(2,1,34),(1,23,4),(1,3,24),(1,2,34)\};\\ 
\cyc_{4,4}&=&\{(3,2,1,4),(3,1,2,4),(2,3,1,4),(2,1,3,4),(1,3,2,4),
(1,2,3,4)\}.
\end{eqnarray*}
\end{example}

\subsection{The Cyc sequence associated to a digraph or to a CB matrix}
$\phantom{+}$\bigskip

\noindent Let $\mbg$ be a digraph of $n$ vertices or, equivalently,
consider an $n\times n$ $\cb$ matrix $L$ (recall
Assumption~\ref{digraph}). For disjoint non-empty subsets $I$ and $J$
of $[n]$, we define
\begin{eqnarray*}
x^{I\to J}=\prod_{i\in I}x_{i}^{\sum_{j\in J}a_{i,j}}.
\end{eqnarray*}
If $I$ or $J$ is the empty set, we put $x^{I\to J}$ equal to $1$.

\begin{definition}\label{cyccomplex}
Let $\mcc_{\mathbb{G}}$ (also denoted by $\mcc_L$, $\cyc$ or just
$\mcc$) be the sequence of homomorphisms of $\mba$-modules defined as
follows:
\begin{eqnarray*}
\mcc &:& 0\leftarrow \mcc_0\xleftarrow{\sdif_1} \mcc_1
\xleftarrow{\sdif_2} \cdots \xleftarrow{\sdif_{n-2}} \mcc_{n-2}
\xleftarrow{\sdif_{n-1}}\mcc_{n-1}\xleftarrow{} 0,
\end{eqnarray*}
where $\mcc_k=\mbk[x]^{{\rm Cyc}_{n,k+1}}$ is the free
$\mbk[x]$-module of rank $r_k:=\lvert\cyc_{n,k+1}\rvert$ and with
basis $\mcb_{k}$, the set of elements of $\cyc_{n,k+1}$. The elements
of $\mcb_k$ are enumerated employing the $\srle$. With this
enumeration, they will be denoted by $\mcb_{k}=\{e_{k,1},\ldots
,e_{k,r_k}\}$. Their images under $\dif_k$ will be denoted by
$\dif_{k}(\mcb_{k})=\{f_{k-1,1},\ldots ,f_{k-1,r_k}\}$, with
$f_{k-1,j}:=\dif_k(e_{k,j})$. The module $\mcc_0$ is $\mbk[x]^{{\rm
    Cyc}_{n,1}}=\mbk[x]$. Here the trivial $1$-block $([n])$ will be
identified with the unit element of $\mbk[x](=\mba)$.  The boundary
map $\dif_{k}:\mcc_k\to\mcc_{k-1}$, in slightly simplified notation
that we use throughout, is given by the formula
\begin{eqnarray}\label{differential}
\dif_{k}(I_1,\ldots,I_{k+1}) & = & \sum_{s=1}^{k}(-1)^{s-1}x^{I_s\to
  I_{s+1}}(I_1,\ldots ,I_s\cup I_{s+1},\ldots ,I_{k+1}) \\ &&
-x^{I_{k+1}\to I_1}(I_2,\ldots,I_{k},I_1\cup I_{k+1}). \notag
\end{eqnarray}
\end{definition}

\begin{remark}\label{dif-and-enum}
Observe that the $k$ terms of the first addition in
\eqref{differential} are enumerated according to the $\srle$; this is
not the case for the last term, which depends on the relationship
between $I_1$ and $I_2$. However, provided that, for all $i=1,\ldots
,n-1$, $a_{n,i}> 0$, the last term is distinguished from the others
because it is the only one whose coefficient contains the variable
$x_n$. This follows from the fact that $n\in I_{k+1}$.
\end{remark}

\subsection{The degree zero component of Cyc}
$\phantom{+}$\bigskip

\noindent Let $\mbg$ be a digraph of $n$ vertices or, equivalently,
consider an $n\times n$ $\cb$ matrix $L$. Let $I(L)$ be the $\cb$
ideal associated to $L$ and let $I(\mcl)$ be the lattice ideal
associated to the lattice $\mcl$ spanned by the columns of $L$ (see
Section~\ref{general}).

\begin{remark}\label{Iind1C}
Let $(I,\overline{I})\in \mcc_1$, with $I\subseteq [n-1]$,
$I\neq\varnothing$ and $\overline{I}=[n]\setminus I$. Let $\chi_I$ be
the incidence vector of $I$. Then
\begin{eqnarray}\label{d1(I,J)}
\phantom{++} x^{I\to\overline{I}}=x^{(L\chi_I)^+}\mbox{, }\;\; x^{\overline{I}\to
  I}=x^{(L\chi_I)^-}\mbox{ and }\;\;\dif_1(I,\overline{I})=
x^{(L\chi_I)^+}-x^{(L\chi_I)^-}.
\end{eqnarray}
If follows that $I(L)\subseteq \dif_1(\mcc_1)\subseteq I(\mcl)$ is a
chain of inclusions of three ideals of $\mbk[x]=\mba$. Furthermore, if
$L$ is an $\icb$ matrix, then $I(L)$, $\dif_1(\mcc_1)$ and $I(\mcl)$ are
homogeneous ideals.
\end{remark}
\begin{proof}
If $i\in I$, the $i$-th coordinate of $L\chi_I=\sum_{j\in I}l_{*,j}$
is $(L\chi_I)_i=a_{i,i}-\sum_{j\in
  I\setminus\{i\}}a_{i,j}=\sum_{j\in\overline{I}}a_{i,j}$, which is
positive. Therefore, in this case, $(L\chi_I)_i=((L\chi_I)^+)_i$.  If
$i\not\in I$, the $i$-th coordinate of $L\chi_I$ is
$(L\chi_I)_i=-\sum_{j\in I}a_{i,j}$, which is non-positive. Therefore,
in this case, $((L\chi_I)^+)_i=0$. Hence
\begin{eqnarray*}
x^{(L\chi_I)^{+}}=\prod_{i\in I}x_i^{\sum_{j\in\overline{I}}a_{i,j}},
\end{eqnarray*}
which coincides with the
definition of $x^{I\to\overline{I}}$. Clearly
$\chi_I+\chi_{\overline{I}}=\mbone^{\top}$. Therefore,
$L\chi_{\overline{I}}=-L\chi_I$ and
$(L\chi_{\overline{I}})^+=(-L\chi_I)^+=(L\chi_I)^-$. Thus
$x^{\overline{I}\to I}=x^{(L\chi_{\overline{I}})^+}=x^{(L\chi_I)^-}$.

We know that $\dif_1(I,\overline{I})=x^{I\to \overline{I}}(I\cup
\overline{I})-x^{\overline{I}\to I}(I\cup \overline{I})$. As we
mentioned before, we make the identification $(I\cup
\overline{I})\equiv ([n])$ with the unit element in $\mbk[x]=\mcc_0$. So
\begin{eqnarray*}
\dif_1(I,\overline{I})=x^{I\to \overline{I}}-x^{\overline{I}\to
  I}=x^{(L\chi_I)^+}-x^{(L\chi_I)^-}, 
\end{eqnarray*}
which proves \eqref{d1(I,J)}. Since $L\chi_I\in\mcl$, then
$x^{(L\chi_I)^+}-x^{(L\chi_I)^-}$ is in $I(\mcl)$. This proves
$\dif_1(\mcc_1)\subseteq I(\mcl)$.

On the other hand, for $i=1,\ldots, n-1$, the $i$-th column $l_{*,i}$
of $L$ can be expressed as $l_{*,i}=L\chi_{\{i\}}$.  So, by
\eqref{d1(I,J)},
$x^{l_{*,i}^+}-x^{l_{*,i}^-}=\dif_1(\{i\},\overline{\{i\}})\in
\dif_1(\mcc_1)$.  Recalling that our partitions $(I,\overline{I})$ of
$[n]$ have the property $n\in \overline{I}$, then $l_{*,n}$ can be
expressed as $l_{*,n}=-(l_{*,1}+\cdots +l_{*,n-1})=-L\chi_{\{1,\ldots
  ,n-1\}}$. Taking into account that, when $\beta=-\alpha\in\mbz^n$,
then $x^{\beta^+}-x^{\beta^-}=-(x^{\alpha^+}-x^{\alpha^-})$, and using
\eqref{d1(I,J)} again,
\begin{eqnarray*}
x^{l_{*,n}^+}-x^{l_{*,n}^-} & = &
  -(x^{(L\chi_{\{1,\ldots,n-1\}})^+}-x^{(L\chi_{\{1,\ldots
      ,n-1\}})^-})\\ & = & -\dif_1(\{1,\ldots ,n-1\},\{n\})\in
  \dif_1(\mcc_1).
\end{eqnarray*}
Therefore, $I(L)=(x^{l_{*,1}^+}-x^{l_{*,1}^+},\ldots
,x^{l_{*,n}^+}-x^{l_{*,n}^+})\subseteq \dif_1(\mcc_1)$.

Now, suppose that $L$ is an $\icb$ matrix. Since $L\chi_I\in\mcl$, it
follows, as in the proof of Remark~\ref{ILhomogeneous}, that
$x^{(L\chi_I)^+}-x^{(L\chi_I)^-}$ is homogeneous (in the grading
considered in Assumption~\ref{grading}). This proves that
$\dif_1(\mcc_1)$ is homogeneous and $\dif_1(\mcc_1)\subseteq
I(\mcl)$. By Remark~\ref{ILhomogeneous}, we already know that $I(L)$
and $I(\mcl)$ are homogeneous.
\end{proof}

\begin{notation}\label{fC}
For ease of notation, if $(C,\overline{C})\in \mcc_1$ (with
$C\subseteq [n-1]$, $C\neq\varnothing$ and $\overline{C}=[n]\setminus
C$) is a partition of $[n]$, we set $m_C:=x^{C\to\overline{C}}$,
$m_{\overline C}:=x^{\overline{C}\to C}$, and
$f_C:=m_C-m_{\overline{C}}=\dif_1(C,\overline{C})$.
\end{notation}

\begin{definition}\label{dif0}
 We define $\dif_0:\mcc_0=\mbk[x]\to \mbk[x]/I(\mcl)$ to be the
 natural projection onto the quotient ring, so that $\dif_0\circ
 \dif_1=0$.
\end{definition}

\subsection{The Cyc sequence in four variables}
$\phantom{+}$\bigskip

\noindent Let us display the sequence of homomorphims $\mcc$ of
Definition~\ref{cyccomplex} for a digraph $\mbg$ of four vertices or,
equivalently, for a $4\times 4$ $\cb$ matrix $L$.

\begin{example}\label{cyc4}
Let $n=4$. To simplify notations, write $x,y,z,t$ instead of
$x_1,x_2,x_3,x_4$, respectively. As in Example~\ref{srlen=4}, write
\begin{eqnarray*}
(i_1i_2\ldots i_r,j_1j_2\ldots j_s)\mbox{ instead of }(\{i_1,i_2,\ldots
,i_r\},\{j_1,j_2,\ldots ,j_s\}).
\end{eqnarray*}
The resulting complex $\mcc$ has the following form.
\begin{eqnarray*}
\mcc&:&0\leftarrow
\mcc_0=\mbk[x]\xleftarrow{\sdif_{1}}\mcc_{1}=\mbk[x]^{7}
\xleftarrow{\sdif_{2}}\mcc_{2}=\mbk[x]^{12}\xleftarrow{\sdif_{3}}
\mcc_{3}=\mbk[x]^{6}\xleftarrow{} 0.
\end{eqnarray*}
The differentials go as follows. Recall that the elements of the basis
$\mcb_k$ are enumerated employing the $\srle$ (see
Definition~\ref{cyccomplex} and Example~\ref{srlen=4}; the underlined
terms will be the leading ones in an ordering specified subsequently
in Assumption~\ref{wrlo}). For $\dif_1$:
\begin{eqnarray*}
\begin{array}{l}
f_{0,1}=\dif_1(e_{1,1})=\dif_1(123,4)=
\underline{x^{a_{1,4}}y^{a_{2,4}}z^{a_{3,4}}}-t^{a_{4,4}},\\
f_{0,2}=\dif_1(e_{1,2})=\dif_1(23,14)=
\underline{y^{a_{2,1}+a_{2,4}}z^{a_{3,1}+a_{3,4}}}-
x^{a_{1,2}+a_{1,3}}t^{a_{4,2}+a_{4,3}},\\
f_{0,3}=\dif_1(e_{1,3})=\dif_1(13,24)=
\underline{x^{a_{1,2}+a_{1,4}}z^{a_{3,2}+a_{3,4}}}-
y^{a_{2,1}+a_{2,3}}t^{a_{4,1}+a_{4,3}},\\
f_{0,4}=\dif_1(e_{1,4})=\dif_1(12,34)=
\underline{x^{a_{1,3}+a_{1,4}}y^{a_{2,3}+a_{2,4}}}-
z^{a_{3,1}+a_{3,2}}t^{a_{4,1}+a_{4,2}},\\
f_{0,5}=\dif_1(e_{1,5})=\dif_1(3,124)=
\underline{z^{a_{3,3}}}-x^{a_{1,3}}y^{a_{2,3}}t^{a_{4,3}},\\
f_{0,6}=\dif_1(e_{1,6})=\dif_1(2,134)=
\underline{y^{a_{2,2}}}-x^{a_{1,2}}z^{a_{3,2}}t^{a_{4,2}},\\
f_{0,7}=\dif_1(e_{1,7})=\dif_1(1,234)=
\underline{x^{a_{1,1}}}-y^{a_{2,1}}z^{a_{3,1}}t^{a_{4,1}}.
\end{array}
\end{eqnarray*}
For $\dif_2$ we have:
\begin{eqnarray*}
\begin{array}{lcl}
f_{1,1} & = & \dif_2(e_{2,1})=\dif_2(23,1,4) 
\\ & = & y^{a_{2,1}}z^{a_{3,1}}(123,4)-\underline{x^{a_{1,4}}(23,14)}-
t^{a_{4,2}+a_{4,3}}(1,234),\\
f_{1,2} & = & \dif_2(e_{2,2})=\dif_2(13,2,4)
\\ & = & x^{a_{1,2}}z^{a_{3,2}}(123,4)-\underline{y^{a_{2,4}}(13,24)}-
t^{a_{4,1}+a_{4,3}}(2,134),\\
f_{1,3} & = & \dif_2(e_{2,3})=\dif_2(12,3,4) 
\\ & = &  x^{a_{1,3}}y^{a_{2,3}}(123,4)-\underline{z^{a_{3,4}}(12,34)}-
t^{a_{4,1}+a_{4,2}}(3,124),\\
f_{1,4} & = & \dif_2(e_{2,4})=\dif_2(3,12,4) 
\\ & = & z^{a_{3,1}+a_{3,2}}(123,4)-\underline{x^{a_{1,4}}y^{a_{2,4}}(3,124)}-
t^{a_{4,3}}(12,34),\\
f_{1,5} & = & \dif_2(e_{2,5})=\dif_2(3,2,14) 
\\ & = & z^{a_{3,2}}(23,14)-\underline{y^{a_{2,1}+a_{2,4}}(3,124)}-
x^{a_{1,3}}t^{a_{4,3}}(2,134),\\
f_{1,6} & = & \dif_2(e_{2,6})=\dif_2(3,1,24) 
\\ & = &  z^{a_{3,1}}(13,24)-\underline{x^{a_{1,2}+a_{1,4}}(3,124)}-
y^{a_{2,3}}t^{a_{4,3}}(1,234),\\
f_{1,7} & = & \dif_2(e_{2,7})=\dif_2(2,13,4)
\\ & = & y^{a_{2,1}+a_{2,3}}(123,4)-\underline{x^{a_{1,4}}z^{a_{3,4}}(2,134)}-
t^{a_{4,2}}(13,24),\\
f_{1,8} & = & \dif_2(e_{2,8})=\dif_2(2,3,14)
\\ & = & y^{a_{2,3}}(23,14)-\underline{z^{a_{3,1}+a_{3,4}}(2,134)}-
x^{a_{1,2}}t^{a_{4,2}}(3,124),\\
f_{1,9} & = & \dif_2(e_{2,9})=\dif_2(2,1,34)
\\ & = & y^{a_{2,1}}(12,34)-\underline{x^{a_{1,3}+a_{1,4}}(2,134)}-
z^{a_{3,2}}t^{a_{4,2}}(1,234),\\
f_{1,10} & = & \dif_2(e_{2,10})=\dif_2(1,23,4)
\\ & = & x^{a_{1,2}+a_{1,3}}(123,4)-\underline{y^{a_{2,4}}z^{a_{3,4}}(1,234)}-
t^{a_{4,1}}(23,14),\\
f_{1,11} & = & \dif_2(e_{2,11})=\dif_2(1,3,24)
\\ & = &  x^{a_{1,3}}(13,24)-\underline{z^{a_{3,2}+a_{3,4}}(1,234)}-
y^{a_{2,1}}t^{a_{4,1}}(3,124),\\
f_{1,12} & = & \dif_2(e_{2,12})=\dif_2(1,2,34)
\\ & = & x^{a_{1,2}}(12,34)-\underline{y^{a_{2,3}+a_{2,4}}(1,234)}-
z^{a_{3,1}}t^{a_{4,1}}(2,134).
\end{array}
\end{eqnarray*}
Finally, for $\dif_3$:
\begin{eqnarray*}
\begin{array}{lcl}
f_{2,1} & = & \dif_3(e_{3,1})=
\dif_3(3,2,1,4)\\ & = & z^{a_{3,2}}(23,1,4)-y^{a_{2,1}}(3,12,4)+
\underline{x^{a_{1,4}}(3,2,14)}-t^{a_{4,3}}(2,1,34),\\
f_{2,2} & = & \dif_3(e_{3,2})=
\dif_3(3,1,2,4)\\ & = & z^{a_{3,1}}(13,2,4)-x^{a_{1,2}}(3,12,4)+
\underline{y^{a_{2,4}}(3,1,24)}-t^{a_{4,3}}(1,2,34),\\
f_{2,3} & = &  \dif_3(e_{3,3})=
\dif_3(2,3,1,4)\\ & = & y^{a_{2,3}}(23,1,4)-z^{a_{3,1}}(2,13,4)+
\underline{x^{a_{1,4}}(2,3,14)}-t^{a_{4,2}}(3,1,24),
\end{array}
\end{eqnarray*}
\begin{eqnarray*}
\begin{array}{ll}
f_{2,4}= & \dif_3(e_{3,4})=
\dif_3(2,1,3,4)=\\ & y^{a_{2,1}}(12,3,4)-x^{a_{1,3}}(2,13,4)
+\underline{z^{a_{3,4}}(2,1,34)}-t^{a_{4,2}}(1,3,24),\\
f_{2,5}= &\dif_3(e_{3,5})=
\dif_3(1,3,2,4)=\\ & x^{a_{1,3}}(13,2,4)-z^{a_{3,2}}(1,23,4)+
\underline{y^{a_{2,4}}(1,3,24)}-t^{a_{4,1}}(3,2,14),\\
f_{2,6}= & \dif_3(e_{3,6})= 
\dif_3(1,2,3,4)=\\ & x^{a_{1,2}}(12,3,4)-y^{a_{2,3}}(1,23,4)+
\underline{z^{a_{3,4}}(1,2,34)}-t^{a_{4,1}}(2,3,14).
\end{array}
\end{eqnarray*}
\end{example}

\subsection{Cyc is a complex for any CB matrix}\label{cyc-complex-cb}
$\phantom{+}$\bigskip

\noindent Let $L$ be a $\cb$ matrix or, equivalently, let $\mbg$ be a
digraph. Let $\mcc_{\mathbb{G}}$ be the sequence of homomorphisms as
in Definition~\ref{cyccomplex}.

\begin{proposition}\label{chain-complex}
The sequence $\mcc_{\mathbb{G}}$ is a chain complex of free
$\mbk[x]$-modules.
\end{proposition}
\begin{proof}
Fix a basis element $(I_1,I_2,...,I_{k+1})\in \mcb_k$ of $\mcc_k$,
where $2\leq k\leq n-1$.  We need to show that
$\dif_{k-1}(\dif_{k}(I_{1} ,I_{2},...,I_{k+1}))=0$.  Clearly,
$\dif_{k-1}(\dif_{k}(I_1,\ldots,I_{k+1}))$ is expressed as a linear
combination of $k(k+1)$ elements of $\mcc_{k-2}$. Indeed:
\begin{eqnarray}\label{dd}
&& \dif_{k-1}(\dif_{k}(I_1,\ldots,I_{k+1}))=\\ &&
  \sum_{s=1}^{k}(-1)^{s-1}x^{I_s\to I_{s+1}}\dif_{k-1}(I_1,\ldots
  ,I_s\cup I_{s+1},\ldots ,I_{k+1})\notag \\ && -x^{I_{k+1}\to
    I_1}\dif_{k-1}(I_2,\ldots,I_{k},I_1\cup I_{k+1}). \notag
\end{eqnarray}
For $s=1$, the expression for $\dif_{k-1}(I_1\cup
I_2,\ldots ,I_{k+1})$ is equal to:
\begin{eqnarray}\label{s=1}
&& x^{I_1\cup I_2\to I_3}(I_1\cup I_2\cup I_3,\ldots ,I_{k+1})\\ && 
 +\sum_{t=3}^{k}(-1)^{t-2}x^{I_t\to I_{t+1}}(I_1\cup I_2,\ldots
  ,I_t\cup I_{t+1},\ldots ,I_{k+1}) \notag \\ && -x^{I_{k+1}\to
    I_{1}\cup I_2}(I_3,\ldots ,I_{k},I_1\cup I_2\cup I_{k+1}).  \notag
\end{eqnarray}
For $s=2$, the expression for $\dif_{k-1}(I_1, I_2\cup I_3,\ldots
,I_{k+1})$ is:
\begin{eqnarray}\label{s=2}
&& x^{I_1\to I_2\cup I_3}(I_1\cup I_2\cup I_3,\ldots ,I_{k+1})\\&&
  -x^{I_2\cup I_3\to I_4}(I_1,I_2\cup I_3\cup I_4,\ldots,I_{k+1}) 
  \notag \\ && +\sum_{t=4}^{k}(-1)^{t-2}x^{I_t\to I_{t+1}}(I_1,I_2\cup
  I_3,\ldots ,I_t\cup I_{t+1},\ldots ,I_{k+1}) \notag \\ &&
  -x^{I_{k+1}\to I_{1}}(I_2\cup I_3,\ldots ,I_{k},I_1\cup
  I_{k+1}). \notag
\end{eqnarray}
For $s\in\{3,\ldots ,k-2\}$, the expression for $\dif_{k-1}(I_1,\ldots
,I_s\cup I_{s+1},\ldots ,I_{k+1})$ is:
\begin{eqnarray}\label{s=3tok-2}
&&\sum_{t=1}^{s-2}(-1)^{t-1}x^{I_t\to I_{t+1}}(I_1,\ldots ,I_t\cup
  I_{t+1},\ldots , I_s\cup I_{s+1},\ldots ,I_{k+1})
  \\ &&+ \sum_{t=s+2}^{k}(-1)^{t-2}x^{I_t\to I_{t+1}}(I_1,\ldots
  ,I_s\cup I_{s+1},\ldots , I_t\cup I_{t+1},\ldots ,I_{k+1})\notag
  \\ &&+ (-1)^{s-2}x^{I_{s-1}\to I_{s}\cup I_{s+1}}(I_1,\ldots
  ,I_{s-1}\cup I_{s}\cup I_{s+1},\ldots, I_{k+1}) \notag
  \\ &&+ (-1)^{s-1}x^{I_{s}\cup I_{s+1}\to I_{s+2}}(I_1,\ldots
  ,I_{s}\cup I_{s+1}\cup I_{s+2},\ldots, I_{k+1}) \notag \\ &&
  -x^{I_{k+1}\to I_{1}}(I_2,\ldots ,I_s\cup I_{s+1},\ldots ,I_1\cup
  I_{k+1}).  \notag
\end{eqnarray}
For $s=k-1$, the expression for $\dif_{k-1}(I_1,\ldots ,I_{k-1}\cup
I_k,I_{k+1})$ is equal to:
\begin{eqnarray}\label{s=k-1}
&&\sum_{t=1}^{k-3}(-1)^{t-1}x^{I_t\to I_{t+1}}(I_1,\ldots ,I_t\cup
  I_{t+1},\ldots , I_{k-1}\cup I_k,I_{k+1})\\ &&
  +(-1)^{k-3}x^{I_{k-2}\to I_{k-1}\cup I_k}(I_1,\ldots ,I_{k-2}\cup
  I_{k-1}\cup I_k,I_{k+1}) \notag \\ && +(-1)^{k-2}x^{I_{k-1}\cup
    I_{k}\to I_{k+1}}(I_1,\ldots , I_{k-1}\cup I_k\cup I_{k+1}) \notag
  \\ && -x^{I_{k+1}\to I_{1}}(I_2,\ldots ,I_{k-1}\cup I_{k}, I_1\cup
  I_{k+1}). \notag
\end{eqnarray}
For $s=k$, we have $\dif_{k-1}(I_1,\ldots ,I_k\cup I_{k+1})$ equal to
\begin{eqnarray}\label{s=k}
&& \sum_{t=1}^{k-2}(-1)^{t-1}x^{I_t\to I_{t+1}}(I_1,\ldots ,I_t\cup
  I_{t+1},\ldots , I_k\cup I_{k+1})\\ && +(-1)^{k-2}x^{I_{k-1}\to
    I_{k}\cup I_{k+1}}(I_1,\ldots ,I_{k-1}\cup I_{k}\cup I_{k+1})
  \notag \\&& -x^{I_{k}\cup I_{k+1}\to I_{1}}(I_2,\ldots ,I_1\cup
  I_k\cup I_{k+1}). \notag
\end{eqnarray}
As for the final term, $\dif_{k-1}(I_2,\ldots,I_{k},I_1\cup
I_{k+1})$ is equal to:
\begin{eqnarray}\label{s=k+1}
&&\sum_{t=2}^{k-1}(-1)^{t-2}x^{I_t\to I_{t+1}}(I_2,\ldots ,I_t\cup
I_{t+1},\ldots ,I_1\cup I_{k+1})\\ 
&&+ (-1)^{k-2}x^{I_{k}\to I_{1}\cup I_{k+1}}(I_2,\ldots ,I_{1}\cup 
I_{k}\cup I_{k+1}) \notag \\ && - x^{I_{1}\cup
I_{k+1}\to I_{2}}(I_3,\ldots ,I_1\cup I_2\cup I_{k+1}). \notag
\end{eqnarray}
The basis elements of $\mcb_{k-2}$ that appear in these linear
combinations, from \eqref{s=1} to \eqref{s=k+1}, solely involve
partitions obtained from $(I_1,\ldots,I_{k+1})$ by coalescing twice
two adjacent blocks. Any such basis element appears twice, obtained by
doing two different mergings done in two different orders. For
instance, the element $(I_1\cup I_2\cup I_3,\ldots ,I_{k+1})$ appears
in the first summand in \eqref{s=1} and in \eqref{s=2}. In
\eqref{s=1}, one has first merged $I_1$ with $I_2$, obtaining
$x^{I_1\to I_2}(I_1\cup I_2,\ldots ,I_{k+1})$ and then, subsequently,
one has merged $I_1\cup I_2$ with $I_3$, giving $x^{I_1\to
  I_2}x^{I_1\cup I_2\to I_3}(I_1\cup I_2\cup I_3,\ldots ,I_{k+1})$. On
the other hand, the first summand in \eqref{s=2} is obtained by first
merging $I_2$ with $I_3$, getting $(-1)x^{I_2\to I_3}(I_1,I_2\cup
I_3,\ldots ,I_{k+1})$, and afterwards, on merging $I_1$ with $I_2\cup
I_3$, one obtains $(-1)x^{I_2\to I_3}x^{I_1\to I_2\cup I_3}(I_1\cup
I_2\cup I_3,\ldots ,I_{k+1})$. Observe that both monomial coefficients
are equal, but opposite in sign, so they cancel each with the other.

We list now the complete set of partitions appearing as basis elements
in the expression for
$\dif_{k-1}(\dif_{k}(I_1,\ldots,I_{k+1}))$. Concretely, and enumerated
accordingly to the $\srle$, the first $k-1$ are:
\begin{eqnarray*}
(I_1\cup I_2\cup I_3,\ldots ,I_{k+1}),(I_1\cup I_2,I_3\cup I_4,\ldots
  ,I_{k+1}),\ldots ,(I_1\cup I_2,\ldots , I_k\cup I_{k+1}).
\end{eqnarray*}
These elements appear in \eqref{s=1} and in \eqref{s=2}. 

Without mentioning where they come from, we have the following, still
enumerated with the $\srle$.

Next we have a set of $k-2$, all of them beginning with $(I_1,I_2\cup
I_3\Box)$, thus succeeding the $k-1$ above:
\begin{eqnarray*}
(I_1,I_2\cup I_3\cup I_4,\ldots ,I_{k+1}), (I_1,I_2\cup I_3, I_4\cup
  I_5,\ldots ,I_{k+1}),\ldots ,(I_1,I_2\cup I_3,\ldots ,I_k\cup I_{k+1}).
\end{eqnarray*}
Then we have the set of $k-3$, beginning with $(I_1,I_2,I_3\cup
I_4\Box)$, and so on. In the second last place, we have the set of $2$
beginning with $(I_1,\ldots ,I_{k-2}\cup I_{k-1}\Box)$, that is,
$(I_1,\ldots ,I_{k-2}\cup I_{k-1}\cup I_{k},I_{k+1})$ and $(I_1,\ldots
,I_{k-2},I_{k-1}\cup I_{k}\cup I_{k+1})$, and finally we have $(I_1,\ldots
I_{k-2}, I_{k-1}\cup I_{k}\cup I_{k+1})$. Altogether then, these
amount in total to $(k-2)+(k-3)+\cdots +2+1$ partitions.

In the expression for $\dif_{k-1}(\dif_{k}(I_1,\ldots,I_{k+1}))$,
there are still $k$ further basis elements of $\mcb_{k-2}$
involved. Concretely, these are the final terms in \eqref{s=1},
\eqref{s=2}, \eqref{s=3tok-2}, \eqref{s=k-1} and \eqref{s=k}, and all
the terms in \eqref{s=k+1}. It is not possible to enumerate them
together with the former ones (employing the $\srle$) unless more
information is given. Each involves the leftmost block $I_1$ being
merged with the rightmost block $I_{k+1}$. These basis elements are
also characterized by the fact that the corresponding coefficient term
contains the variable $x_n$ (possibly with exponent zero). Enumerated
among themselves (in the $\srle$), the first $k-1$ are as follows:
\begin{multline*}
(I_2\cup I_3,\ldots ,I_1\cup I_{k+1}),(I_2,I_3\cup I_4,\ldots ,I_1\cup
  I_{k+1}),\ldots ,\\(I_2,\ldots ,I_{k-1}\cup I_k,I_1\cup I_{k+1}),
  (I_2,\ldots ,I_1\cup I_k\cup I_{k+1}).
\end{multline*}
Finally, we have the last term $(I_3,\ldots ,I_1\cup I_2\cup
I_{k+1})$.

Thus a total amount of distinct $(k-1)+\sum_{i=1}^{k-2}i+k=k(k+1)/2$
basis elements of $\mcb_{k-2}$ appear in
$\dif_{k-1}(\dif_{k}(I_1,\ldots,I_{k+1}))$. As mentioned before, each
of them appears twice. This is consistent with the total number of
$k(k+1)$ elements in the linear expression for
$\dif_{k-1}(\dif_{k}(I_1,\ldots,I_{k+1}))$.

To finish the proof, one can check that the two terms involving each
basis element have the same monomial coefficient, but with opposite
sign (alternatively, one can use \eqref{disjoint} below). This will
prove that $\dif_{k-1}(\dif_{k}(I_1,\ldots,I_{k+1}))=0$.
\begin{itemize}
\item Terms $(I_1,\ldots,I_s\cup I_{s+1}\cup I_{s+2},\ldots
  ,I_{k+1})$, with $1\leq s\leq k-1$. The two monomial coefficients
  involving this basis element are:
\begin{eqnarray*}
&& (-1)^{s-1}x^{I_{s}\to I_{s+1}}(-1)^{s-1}x^{I_{s}\cup I_{s+1}\to
    I_{s+2}}\mbox{ and }\\ && (-1)^{s}x^{I_{s+1}\to
    I_{s+2}}(-1)^{s-1}x^{I_{s}\to I_{s+1}\cup I_{s+2}}.
\end{eqnarray*}
\item Terms $(I_1,\ldots,I_s\cup I_{s+1},\ldots ,I_{t}\cup
  I_{t+1},\ldots,I_{k+1})$, with $1\leq s$, $s+1<t$ and $t+1\leq k+1$.
  The two monomial coefficients involving this basis element are:
\begin{eqnarray*}
&& (-1)^{s-1}x^{I_{s}\to I_{s+1}}(-1)^{t-2}x^{I_{t}\to
    I_{t+1}}\mbox{ and }\\ && (-1)^{t-1}x^{I_{t}\to I_{t+1}}(-1)^{s-1}x^{I_{s}\to
    I_{s+1}}.
\end{eqnarray*}
\item Terms $(I_2,\ldots ,I_s\cup I_{s+1},\ldots ,I_1\cup I_{k+1})$,
  with $2\leq s\leq k-1$. Here, the two monomial coefficients are:
\begin{eqnarray*}
&& (-1)x^{I_{k+1}\to I_1}(-1)^{s-2}x^{I_s\to I_{s+1}}\mbox{ and }\\ &&
  (-1)^{s-1}x^{I_s\to I_{s+1}}(-1)x^{I_{k+1}\to I_{1}}.
\end{eqnarray*}
\item Terms $(I_2,\ldots ,I_1\cup I_{k}\cup I_{k+1})$. The two
  monomial coefficients are:
\begin{eqnarray*}
&& (-1)^{k-1}x^{I_k\to I_{k+1}}(-1)x^{I_{k}\cup I_{k+1}\to I_1}\mbox{ and }
\\ && (-1)x^{I_{k+1}\to
    I_1}(-1)^{k-2}x^{I_{k}\to I_1\cup I_{k+1}}. 
\end{eqnarray*}
\item Finally, the two monomial coefficients of $(I_3,\ldots ,I_1\cup
I_2\cup I_{k+1})$ are 
\begin{eqnarray*}
&& (-1)x^{I_{k+1}\to I_1}(-1)x^{I_{k+1}\cup I_1\to I_2}\mbox{ and }\\ && 
x^{I_1\to I_2}(-1)x^{I_{k+1}\to I_1\cup I_2}.
\end{eqnarray*}
\end{itemize}
It is easy to check that they are equal in pairs, but opposite in
sign. This finishes the proof.
\end{proof}

\begin{remark}\label{dif-m}
If $\mbg$ is strongly complete, or equivalently $L$ is a $\pcb$
matrix, then $\dif_k(\mcc_k)\subseteq \mfm\mcc_{k-1}$, since
$a_{i,j}>0$, for all $i,j=1,\ldots ,n$, ensures the containment
$x^{I_t\to I_{t+1}}\in (x_1,\ldots ,x_n)=\mfm$.
\end{remark}

\section{Gr\"obner bases for lattice ideals associated to 
strongly connected digraphs}

\subsection{Block echelon form of an ICB matrix}\label{block-icb}
$\phantom{+}$\bigskip

\noindent We start by defining what we will consider to be a block
echelon form of an $\icb$ matrix. In order to simplify the exposition,
the expression ``$I$-rows'' will stand for the set of rows whose row
subindex is in $I$, and analogously as regards the expression
``$J$-columns''. As in \eqref{reducedform}, given $I,J\subseteq [n]$,
let $M_{I,J}$ be the submatrix of $M$ defined by the $I$-rows and the
$J$-columns.

\begin{definition}\label{def-echelon}
Let $L$ be an $n\times n$ $\icb$ matrix and let $\delta$ be a positive
integer, $1\leq \delta\leq n-1$. Given $q_1,\ldots ,q_\delta\in
\mbn_+$, with $q_1+\cdots+q_\delta=n-1$, let $(I_1,\ldots ,I_\delta)$
be the partition of $[n-1]$ into $\delta$ blocks of cardinalities
$q_1,\ldots ,q_\delta$, defined by $I_1=[q_1]$, $I_1\cup
I_2=[q_1+q_2],\ldots ,I_1\cup\ldots\cup I_\delta=[n-1]$.  Set
$I_{\delta+1}=\{n\}$. Then $L$ is said to be in $\delta$-block echelon
form if $L_{I_{i},I_{j}}=0$, for every $i$ and $j$ such that $1\leq
j\leq \delta-1$ and $j+2\leq i\leq \delta+1$, and each column of
$L_{I_{j+1},I_j}$ is nonzero, for every $j=1,\ldots ,\delta$.
\end{definition}

\begin{example}
Let $L$ be the $\icb$ matrix below which is in $\delta$-block echelon
form, where $\delta=3$, and $q_1=1$, $q_2=2$ and $q_3=2$. Then
$(I_1,I_2,I_3)=(\{1\},\{2,3\},\{4,5\})$. Note that the $4$ blocks on
the diagonal have diagonal positive entries. The $3$ blocks below the
diagonal blocks have nonzero columns. All the remaining blocks below
the aforementioned are zero.
\begin{eqnarray*}
\newcommand\y{\cellcolor{gray!40}}
L=\left(\begin{array}{rrrrrr}
1&0&0&0&0&-1\\
\y0&1&-1&0&0&0\\
\y-1&0&1&0&0&0\\
0&\y0&\y-1&1&0&0\\
0&\y-1&\y0&0&1&0\\
0&0&0&\y-1&\y-1&2\\
\end{array}\right).
\end{eqnarray*}
\end{example}

\begin{remark}\label{echelon-reduction}
Any $n\times n$ $\icb$ matrix $L$ can be reduced by permutations to a
$\delta$-block echelon form.
\end{remark}
\begin{proof} (We remark that it may help to consider 
the shape of the matrix in the previous example.)

Indeed, set $I_{\delta+1}=\{n\}$. The notation $J_\delta$-columns,
with $J_{\delta}\subseteq [n-1]$, will denote the set of columns with
column subindex in $[n-1]$, having a nonzero entry in the last
row. Since $L$ is irreducible, then, necessarily, in the last row of
$L$ there are other nonzero entries apart from $a_{n,n}$ (this follows
too from the fact that since $L$ is, in particular, a $\cb$ matrix,
$0<a_{n,n}=\sum_{j\neq n}a_{n,j}$). Therefore $\lvert
J_\delta\rvert:=q_\delta\geq 1$. Move these $q_\delta$ columns to the
outer right of the first $n-1$ columns; we use the notation
$I_\delta$-columns to denote the corresponding set of columns. Then
perform the corresponding permutation of rows, so preserving the
structure of a $\cb$ matrix, while in the last row, which remains
untouched by the row permutations just carried out, we maintain the
property that all its $q_\delta+1$ nonzero entries are in the set of
$I_\delta\cup I_{\delta+1}$-columns.

Since $L$ is irreducible, then there must be at least one column to
the left of the set of $I_\delta\cup I_{\delta+1}$-columns having a
nonzero entry in the set of $I_\delta\cup I_{\delta+1}$-rows, these
being strictly below the diagonal. For, if not, setting
$I=[n]\setminus I_\delta\cup I_{\delta+1}$ and $J=I_\delta\cup
I_{\delta+1}$, one would have a partition $(I,J)$ of $[n]$ into two
non-empty disjoint sets and such that $L_{J,I}=0$. Thus, $L$ would be
reducible, a contradiction (recall the definition of irreducible in
Section~\ref{ibm}). So, denote by $J_{\delta-1}\subseteq [n]\setminus
I_{\delta}\cup I_{\delta+1}$ the set of column indices to the left of
the set of $I_\delta\cup I_{\delta+1}$-columns which have a nonzero
entry in the set of $I_\delta\cup I_{\delta+1}$-rows, these being
strictly below the diagonal, where $\lvert
J_{\delta-1}\rvert:=q_{\delta-1}\geq 1$. Now proceed recursively.
\end{proof}

\begin{remark}\label{rightmost}
Let $L$ be an $\icb$ matrix in $\delta$-block echelon form. Let
$C\subseteq [n-1]$, $C\neq \varnothing$. Then the rightmost nonzero
coefficient of $(L\chi_C)^{\top}$ is negative.
\end{remark}
\begin{proof}
Keeping the notations in Definition~\ref{def-echelon}, set $C_i=C\cap
I_i$ and let $m\geq 1$ be the maximum integer with $C_m\neq
\varnothing$. Clearly
$\chi_C=\sum_{i=1}^{\delta}\chi_{C_i}=\sum_{i=1}^m\chi_{C_i}$ and
$L\chi_C=\sum_{i=1}^mL\chi_{C_i}$. Since all the columns to the left
of $I_m$ have zero $I_{m+1}$-rows, it is enough to see that
$L\chi_{C_m}\leq 0$ and different from zero. This follows from the
fact that $L_{I_{m+1},I_m}\leq 0$ and that each column of
$L_{I_{m+1},I_m}$ is nonzero.
\end{proof}

To finish this section, let us prove that the $\delta$-block-echelon
form reflects the natural partition of the set of vertices into
subsets of elements whose distance from the last vertex is constant.
This brings out another aspect of the relationship between $\icb$
matrices $L$ and $\delta$-block echelon form, this time arising from
the natural metric $\ud$ on $\mbg_{L}$.

\begin{lemma}\label{blockthenomega}
Let $L$ be an $n\times n$ $\icb$ matrix and let $\mbg_L=(\mbv,\mbe)$
its associated digraph. Suppose that $L$ is in $\delta$-block echelon
form and let $(I_1,\ldots, I_\delta)$ be the associated partition of
$[n-1]$ into $\delta$ blocks, with $I_{\delta+1}=\{n\}$. For
$i=1,\ldots ,\delta+1$, set $W_i:=\{v_k\in\mbv\mid k\in I_i\}$, so
that $W_{\delta+1}=\{v_n\}$. Then $(W_1,\ldots ,W_{\delta})$ is a
partition of $\mbv\setminus\{v_n\}$ into $\delta$ non-empty subsets of
$\mbv$. Moreover, for each $i=1,\ldots, \delta$, the set $W_i$
coincides with the set $V_i:=\{v\in\mbv\mid \ud(v_n,v)=\delta+1-i\}$.
\end{lemma}
\begin{proof}
Clearly, since $(I_1,\ldots, I_\delta)$ is a partition of $[n-1]$ into
$\delta$ non-empty blocks, then $(W_1,\ldots ,W_{\delta})$ is a
partition of $\mbv\setminus\{v_n\}$ into $\delta$ non-empty subsets of
$\mbv$. Let us prove that $W_i=V_i$ by decreasing induction on
$i=1,\ldots ,\delta+1$. Clearly $W_{\delta+1}=\{v_n\}=V_{\delta+1}$.

Fix $i$, with $1\leq i\leq\delta$. Suppose that
$W_{i+1}=V_{i+1},\ldots ,W_{\delta+1}=V_{\delta+1}$ and let us prove
$W_i=V_i$.

Take $v_{k_i}\in W_i$, so that $k_i\in I_i$.  Since
$L_{I_{i+1},I_{i}}$ has non-zero columns, there exists $k_{i+1}\in
I_{i+1}$, with $a_{k_{i+1},k_i}>0$, so there is an arc $v_{k_{i+1}}\to
v_{k_i}$ in $\mbg$. Similarly, since $L_{I_{i+2},I_{i+1}}$ has
non-zero columns, there exists $k_{i+2}\in I_{i+2}$, with
$a_{k_{i+2},k_{i+1}}>0$, and an arc $v_{k_{i+2}}\to v_{k_{i+1}}$ in
$\mbg$. Recursively, there are vertices $v_{k_j}\in W_{j}$ and a
directed path $v_n\to v_{k_{\delta}}\to\cdots \to v_{k_{i+1}}\to
v_{k_i}$. In particular, $\ud(v_n,v_{k_i})\leq \delta+1-i$ and
$v_{k_i}\in V_{i}\cup\ldots \cup V_{\delta+1}$. On the other hand, since
$(W_1,\ldots ,W_{\delta+1})$ is a partition of $\mbv$ and $v_{k_i}\in
W_i$, then $v_{k_i}\not\in W_{i+1}\cup\ldots \cup
W_{\delta+1}=V_{i+1}\cup\ldots\cup V_{\delta+1}$. Thus $v_{k_i}\in
V_i$ and it follows that $W_i\subseteq V_i$.

Now, take $v_{k}\in V_i$. Then there exist $v_n\to
v_{k_\delta}\to\cdots \to v_{k_{i+1}}\to v_k$, a directed path of
minimum unweighted length $\delta+1-i$. In particular,
$\ud(v_n,v_{k_{i+1}})\leq \delta-i$ and, in fact, this is an equality,
otherwise $\ud(v_n,v_k)<\delta+1-i$. Hence $v_{k_{i+1}}\in
V_{i+1}=W_{i+1}$ and $k_{i+1}\in I_{i+1}$. The existence of the arc
$v_{k_{i+1}}\to v_k$ ensures $a_{k_{i+1},k}>0$. Since $L$ is in
$\delta$-block echelon form, the index $k$ must sit inside
$I_i\cup\ldots \cup I_{\delta +1}$. If $k\in I_{i+1}\cup\ldots \cup
I_{\delta+1}$, then $v_k\in W_{i+1}\cup\ldots \cup
W_{\delta+1}=V_{i+1}\cup\ldots\cup V_{\delta+1}$, a contradiction with
the fact that $V_i$ is clearly disjoint with $V_{i+1}\cup\ldots\cup
V_{\delta+1}$.  Thus $k\in I_{i}$ and $v_k\in V_i$ and we derive the
opposite inclusion $V_i\subseteq W_i$.
\end{proof}

\subsection{A convenient enumeration of the vertices} 
$\phantom{+}$\bigskip

\noindent Taking into acccount the former lemma, let us define a
special enumeration on the set of vertices and arcs of a strongly
connected digraph, and hence on the set of indeterminates $x_1,\ldots
,x_n$ (see \cite[p.~7]{crs}, \cite[Section~2.2.3]{mos}, for the case
where $\mbg$ is undirected, or, equivalently, $L_{\mathbb{G}}$ is
symmetric).

\begin{definition}\label{def-omega}
Let $\mbg=(\mbv,\mbe)$ be a strongly connected digraph.  Fix a
distinguished vertex $\omega$ in $\mbv$. Since $\mbg$ is finite,
$\delta:=\max\{\ud(\omega,v)\mid v\in \mbv\}$ is finite and
$\delta\geq 1$. Consider the partition $(V_1,\ldots ,V_{\delta})$ of
$\mbv\setminus\{\omega\}$ into $\delta$ non-empty subsets defined by
$V_i=\{v\in \mbv\mid \ud(\omega,v)=\delta+1-i\}$, for $i=1,\ldots
,\delta$. Set $V_{\delta+1}=\{\omega\}$.  Let $q_i=\lvert V_i\rvert$
be the cardinality of each $V_i$. Take the partition $(I_1,\ldots
,I_\delta)$ of $[n-1]$ into $\delta$ blocks of cardinalities
$q_1,\ldots ,q_\delta$, defined by $I_1=[q_1]$, $I_1\cup
I_2=[q_1+q_2],\ldots ,I_1\cup\ldots\cup I_\delta=[n-1]$. Set
$I_{\delta+1}=\{n\}$. The $(\omega,\delta)$-enumeration of the
vertices and arcs of $\mbg$ is determined by assigning the set of
indices $I_i$ to the set of vertices $V_i$.  That is, $V_i=\{v_k\in
\mbv\mid k\in I_i\}$, for $i=1,\ldots ,\delta$. The index $n$ is
assigned to the vertex $\omega$, so that $\omega$ becomes $v_n$. In
particular, if $i<j$, then $\ud(v_n,v_i)\geq \ud(v_n,v_j)$. On the
other hand, if $\ud(v_n,v_i)>\ud(v_n,v_j)$, then $i<j$. Clearly,
whenever $\ud(v_n,v_i)=\ud(v_n,v_j)$, the three possibilities $i<j$,
$i=j$ or $i>j$ may happen.
\end{definition}

In Lemma~\ref{blockthenomega} we have seen that if $L$ is an $n\times
n$ $\icb$ matrix in $\delta$-block echelon form, then the
corresponding digraph $\mbg_L$ has an $(\omega,\delta)$-enumeration
(where $\omega$ is taken to be the last vertex). Now let us prove that
a digraph $\mbg$ with an $(\omega,\delta)$-enumeration has the
corresponding $\icb$ matrix $L_{\mathbb{G}}$ in $\delta$-block echelon
form.

\begin{lemma}\label{omegathenblock}
Let $\mbg=(\mbv,\mbe)$ be a strongly connected digraph with an
$(\omega,\delta)$-enumera-tion. Let $L=L_{\mathbb{G}}$ be the $\icb$
matrix associated to $\mbg$. Then $L$ is in $\delta$-block echelon
form.
\end{lemma}
\begin{proof}
With the notation in Definition~\ref{def-omega}, 
\begin{eqnarray*}
\{v_k\in \mbv\mid
k\in I_i\}=\{v_k\in \mbv\mid \ud(\omega,v_k)=\delta+1-i\}=:V_i, 
\end{eqnarray*}
for $i=1,\ldots ,\delta$, and $I_{\delta+1}=\{n\}$ and $\omega=v_n$.
We want to prove that $L_{I_{i},I_{j}}=0$, for every $i$ and $j$ such
that $1\leq j\leq \delta-1$ and $j+2\leq i\leq \delta+1$, and each
column of $L_{I_{j+1},I_j}$ is nonzero, for every $j=1,\ldots
,\delta$.

So, fix $j$ with $1\leq j\leq \delta-1$ and $v_k\in V_j$. Then, we
need to prove that
\begin{eqnarray}\label{distances}
a_{i,k}=0\mbox{, for all }i\in I_{j+2}\cup\ldots \cup
  I_{\delta+1}\mbox{ ; } a_{i,k}\neq 0\mbox{, for some }i\in I_{j+1}.
\end{eqnarray}
Note that, for $j=\delta$, then $I_{j+1}=I_{\delta+1}=\{n\}$ and
$a_{i,k}=a_{n,k}\neq 0$, for all $k\in I_\delta$, since the vertices
$v_k\in V_\delta$ are such that $\ud(\omega,v_k)=1$.

Suppose that $a_{i,k}\neq 0$ for some $i\in I_{j+2}\cup\ldots \cup
I_{\delta+1}$. Let $v_n\to v_{i_1}\to \cdots \to v_{i_l}\to v_i$ be a
directed path of minimum unweighted length $l+1$, where $l+1\leq
\delta-j-1$ (because $i\in I_{j+2}\cup\ldots \cup
I_{\delta+1}$). Since $a_{i,k}\neq 0$, there is an arc $v_i\to v_k$ in
$\mbg$, and so a directed path from $v_n$ to $v_k$ of unweighted
length at most $\delta-j$, a contradiction, since
$\ud(v_n,v_k)=\delta+1-j$ (recall that $v_k\in V_j$). This proves the
first part of \eqref{distances}.

Now, take a directed path $v_n\to v_{i_1}\to \cdots \to
v_{i_{\delta-j}}\to v_k$ of minimum unweighted length
$\delta+1-j$. In particular, $\ud(\omega
,v_{i_{\delta-j}})=\delta-j$. Thus $i_{\delta-j}=:i\in I_{j+1}$ and
$a_{i,k}\neq 0$, since $v_i\to v_k$ is an arc of $\mbg$. This proves the
second part of \eqref{distances}.
\end{proof}

\begin{assumption}\label{echelon-omega}
Let $L$ be an $\icb$ matrix or, equivalently, let $\mbg_L$ be a
strongly connected digraph. From now on, we will assume that $L$ is
always in $\delta$-block echelon form or, equivalently, that $\mbg$
has a $(\omega,\delta)$-enumeration (see Lemmas~\ref{blockthenomega}
and \ref{omegathenblock}).

Note that, if $L$ is a $\pcb$ matrix or, equivalently, $\mbg$ is
strongly complete, this assumption is vacuous.
\end{assumption}

\subsection{The weighted reverse lexicographic order on the polynomial 
ring}$\phantom{+}$\bigskip

\noindent Let $L$ be an $\icb$ matrix or, equivalently, let $\mbg_L$
be a strongly connected digraph. Recall that, by
Assumption~\ref{grading}, $\deg(x_i)=\nu_i$ and
$\deg(x^\alpha)=\nu_1\alpha_1+\cdots +\nu_n\alpha_n$.

\begin{assumption}\label{wrlo}
From now on, we will always suppose that $\mba=\mbk[x]=\mbk[x_1,\ldots
  ,x_n]$ is endowed with the weighted reverse lexicographic order
($\wrlo$, for short). Concretely, given $\alpha,\beta\in\mbn^{n}$, we
set $x^\alpha>x^\beta$ precisely when either
$\deg(x^\alpha)>\deg(x^\beta)$, or $\deg(x^\alpha)=\deg(x^\beta)$ and
there exists $i\in\{1,\ldots ,n\}$ such that $\alpha_n=\beta_n,\ldots
,\alpha_{i+1}=\beta_{i+1}$ and $\alpha_i<\beta_i$. This last
condition will be written $\alpha-\beta=(*,\ldots,*,-,0,\ldots
,0)$. See, e.g, \cite[p.~13]{gp}, where this ordering is denoted by
${\bf wp}(\nu_1,\ldots ,\nu_n)$.
\end{assumption}

We insert a note of caution: it need no longer be the case that, under
this $\wrlo$, $x_1>\cdots >x_n$, since, for example, we have
$\deg(x_2)>\deg(x_1)$ whenever $\nu_2>\nu_1$.

Given a non-zero polynomial $f\in \mba$, we let $\lc(f)$, $\lm(f)$
and $\lt(f)$ denote the leading coefficient, the leading monomial
and the leading term, respectively, of $f$ with respect to the
$\wrlo$ (see, e.g., \cite[\S~2.2,~Definition~7]{clo}).

Recall that $f_C:=m_C-m_{\overline{C}}=\dif_1(C,\overline{C})$, where
$C\subseteq [n-1]$, $C\neq\varnothing$ (see Notation~\ref{fC}).

\begin{lemma}\label{lm}
Let $(C,\overline{C})\in\mcc_1$ and set
$f_C:=m_C-m_{\overline{C}}=\dif_1(C,\overline{C})$. Then $f_C$ is
homogeneous and $\lm(f_C)=m_C$. In particular, $\lt(f_C)=\lm(f_C)$.
\end{lemma}
\begin{proof}
That $f_C=m_C-m_{\overline{C}}$ is homogeneous is shown in
Remark~\ref{Iind1C}. So $\grade(m_C)=\grade(x^{C\to\overline{C}})$ is
equal to $\grade(m_{\overline{C}})=\grade(x^{\overline{C}\to
  C})$. Also by Remark~\ref{Iind1C}, we know that,
$m_C=x^{C\to\overline{C}}=x^{(L\chi_C)^+}$ and
$m_{\overline{C}}=x^{\overline{C}\to C}=x^{(L\chi_C)^-}$, where
$\chi_C$ is the incidence vector of $C$ (see Remark~\ref{srle}). Then
$(L\chi_C)^+-(L\chi_C)^-=L\chi_C$. So we must see that the rightmost
nonzero coefficient of $(L\chi_C)^{\top}$ is negative. But this
follows from Remark~\ref{rightmost}.
\end{proof}

If the variables $x_1,\ldots ,x_n$ are not enumerated according to
Assumption~\ref{echelon-omega}, then Lemma~\ref{lm} may fail.

\begin{example}
Consider the following $\cb$ matrices $L$ and $L^{\prime}$:
\begin{eqnarray*}
\newcommand\y{\cellcolor{gray!40}}
L=\left(\begin{array}{rrrr}
2&-2&0&0\\0&3&-3&0\\-1&0&5&-4\\0&0&-4&4
\end{array}\right), \mbox{ corresponding to }\mbg_L:
\vcenter{\xymatrix{ & \bullet_2 \ar[dr]^3 & \\ \bullet_1 \ar[ur]^2 & &
    \bullet_3 \ar[ll]_1 \ar@/^/[dl]^4 \\ & \bullet_4 \ar@/^/[ur]^4}};
\end{eqnarray*}
\begin{eqnarray*}
\newcommand\y{\cellcolor{gray!40}}
L^{\prime}=\left(\begin{array}{rrrr}
3&0&-3&0\\\y-2&2&0&0\\0&\y-1&5&-4\\0&0&\y-4&4
\end{array}\right), \mbox{ corresponding to }\mbg_{L^{\prime}}:
\vcenter{\xymatrix{ & \bullet_2 \ar[dl]_2 & \\ \bullet_1 \ar[rr]^3 & &
    \bullet_3 \ar[ul]_1 \ar@/^/[dl]^4 \\ & \bullet_4 \ar@/^/[ur]^4}}.
\end{eqnarray*}
Observe that $L^{\prime}$ arises from $L$ by the transposition of rows
and columns $1$ and $2$ of $L$. The corresponding digraphs $\mbg_L$
and $\mbg_{L^{\prime}}$ are the same, but with a re-enumeration of
vertices and arcs. Since $\mbg_L$ and $\mbg_{L^{\prime}}$ are strongly
connected, $L$ and $L^{\prime}$ are $\icb$ matrices. Note that
$L^{\prime}$ is in $3$-block echelon form, whereas $L$ is not.

A simple check shows that
\begin{eqnarray*}
&& \mu(L)=(12,8,24,24), \mbox{ so }\nu(L)=(3,2,6,6),\mbox{ and }\\&&
\mu(L^{\prime})=(8,12,24,24),\mbox{ so }\nu(L^{\prime})=(2,3,6,6).
\end{eqnarray*}
Imposing the usual $\wrlo$ on our polynomial ring $\mbk[x,y,z,t]$,
with $x_1,x_2,x_3,x_4$ being replaced by $x,y,z,t$, respectively, then
$\deg(x)=3$, $\deg(y)=2$, $\deg(z)=6$ and $\deg(t)=6$, when working
with $L$. Considering $L$ and $C:=\{2\}$, we have:
$f_{C}=m_{C}-m_{\overline{C}}$, $m_{C}=y^3$,
$m_{\overline{C}}=x^2$. Note that $y^3<x^2$ because
$(0,3,0,0)-(2,0,0,0)=(-2,3,0,0)$, which is consistent with
$(L\chi_C)^{\top}=(-2,3,0,0)$. Thus $\lm(f_C)=x^2=m_{\overline{C}}$.
\end{example}

\subsection{A Gr\"obner basis for the image of the first differential}
$\phantom{+}$\bigskip

\noindent Let $L$ be an $\icb$ matrix or, equivalently, let $\mbg_L$
be a strongly connected digraph.  Recall that $\mcb_k$ is the set of
elements of $\cyc_{n,k+1}$, enumerated employing the $\srle$ (see
Definition~\ref{cyccomplex}), forming a basis of the free
$\mbk[x]$-module $\mcc_k=\mbk[x]^{{\rm Cyc}_{n,k+1}}$. Concretely, in
degree zero: 
\begin{eqnarray*}
\dif_1(\mcb_1)=\{\dif_{1}(e_{1,1}),\ldots ,\dif_{1}(e_{1,r_1})\}=
\{f_{0,1},\ldots,f_{0,r_1}\}=\{f_C\mid C\subseteq [n-1],
C\neq\varnothing\}\subset \mcc_0.
\end{eqnarray*}
The purpose of this subsection is to prove that $\dif_1(\mcb_1)$ is a
Gr\"obner basis of $\dif_1(\mcc_1)$ (see also Notation~\ref{fC}). We
begin with some useful notations and remarks.

\begin{notation}\label{lcm-notation}
Given three disjoint subsets $A,B,C\subset [n]$, note that:
\begin{eqnarray}\label{disjoint}
x^{A\cup B\to C}=x^{A\to C}x^{B\to C}\mbox{ and }x^{A\to B\cup
  C}=x^{A\to B}x^{A\to C}.
\end{eqnarray}
We set:
\begin{eqnarray}
x^{(A\to B,C)^+} & := &
\frac{\Lcm(x^{A\to B},x^{A\to C})}{x^{A\to C}}= 
\frac{\prod_{i\in A}x_i^{\max(\sum_{j\in C}a_{i,j},\sum_{j\in B}a_{i,j})}}
{\prod_{i\in A}x_i^{\sum_{j\in C}a_{i,j}}} \\ & = & 
\prod_{i\in A}x_i^{(\sum_{j\in B}a_{i,j}-\sum_{j\in C}a_{i,j})^{+}}. \notag
\end{eqnarray}
In particular,
\begin{eqnarray}\label{lcm-equality}
\Lcm(x^{A\to B},x^{A\to C})=x^{(A\to B,C)^+}x^{A\to C}=x^{(A\to C,B)^+}x^{A\to B}.
\end{eqnarray}
\end{notation}

\begin{notation}\label{FG-notation}
Let $C$ and $D$ be two non-empty subsets of $[n-1]$. Write $E=C\cap
D$, $F=C\setminus E$, $G=D\setminus E$, $U=C\cup D$ and
$V=\overline{U}=[n]\setminus U$. In particular,
$\overline{C}=[n]\setminus C=G\cup V$ and $\overline{D}=[n]\setminus
D=F\cup V$.  The following picture may help in reading everything that
will follow.
\begin{eqnarray*}
\begin{array}{c|c|c}
\cap &D&\overline{D}\\\hline
C&E&F\\\hline
\overline{C}&G&V.
\end{array}
\end{eqnarray*}
\end{notation}

\begin{lemma}\label{fc=fd}
The restriction of the map $\dif_1:\mcc_1\to\mcc_0$ to the set
$\mcb_1$ is injective. Concretely, given $C,D$, two non-empty subsets
of $[n-1]$, then $f_C=f_D$ if and only if $C=D$.
\end{lemma}
\begin{proof}
Suppose that $f_C=f_D$. By Lemma~\ref{lm},
$m_C=\lt(f_C)=\lt(f_D)=m_D$. Thus $m_C=m_D$ and hence
$m_{\overline{C}}=m_{\overline{D}}$. Suppose that $C\neq D$. Since the
initial hypotheses are symmetrical in $C$ and $D$, we may assume
without loss of generality that $C\not\subseteq D$, that is, that
$F\neq\varnothing$, with Notation~\ref{FG-notation} in force. Using
\eqref{disjoint}, one has:
\begin{eqnarray*}
f_C & = & m_{C}-m_{\overline{C}}=x^{C\to \overline{C}}-x^{\overline{C}\to C}\\ & = &
  x^{E\to G}x^{E\to V}x^{F\to G}x^{F\to V}-x^{G\to E}x^{G\to F}x^{V\to
    E}x^{V\to F}
\end{eqnarray*}
and 
\begin{eqnarray*}
f_D & = & m_{D}-m_{\overline{D}}=x^{D\to \overline{D}}-x^{\overline{D}\to
    D} \\ & = & x^{E\to F}x^{E\to V}x^{G\to F}x^{G\to V}-x^{F\to
    E}x^{F\to G}x^{V\to E}x^{V\to G}.
\end{eqnarray*}
It follows from the equality $m_C=m_D$ that $x^{F\to G\cup V}=1$,
i.e., that $a_{i,j}=0$, for all $i\in F$ and $j\in \overline{C}=G\cup
V$. We deduce that $x^{F\to E\cup G}=1$ from the equality
$m_{\overline{C}}=m_{\overline{D}}$, i.e., that $a_{i,j}=0$, for all
$i\in F$ and $j\in D=E\cup G$. Thus $a_{i,j}=0$, for all $i\in F$ and
$j\in E\cup G\cup V=\overline{F}$. Therefore there is no directed path
joining any vertex in $F\neq\varnothing$ to the vertex $v_n\in
\overline{F}$, which is a contradiction to the hypothesis that
$\mbg_L$ is a strongly connected digraph.
\end{proof}

Let us calculate the $S$-polynomial $S(f_C,f_D)$ of
$f_C=\dif_1(C,\overline{C})$ and $f_D=\dif_1(D,\overline{D})$.
Observe that if $C=D$, then $f_C=f_D$ and $S(f_C,f_D)=0$.

\begin{lemma}\label{s-poly}
Let $C$ and $D$ be two non-empty distinct subsets of $[n-1]$. Consider
the monomials $l_{C,D}$ and $l_{D,C}$, defined as follows:
\begin{eqnarray*}
l_{C,D}=x^{(E\to G,F)^{+}}x^{F\to G}x^{V\to D}\phantom{+}\mbox{ and }\phantom{+}
l_{D,C}=x^{(E\to F,G)^{+}}x^{G\to F}x^{V\to C}.
\end{eqnarray*}
\begin{itemize}
\item[$(a)$] Suppose that $C\not\subset D$ and that $D\not\subset
  C$. Then $F,G$ are non-empty subsets of $[n-1]$, so
  $(F,\overline{F})$ and $(G,\overline{G})$ are in $\mcc_1$. Then the
  $S$-polynomial of $f_C$ and $f_D$ is
\begin{eqnarray}\label{sfcfd}
S(f_C,f_D)=l_{C,D}f_F-l_{D,C}f_G,
\end{eqnarray}
with $f_F=\dif_1(F,\overline{F})$ and
$f_G=\dif_1(G,\overline{G})$. Moreover,
\begin{eqnarray*}
\lm(S(f_C,f_D))\geq\lm(l_{C,D}f_F),\lm(l_{D,C}f_G).
\end{eqnarray*}
\item[$(b)$] Suppose, without loss of generality, that $C\subsetneq
  D$. Then $F=\varnothing$ and $G$ is a non-empty subset of $[n-1]$, so
  $(G,\overline{G})\in\mcc_1$. Then the $S$-polynomial of $f_C$ and
  $f_D$ is
\begin{eqnarray}\label{sfcfd0}
S(f_C,f_D)=-l_{D,C}f_G, 
\end{eqnarray}
with $f_G=\dif_1(G,\overline{G})$.  Moreover,
$\lm(S(f_C,f_D))\geq\lm(l_{D,C}f_G)$.
\end{itemize}
\end{lemma}
\begin{proof}
By definition (see, e.g., \cite[\S~2.6, Definition~4]{clo}),
$S(f_C,f_D)=m_{D,C}f_C-m_{C,D}f_D$, where
\begin{eqnarray}\label{mdc}
\phantom{+++} m_{D,C}=\frac{\Lcm(\lm(f_D),\lm(f_C))}{\lt(f_C)}\;
\mbox{ and }\; m_{C,D}=\frac{\Lcm(\lm(f_C),\lm(f_D))}{\lt(f_D)}.
\end{eqnarray}
Suppose that $C\not\subset D$ and $D\not\subset C$, so that
$F,G\neq\varnothing$. By Lemma~\ref{lm},
$\lt(f_C)=\lm(f_C)=m_C=x^{C\to\overline{C}}$ and
$\lt(f_D)=\lm(f_D)=m_D=x^{D\to\overline{D}}$. Using \eqref{disjoint}
and \eqref{lcm-equality}, and the fact that $\Lcm(fg,fh)=f\cdot
\Lcm(g,h)$ for elements $f,g,h\in \mba=\mbk[x]$, one has:
\begin{eqnarray}\label{lcmCD/C}
m_{D,C} & = &\frac{\Lcm(\lm(f_D),\lm(f_C))}{\lt(f_C)}\\ & = &
  \frac{\Lcm(x^{E\to F},x^{E\to G})}{x^{E\to G}}x^{G\to F}x^{G\to
    V}=x^{(E\to F,G)^{+}}x^{G\to F}x^{G\to V}. \notag
\end{eqnarray}
In a symmetric manner,
\begin{eqnarray}\label{lcmCD/D}
m_{C,D}=\frac{\Lcm(\lm(f_C),\lm(f_D))}{\lt(f_D)}=x^{(E\to
  G,F)^{+}}x^{F\to G}x^{F\to V}.
\end{eqnarray}
Set $g_{G,F}:=x^{(E\to F,G)^{+}}x^{G\to F}$ and $g_{F,G}:=x^{(E\to
  G,F)^{+}}x^{F\to G}$. Therefore $m_{D,C}=g_{G,F}x^{G\to V}$ and
$l_{D,C}=g_{G,F}x^{V\to C}$. Analogously, $m_{C,D}=g_{F,G}x^{F\to V}$
and $l_{C,D}=g_{F,G}x^{V\to D}$. Then
\begin{eqnarray*}
S(f_C,f_D) & = &g_{G,F}x^{G\to V}f_C-g_{F,G}x^{F\to V}f_{D}
\\ & = & g_{G,F}x^{G\to V}m_C-g_{G,F}x^{G\to V}m_{\overline{C}}-
g_{F,G}x^{F\to V}m_{D}+g_{F,G}x^{F\to V}m_{\overline{D}}.
\end{eqnarray*}
Set
\begin{eqnarray*}
&&m_1:=g_{G,F}x^{G\to V}m_C\phantom{+}\mbox{ ,
  }\phantom{+}m_2:=g_{G,F}x^{G\to V}m_{\overline{C}}\mbox{ , }
  \\&&m_3:=g_{F,G}x^{F\to V}m_{D}\phantom{+}\mbox{ ,
  }\phantom{+}m_4:=g_{F,G}x^{F\to V}m_{\overline{D}}\mbox{ , }
\end{eqnarray*}
so that $S(f_C,f_D)=m_1-m_2-m_3+m_4$. An easy computation using
\eqref{disjoint} shows that
\begin{eqnarray*}
m_1-m_3=(g_{G,F}x^{C\to G}-g_{F,G}x^{D\to F})x^{U\to V}.
\end{eqnarray*}
However, using \eqref{disjoint} and then \eqref{lcm-equality}, we get
\begin{eqnarray}\label{zero}
&& g_{G,F}x^{C\to G}-g_{F,G}x^{D\to F}  \\
&& =  x^{(E\to F,G)^{+}}x^{G\to F}x^{E\to G}x^{F\to G}-
x^{(E\to G,F)^{+}}x^{F\to G}x^{E\to F}x^{G\to F}\notag \\
&& =  (x^{(E\to F,G)^{+}}x^{E\to G}-x^{(E\to G,F)^{+}}x^{E\to F})x^{F\to G}x^{G\to F}=0.
\notag
\end{eqnarray}
Therefore, $S(f_C,f_D)=-m_2+m_4$. On the other hand,
\begin{eqnarray*}
l_{C,D}f_F-l_{D,C}f_G=l_{C,D}m_F-l_{C,D}m_{\overline{F}}-
l_{D,C}m_G+l_{D,C}f_{\overline{G}}.
\end{eqnarray*}
Set
\begin{eqnarray*}
&&m_5:=l_{C,D}m_{F}=g_{F,G}x^{V\to D}m_F\phantom{+}\mbox{ ,
  }\phantom{+} m_6:=l_{C,D}m_{\overline{F}}=g_{F,G}x^{V\to
    D}m_{\overline{F}}\mbox{ , }\\&& m_7:=l_{D,C}m_G=g_{G,F}x^{V\to
    C}m_G\phantom{+}\mbox{ ,
  }\phantom{+}m_8:=l_{D,C}m_{\overline{G}}=g_{G,F}x^{V\to
    C}m_{\overline{G}}\mbox{ , }
\end{eqnarray*}
so that $l_{C,D}f_F-l_{D,C}f_G=m_5-m_6-m_7+m_8$. An easy computation
using \eqref{disjoint} shows that
\begin{eqnarray*}
-m_6+m_8=(-g_{F,G}x^{D\to F}+g_{G,F}x^{C\to G})x^{V\to U}.
\end{eqnarray*}
The situation is symmetrical in $C$ and $D$. Swapping $C$ and $D$ in
an expression forces a swap in $F$ and $G$, with $E$, $U$ and $V$
invariant. Thus $-g_{F,G}x^{D\to F}+g_{G,F}x^{C\to G}=0$ follows by
symmetry from the equality \eqref{zero}. Therefore, $-m_6+m_8=0$ and
$l_{C,D}f_F-l_{D,C}f_G=m_5-m_7$. An easy computation using
\eqref{disjoint} shows that $m_2=m_7$ and, by symmetry,
$m_4=m_5$. Therefore,
$S(f_C,f_D)=-m_2+m_4=m_5-m_7=l_{C,D}f_F-l_{D,C}f_G$.

Observe that one can consider
$(F,\overline{F}),(G,\overline{G})\in\mcc_1$ due to the hypothesis
$F,G\neq\varnothing$. Using Lemma~\ref{lm}, it follows that
$\lm(f_F)=m_F$ and $\lm(f_G)=m_G$. Since $l_{C,D}$ and $l_{D,C}$ are
monomials, then $\lm(l_{C,D}f_F)=l_{C,D}m_F=m_5$ and
$\lm(l_{D,C}f_G)=l_{D,C}m_G=m_7$ (see, e.g., \cite[\S2.2,
  Exercise~11]{clo}).  On the other hand, $S(f_C,f_D)=-m_2+m_4$. Thus,
one of the two monomials, either $m_2$, or else $m_4$, is the leading
monomial of $S(f_C,f_D)$. If $m_2<m_4$, then
$\lm(l_{C,D}f_F)=m_5=m_4=\lm(S(f_C,f_D))$ and
$\lm(l_{D,C}f_G)=m_7=m_2<m_4=\lm(S(f_C,f_D))$. On the contrary, if
$m_4<m_2$, then $\lm(l_{C,D}f_F)=m_5=m_4<m_2=\lm(S(f_C,f_D))$ and
$\lm(l_{D,C}f_G)=m_7=m_2=\lm(S(f_C,f_D))$.

Finally, assume, without loss of generality, that $C\subsetneq
D$. Then $F=\varnothing$ and $f_F=0$. Hence $g_{G,F}=1$ and
$g_{F,G}=x^{C\to G}$; $m_{D,C}=x^{G\to V}$ and $l_{D,C}=x^{V\to C}$;
$m_{C,D}=x^{C\to G}$ and $l_{C,D}=x^{C\to G}x^{V\to D}$. As before,
$m_1-m_3=0$ and $S(f_C,f_D)=-m_2+m_4$. On the other hand, since
$F=\varnothing$, then $f_F=0$ and $l_{C,D}f_F-l_{D,C}f_G=-l_{D,C}f_G$,
which is readily seen to be equal to $-m_2+m_4$. Moreover, trivially
then, one has $\lm(S(f_C,f_D))\geq\lm(l_{D,C}f_G)$.
\end{proof}

\begin{proposition}\label{GB1}
The set $\dif_1(\mcb_1)=\{f_C\mid C\subseteq [n-1],
C\neq\varnothing\}$ is a Gr\"obner basis for the ideal
$\dif_1(\mcc_1)$.
\end{proposition}
\begin{proof}
This follows from Lemma~\ref{s-poly} and the criterion of Buchberger
(see, e.g., \cite[\S~2.6, Theorem~6 and \S~2.9, Theorem~3]{clo}).
\end{proof}

\subsection{The Cyc complex is exact in degree zero}\label{exact-degree0}
$\phantom{+}$\bigskip

\noindent Let $L$ be an $\icb$ matrix or, equivalently, let $\mbg$ be
a strongly connected digraph.  Recall that $\dif_0:\mcc_0=\mbk[x]\to
\mbk[x]/I(\mcl)$ is defined to be the natural projection onto the
quotient ring, so that $\dif_1(\mcc_1)\subseteq I(\mcl)$
(Definition~\ref{dif0}).

\begin{proposition}\label{colonideal}
We have $\dif_1(\mcc_1):x_n=\dif_1(\mcc_1)$ and
$\dif_1(\mcc_1)=I(\mcl)=\ker(\dif_0)$.
\end{proposition}
\begin{proof}
According to Definition~\ref{cyccomplex}, let 
\begin{eqnarray*}
\mcb_1=\{e_{1,1},\ldots ,e_{1,r_1}\}\mbox{ and
}\dif_1(\mcb_1)=\{f_{0,1},\ldots ,f_{0,r_1}\}\subset \mbk[x],
\end{eqnarray*}
enumerated employing the $\srle$.  By
Proposition~\ref{GB1}, $\dif_1(\mcb_1)$ is a Gr\"obner basis for the
ideal $\dif_1(\mcc_1)$. Now, take $f\in \dif_1(\mcc_1):x_n$. Using the
Division Algorithm, $f$ can be written as
\begin{eqnarray}\label{divalg}
f=g_1f_{0,1}+\cdots +g_{r_1}f_{0,r_1}+h, 
\end{eqnarray}
where either $h=0$, or else $h\neq 0$ and no term in $h$ is divisible
by any one of $\lt(f_{0,i})$, for $i=1,\ldots ,r_1$ (see, e.g.,
\cite[\S~2.3, Theorem~3]{clo}). Suppose that $h\neq 0$. On multiplying
by $x_n$ in \eqref{divalg}, one deduces that $x_nh\in \dif_1(\mcc_1)$
and so $\lt(x_nh)\in \lt(\dif_1(\mcc_1))$. Since $f_{0,1},\ldots
,f_{0,r_1}$ is a Gr\"obner basis of $\dif_1(\mcc_1)$,
$\lt(\dif_1(\mcc_1))=\langle \lt(f_{0,1}),\ldots
,\lt(f_{0,r_1})\rangle$. By, e.g., \cite[\S~2.4, Lemma~2]{clo},
$\lt(x_nh)=x_n\lt(h)$ is a multiple of some $\lt(f_{0,_i})$.  Since
$x_n$ does not occur in any $\lt(f_{0,i})$, it follows that $\lt(h)$
is a multiple of some $\lt(f_{0,i})$, a contradiction. Therefore
$\dif_1(\mcc_1):x_n=\dif_1(\mcc_1)$. In particular,
$\dif_1(\mcc_1):x_n^{m}=\dif_1(\mcc_1)$, for all $m\geq 1$.

By Remark~\ref{Iind1C}, $I(L)\subseteq \dif_1(\mcc_1)\subseteq
I(\mcl)$.  By \cite[Proposition~5.7]{opv}, $I(\mcl)=I(L):x_n^m$, for
some $m\gg 0$. Therefore
\begin{eqnarray*}
I(\mcl)=I(L):x_n^m\subseteq
\dif_1(\mcc_1):x_n^m=\dif_1(\mcc_1)\subseteq I(\mcl),
\end{eqnarray*}
so $I(\mcl)=\dif_1(\mcc_1)$ as desired.
\end{proof}

Putting together Propositions~\ref{GB1} and \ref{colonideal}, we get
the following result.

\begin{corollary}\label{GBI(L)}
The set $\dif_1(\mcb_1)\subset\mcc_0$ is a Gr\"obner basis of
$\dif_1(\mcc_1)=I(\mcl)=\ker(\dif_0)$.
\end{corollary}

\subsection{Minimality in the strongly complete case}
$\phantom{+}$\bigskip

\noindent Let $L$ be an $\icb$ matrix or, equivalently, let $\mbg$ be
a strongly connected digraph.

\begin{proposition}\label{minGB}
Suppose that $\mbg$ is strongly complete. Then $\dif_1(\mcb_1)$ is a
minimal Gr\"obner basis and a minimal homogeneous system of generators
of $I(\mcl)$.
\end{proposition}
\begin{proof}
Recall that $\dif_1(\mcb_1)=\{f_{0,1},\ldots ,f_{0,r_1}\}=\{f_C\mid
C\subseteq [n-1], C\neq\varnothing \}$. By Proposition~\ref{GBI(L)},
$\dif_1(\mcb_1)$ is a Gr\"obner basis of $I(\mcl)$. Let us see that it
is a minimal such basis (see, e.g.,
\cite[\S~2.7. Definition~4]{clo}). Fix $f_C$, $C\subseteq [n-1]$,
$C\neq\varnothing$. Clearly, the leading coefficient of $f_C$ is
$1$. Moreover, the leading term of $f_C$ is
$\lt(f_C)=m_C=x^{C\to\overline{C}}$ (see Remark~\ref{fC} and
Lemma~\ref{lm}). Suppose that $\lt(f_C)$ lies in the ideal
$\langle\lt(f_D)\mid D\subseteq [n-1], D\neq C, D\neq
\varnothing\rangle$. Since the latter is a monomial ideal, it follows
that $\lt(f_D)$ divides $\lt(f_C)$, for some $D\subseteq [n-1]$,
$D\neq C$, $D\neq\varnothing$. Now, use that $G$ is strongly complete,
or equivalently, $L$ is a $\pcb$ matrix. Therefore, $a_{i,j}>0$, for
all $i,j=1,\ldots ,n$. Since 
\begin{eqnarray*}
x^{D\to\overline{D}}=\prod_{i\in
  D}x_i^{\sum_{j\in\overline{D}}a_{i,j}}\mbox{ divides }
x^{C\to\overline{C}}=\prod_{i\in
  C}x_i^{\sum_{j\in\overline{C}}a_{i,j}} 
\end{eqnarray*}
and $\sum_{j\in\overline{D}}a_{i,j}>0$, it follows that $D\subsetneq
C$, and so $\overline{C}$ is a proper subset of
$\overline{D}$. However, if $i\in D$, then the exponent of $x_i$ in
$m_D$ is $\sum_{j\in\overline{D}}a_{i,j}$, whereas the exponent of
$x_i$ in $m_C$ is $\sum_{j\in\overline{C}}a_{i,j}$, which is strictly
smaller, because $\overline{C}\subsetneq \overline{D}$, a
contradiction.

By Corollary~\ref{GBI(L)}, $\dif_1(\mcb_1)$ is a homogeneous system of
generators for $\dif_1(\mcc_1)=I(\mcl)$. Let us see that it is a
minimal one.  Suppose that
\begin{eqnarray}\label{linearcomb}
u_1f_{0,1}+\cdots +u_{r_1}f_{0,r_1}=0,
\end{eqnarray}
for some polynomials $u_i\in \mba=\mbk[x]$. In other words,
$(u_1,\ldots ,u_{r_1})\in\mba^{r_1}$ is in the kernel of
$\varphi:\mba^{r_1}\to \dif_1(\mcc_1)$, the free $\mba$-linear
presentation of $\dif_1(\mcc_1)$, sending each element of the
canonical basis $e_i$ to $f_{0,i}$. Write each $u_i$ in the form
$u_i=v_i+x_nw_i$, $v_i,w_i\in\mbk[x]$, where $v_i$ contains no term
involving $x_n$. Set $x_n=0$ in the equality \eqref{linearcomb}
above. Then we get
\begin{eqnarray*}
v_1\lt(f_{0,1})+\cdots + v_{r_1}\lt(f_{0,r_1})=0,
\end{eqnarray*}
so $(v_1,\ldots ,v_{r_1})$ is a syzygy on the leading terms
$\lt(f_{0,1}),\ldots ,\lt(f_{0,r_1})$.  By
\cite[\S~2.9. Proposition~8, see also Definition~5]{clo}, $(v_1,\ldots
,v_{r_1})$ is of the form $\sum_{i<j}v_{i,j}S_{i,j}$, where
\begin{eqnarray*}
S_{i,j}=\frac{\Lcm(\lm(f_{0,_j}),\lm(f_{0,i}))}{\lt(f_{0,i})}e_i-
\frac{\Lcm(\lm(f_{0,i}),\lm(f_{0,j}))}{\lt(f_{0,j})}e_j.
\end{eqnarray*}
We have $f_{0,i}=f_C$ and $f_{0,j}=f_D$, for some $C,D\subseteq
[n-1]$, $C,D\neq \varnothing$ and $C\neq D$. Then
\begin{eqnarray*}
&&\frac{\Lcm(\lm(f_{0,_j}),\lm(f_{0,i}))}{\lt(f_{0,i})}=
\frac{\Lcm(\lm(f_D),\lm(f_C))}{\lt(f_C)}=m_{D,C}\mbox{ and }
\\&& \frac{\Lcm(\lm(f_{0,_i}),\lm(f_{0,j}))}{\lt(f_{0,j})}=
\frac{\Lcm(\lm(f_C),\lm(f_D))}{\lt(f_D)}=m_{C,D},
\end{eqnarray*}
with the notations as in the equality \eqref{mdc}. Writing $E=C\cap
D$, $F=C\setminus E$, $G=D\setminus E$, $U=C\cup D$ and
$V=\overline{U}=[n]\setminus U$ and using the equalities
\eqref{lcmCD/C} and \eqref{lcmCD/D} in the proof of
Lemma~\ref{s-poly}, then
\begin{eqnarray*}
m_{D,C}=x^{(E\to F,G)^{+}}x^{G\to F}x^{G\to V}\mbox{ and }
m_{C,D}=x^{(E\to G,F)^{+}}x^{F\to G}x^{F\to V}.
\end{eqnarray*}
Suppose that $C\not\subset D$ and $D\not\subset C$. Then $F$ and $G$
are non-empty. Since $a_{i,j}>0$, for all $i,j=1,\ldots ,n$, we deduce
that $x^{G\to F}$ and $x^{F\to G}$, and so $m_{D,C}$ and $m_{C,D}$,
are in the irrelevant maximal ideal $\mfm$. Suppose, without loss of
generality, that $C\subsetneq D$. Then $m_{D,C}=x^{G\to V}$ and
$m_{C,D}=x^{C\to G}$, where $G$, $V$ and $C$ are non-empty (see the
last paragraph in the proof of Lemma~\ref{s-poly}). Again, we deduce
that $m_{D,C}$ and $m_{C,D}$ are in $\mfm$. Hence each $u_i$ also lies
in $\mfm$ and $(u_1,\ldots ,u_{r_1})\in\mfm\mba^{r_1}$. On tensoring
the short exact sequence
\begin{eqnarray*}
0\to\ker(\varphi)\to\mba^{r_1}\buildrel{\varphi}\over{\to}\dif_1(\mcc_1)\to
0 
\end{eqnarray*}
by $\mba/\mfm=k$, we deduce that the induced map
$\overline{\varphi}:\mba^{r_1}/\mfm\mba^{r_1}
\to\dif_1(\mcc_1)/\mfm\dif_1(\mcc_1)$ is an isomorphism of $k$-vector
spaces. Applying the graded Nakayama Lemma, we deduce that
$\dif_1(\mcb_1)$ is a minimal homogeneous system of generators for
$\dif_1(\mcc_1)$ (see, e.g., \cite[Exercise~1.5.24]{bh}).
\end{proof}

\begin{remark}
The minimality for the strongly complete case will also follow from
Corollary~\ref{min-resolution}, since
$\ker(\dif_1)=\dif_2(\mcc_2)\subset\mfm\mcc_1$ and so
$\mcc_1/\mfm\mcc_1\cong\dif_1(\mcc_1)/\mfm\dif_1(\mcc_1)$.
\end{remark}

\subsection{Non-minimality in the non-strongly complete case}
$\phantom{+}$\bigskip

\noindent Let $L$ be an $\icb$ matrix or, equivalently, let $\mbg$ be
a strongly connected digraph. It is easy to give an example where, if
$\mbg$ is not strongly complete, then $\dif_1(\mcb_1)$ is not a
minimal system of generators of $I(\mcl)$.

\begin{example}
Consider the following $\icb$ matrix $L$, which is not a $\pcb$
matrix:
\begin{eqnarray*}
\newcommand\y{\cellcolor{gray!40}}
L=\left(\begin{array}{rrrr}
1&0&0&-1\\\y-1&1&0&0\\0&\y-1&1&0\\0&0&\y-1&1
\end{array}\right), \mbox{ corresponding to }\mbg_{L}:
\vcenter{\xymatrix{ & \bullet_2 \ar[dl]_1 & \\ \bullet_1 \ar[dr]_1 & &
    \bullet_3 \ar[ul]_1\\ & \bullet_4 \ar[ur]_1}}.
\end{eqnarray*}
Then $\dif_1(\mcb_1)=\{x-t,y-t,xz-yt,x-z,z-t,y-z,x-y\}$, whereas
$I(L)=(x-y,y-z,z-t,-x+t)$, which is clearly minimally generated by
$\{x-y,y-z,z-t\}$. Thus $I(L)$ is a complete intersection ideal. In
particular, it is unmixed and $I(L)=I(\mcl)$. Therefore
$\dif_1(\mcb_1)$ cannot be a minimal homogeneous system of generators
of $I(\mcl)$.
\end{example}

\section{Computing syzygies of a Gr\"obner basis}

\subsection{Preliminary notations}
$\phantom{+}$\bigskip

\noindent Let us recall some preliminary notations. We follow
\cite[Section~2]{emss}, \cite{bs} and \cite[\S~15]{eisenbud}.

\begin{definitions}\label{firstsetofdefinitions}
Let $F=\mba^r$ be the free $\mba=\mbk[x]$-module of rank $r$, and let
$e_1,\ldots ,e_r$ be the canonical basis of $F$.
\begin{itemize}
\item[$(1)$] \underline{\sc Monomials}. A monomial $me_i$ in $F$ is
  the product of a monomial $m$ in $\mba=\mbk[x]$ with a basis element
  $e_i$. A term $\lambda me_i$ in $F$ is the product of a monomial
  $me_i$ in $F$ with a scalar $\lambda$ in $\mbk$.
\item[$(2)$] \underline{\sc Divisibility}. A monomial $m_1e_i$ divides
  a monomial $m_2e_j$ if $i=j$ and $m_1$ divides $m_2$; in this case,
  $m_2e_j/m_1e_i$ is defined as $m_2/m_1$. Similarly, a monomial $m_1$
  divides $m_2e_j$ if $m_1$ divides $m_2$ and in this case
  $m_2e_j/m_1$ is defined as $(m_2/m_1)e_j$. The least common multiple
  of two monomials $m_1e_i$ and $m_2e_j$ of $F$ is
  $\Lcm(m_1e_i,m_2e_j)=\Lcm(m_1,m_2)e_i$, if $i=j$, or $0$ otherwise.
\item[$(3)$] \underline{\sc Orderings}. A (global) monomial ordering
  on $F$ is a total ordering $>$ on the set of monomials of $F$ such
  that if $m_1e_i$ and $m_2e_j$ are monomials in $F$ and $m$ is a
  monomial in $\mba$, then $m_1e_i> m_2e_j\Rightarrow (m\cdot m_1)e_i>
  (m\cdot m_2)e_j$. Moreover, $me_i>e_i$ for all $i$ and all monomials
  $m\neq 1$, and we require in addition that
  $m_1e_i>m_2e_i\Leftrightarrow m_1e_j>m_2e_j$, for all $i,j$.
\item[$(4)$] \underline{\sc Leading terms}. Let $>$ be a monomial
  ordering on $F$ and $f=\lambda me_i+\mbox{ lower order terms}$,
  $f\in F\setminus\{0\}$, where $\lambda\neq 0$. Then the leading
  coefficient, the leading monomial and the leading term of $f$ are
  $\lc(f)=\lambda$, $\lm(f)=me_i$ and $\lt(f)=\lambda me_i$,
  respectively. For any subset $S\subset F$, the leading module of $S$
  is defined as $\lmod(S):=\langle\lm(f)\mid f\in
  S\setminus\{0\}\rangle$.
\end{itemize}
\end{definitions}

\begin{definitions}\label{secondsetofdefinitions}
Let $F_0=\mba^s$ be the free $\mba$-module of rank $s$ and let
$G_0=\{f_1,\ldots ,f_r\}$ be a finite subset of
$F_0\setminus\{0\}$. Take $F_1=\mba^r$, the free $\mba$-module of rank $r$,
the cardinality of $G_0$, and let $\varphi_1:F_1\rightarrow F_0$ be
defined by $\varphi_1(e_i)=f_i$.
\begin{itemize}
\item[$(5)$] \underline{\sc Syzygies}. The (first) syzygy module of
  $G_0$ is $\syz(G_0):=\ker(\varphi_1)$. An element of $\ker(\varphi_1)$ is
  called a syzygy of $G_0$.
\item[$(6)$] \underline{\sc Induced orderings}. Given a monomial
  ordering $>$ on $F_0$, the induced ordering on $F_1$ (w.r.t. $>$ and
  $G_0$) is the monomial ordering $>$ defined by
  $m_1e_i>m_2e_j\Leftrightarrow \lm(m_1f_i)>\lm(m_2f_j)$, or
  $\lm(m_1f_i)=\lm(m_2f_j)$ and $i>j$.
\item[$(7)$] \underline{\sc The $S$-vectors}. For $i,j\in\{ 1,\ldots
  ,r\}$, the $S$-vector of $f_i$ and $f_j$ is defined as
\begin{eqnarray*}
S(f_i,f_j)=m^1_{j,i}f_i-m^1_{i,j}f_j\in \langle G_0 \rangle\subset
F_0,\mbox{ where } m^1_{j,i}:=\frac{\Lcm(\lm(f_j),\lm(f_i))}{\lt(f_i)}.
\end{eqnarray*}
In particular,
$\lt(m^1_{j,i}f_i)=m^1_{j,i}\lt(f_i)=\Lcm(\lm(f_j),\lm(f_i))$, which,
by symmetry, will be equal to $\lt(m^1_{i,j}f_j)$. Therefore
$\lm(m^1_{j,i}f_i)=\lm(m^1_{i,j}f_j)>\lm(S(f_i,f_j))$.
\item[$(8)$] \underline{\sc Standard expression}. Let $g\in F_0$. An
  equality $g=g_1f_1+\cdots +g_rf_r+h$, with $g_i\in \mba$ and $h\in
  F_0$, is a standard expression for $g$ with remainder $h$ (and
  w.r.t. $>$ and $G_0$) if the following two conditions are satisfied:
\begin{itemize}
\item[$(a)$] $\lm(g)\geq \lm(g_if_i)$, for all $i=1,\ldots ,r$,
  whenever both $g$ and $g_if_i$ are nonzero.
\item[$(b)$] If $h$ is nonzero, then $\lt(h)$ is not divisible by any
  $\lt(f_i)$.
\end{itemize}
\end{itemize}
\end{definitions}

\begin{discussion}\label{disc}
Let $F_0=\mba^s$ be the free $\mba$-module of rank $s$ and let
$G_0=\{f_1,\ldots ,f_r\}$ be a finite subset of
$F_0\setminus\{0\}$. Take $F_1=\mba^r$, the free $\mba$-module of rank
$r$, the cardinality of $G_0$, and let $\varphi_1:F_1\rightarrow F_0$
be defined by $\varphi_1(e_i)=f_i$. Suppose that $G_0$ is a Gr\"obner
basis w.r.t. some ordering $>$. By the Buchberger Criterion (see,
e.g., \cite[Theorem~15.8]{eisenbud}), there are standard
expressions with remainder zero:
\begin{eqnarray}\label{standardexp}
S(f_i,f_j)=m^1_{j,i}f_i-m^1_{i,j}f_j=g_1^{(i,j)}f_1+\cdots
+g_r^{(i,j)}f_r.
\end{eqnarray}
In particular, $\lm(m^1_{j,i}f_i)=\lm(m^1_{i,j}f_j)>\lm(S(f_i,f_j))
\geq \lm(g_s^{(i,j)}f_s)$, for all $s=1,\ldots, r$, whenever both
$S(f_i,f_j)$ and the particular $g_s^{(i,j)}f_s$ are nonzero.

Each such expression \eqref{standardexp} defines a syzygy of $G_0$ in
$F_1$, associated to the $S$-vector of $f_i$ and $f_j$:
\begin{eqnarray}\label{syzygy}
m^1_{j,i}e_i-m^1_{i,j}e_j-(g_1^{(i,j)}e_1+\cdots
+g_r^{(i,j)}e_r).
\end{eqnarray}
This element will be called a $\tau$-syzygy associated to the
$S$-vector of $f_i$ and $f_j$. Note that, unless the standard
expression \eqref{standardexp} is a ``determinate division with
reminder'', the elements $g_s^{(i,j)}$ are not necessarily uniquely
determined (see \cite[Theorem~1.3]{bs}, \cite[Theorem~2.2.12]{ds}).

However, if $i>j$, since $\lm(m^1_{j,i}f_i)=\lm(m^1_{i,j}f_j)$, then
$m^1_{j,i}e_i>m^1_{i,j}e_j$.  Moreover, since
$\lm(m^1_{i,j}f_j)>\lm(S(f_i,f_j))\geq \lm(g_s^{(i,j)}f_s)$, then
$m^1_{i,j}e_j>\lm(g_s^{(i,j)})e_s$, for all $s=1,\ldots ,r$, whenever
the particular $g_s^{(i,j)}f_s$ is nonzero.

Thus, if $\sigma\in\ker(\varphi_1)$ is a $\tau$-syzygy associated to
the $S$-vector of two elements of $G_0$, obtained as described above,
setting $e_i$ (respectively, $e_j$) to be the unique basis element
involved in $\lt(\sigma)/\lc(\sigma)$ (respectively, 
$\lt(\sigma-\lt(\sigma))/\lc(\sigma-\lt(\sigma))$), one deduces
that $\sigma$ is a $\tau$-syzygy associated to the $S$-vector of
$f_i=\varphi_1(e_i)$ and $f_j=\varphi_1(e_j)$, with $i>j$. In
particular, $\lt(\sigma)=m^1_{j,i}e_i$ (see \cite[p.6]{bs},
\cite[p.68]{ds} or \cite[Remark~3.6]{emss}).

A $\tau$-syzygy associated to the $S$-vector of $f_i$ and $f_j$ will
be denoted by $\tau_1(e_i,e_j)$. However, this notation can be
somewhat misleading, in the following sense. Suppose that
$f_{i},f_{j},f_{p},f_{q}$ are four elements of $G_0$, with
$i,j,p,q\in\{1,\ldots ,r\}$. The equality $(i,j)=(p,q)$ does not
ensure $\tau_1(e_i,e_j)$ is equal to $\tau_1(e_p,e_q)$, because of the
non-uniqueness in the choice of the elements $g_s^{(i,j)}$ (unless
``determinate division with remainder'' is used, as mentioned
above). Nevertheless, the equality $\tau_1(e_i,e_j)=\tau_1(e_p,e_q)$
does ensure that $e_i=e_p$ and that $e_j=e_q$, as seen above.

For ease of reference, we highlight the property:
\begin{eqnarray}\label{inducedlm}
\lt(\tau_1(e_i,e_j))=m^1_{j,i}e_i.
\end{eqnarray}
\end{discussion}

\subsection{The Schreyer algorithm to compute
the syzgygies of a Gr\"obner basis}\label{schreyer-alg}
$\phantom{+}$\bigskip

\noindent We recall the Schreyer algorithm to compute the syzygies of
a Gr\"obner basis (see \cite[Theorem~3.2 and Proposition~3.5]{emss},
\cite[Corollary~2.3.19]{ds}, \cite[Corollary~1.11]{bs},
\cite[Theorem~15.10, Corollary~15.11]{eisenbud}).  We will refer to
\cite{emss} particularly for easy quoting. First we introduce some more
notations.

\begin{notation}\label{module-quotients}
Let $F_0=\mba^s$ be the free $\mba$-module of rank $s$ and let
$G_0=\{f_1,\ldots ,f_r\}$ be a finite subset of
$F_0\setminus\{0\}$. Recall that
$m^1_{j,i}=\Lcm(\lm(f_j),\lm(f_i))/\lt(f_i)$.  Let
$\{M_i(G_0)\}_{1\leq i\leq r}$ be the set of ideals of $\mba=\mbk[x]$
defined as follows: $M_1(G_0)=0$, and for each $i=2,\ldots,r$,
\begin{eqnarray*}
M_i(G_0) & = & (\lt(f_1),\ldots,\lt(f_{i-1})):\lt(f_i)\\ & = &
\langle\lt(f_1)\rangle:\lt(f_i)+\cdots+\langle\lt(f_{i-1})\rangle:\lt(f_i).
\end{eqnarray*}
Note that, for each $j=1,\ldots, i-1$,
\begin{eqnarray*}
\langle \lt(f_j)\rangle: \lt(f_i)& = & \langle
  \lm(f_j)\rangle:\lm(f_i)\\ & = & \langle
  \frac{\Lcm(\lm(f_j),\lm(f_i))}{\lm(f_i)}\rangle= \langle
  \frac{\Lcm(\lm(f_j),\lm(f_i))}{\lt(f_i)}\rangle=(m^1_{j,i}).
\end{eqnarray*}
Therefore, $M_i(G_0)=(m^1_{1,i},\ldots , m^1_{i-1,i})$. Observe that,
in fact, the $M_i(G_0)$ are monomial ideals.
\end{notation}

\begin{theorem}\label{algorithm} {\rm (\cite[Proposition~3.5]{emss})}
Let $F_0=\mba^s$ be the free $\mba$-module of rank $s$ and let
$G_0=\{f_1,\ldots ,f_r\}$ be a finite subset of
$F_0\setminus\{0\}$. Suppose that $G_0$ is a Gr\"obner basis of
$\langle G_0\rangle$ w.r.t. some monomial ordering $>$ on $F_0$. Let
$F_1=\mba^r$ be the free $\mba$-module of rank $r$, the cardinality of
$G_0$, and with canonical basis $e_1,\ldots ,e_r$. Let
$\varphi_1:F_1\rightarrow F_0$ be defined by $\varphi_1(e_i)=f_i$. For
each $i=2,\ldots ,r$, and for each minimal generator $x^{\alpha}$ of
$M_i(G_0)$, choose exactly one index $j=j(i,\alpha)$, with $1\leq
j<i$, such that $m^1_{j,i}$ divides $x^\alpha$, and compute a standard
expression for $S(f_i,f_j)$ with remainder zero:
\begin{eqnarray*}
S(f_i,f_j)=m^1_{j,i}f_i-m^1_{i,j}f_j=g_1^{(i,j)}f_1+\cdots
+g_r^{(i,j)}f_r.
\end{eqnarray*}
Let $\tau_1(e_i,e_j)=m^1_{j,i}e_i-m^1_{i,j}e_j-(g_1^{(i,j)}e_1+\cdots
+g_r^{(i,j)}e_r)\in \ker(\varphi_1)$ be the corresponding
$\tau$-syzygy associated to the $S$-vector of $f_i$ and $f_j$. Then
\begin{eqnarray*}
G_1:=\bigcup_{i=2}^{r}\{\tau_1(e_i,e_j)\mid x^\alpha\mbox{ a minimal
  generator of }M_i(G_0)\mbox{ and }j=j(i,\alpha)\mbox{ with }
m^1_{j,i}\vert x^\alpha\}
\end{eqnarray*}
is a Gr\"obner basis of $\ker(\varphi_1)\subset F_1$ w.r.t. the
monomial ordering on $F_1$ induced by $>$ and $G_0$. In particular,
$\ker(\varphi_1)=\langle G_1\rangle$.
\end{theorem}

\begin{remark}\label{uptosign} 
Fix $i\in \{2,\ldots r\}$. As remarked above,
$M_i(G_0)=(m^1_{j,i},\mid 1\leq j<i)$. By definition, $x^{\alpha}$ is
a minimal generator of $M_i(G_0)$. Hence if we choose an index
$j=j(i,\alpha)$ with $1\leq j<i$ such that the generator $m^1_{j,i}$
of $M_i(G_0)$ divides $x^{\alpha}$, then necessarily $m^1_{j,i}$
equals $x^\alpha$ up to a nonzero scalar. In the situation where the
$m^1_{j,i}$ are monomials up to $\pm$, which will be the case in our
applications of Theorem~\ref{algorithm}, then $m^1_{j,i}=\pm
x^\alpha$.
\end{remark}

\begin{remark}\label{nonminimal}
Suppose that we have a $\tau$-syzygy $\tau_1(e_i,e_j)$ such that
$m^1_{j,i}$ is a non-minimal generator of $M_i(G_0)$ and let $G_1$ be
as in Theorem~\ref{algorithm}. Then 
\begin{eqnarray*}
\tilde{G}_1=G_1\cup \{\tau_1(e_i,e_j)\}
\end{eqnarray*}
is still a Gr\"obner basis of
$\ker(\varphi_1)$. This follows from the definition of Gr\"obner
basis. Indeed, since $G_1$ is a Gr\"obner basis, then
$\lmod(G_1)=\lmod(\ker(\varphi_1))$. Since
$\tau_1(e_i,e_j)\in\ker(\varphi_1)$, then
$\lt(\tau_1(e_i,e_j))\in\lmod(\ker(\varphi_1))$ and
$\lmod(\tilde{G}_1)$ is still equal to $\lmod(\ker(\varphi_1))$.
\end{remark}

\begin{remark}\label{induced-srle}
We will consider in $G_1$ the following enumeration: let
$\tau_1(e_p,e_q)$ and $\tau_1(e_i,e_j)$ be two elements of $G_1$, with
$i,j,p,q\in\{1,\ldots ,r\}$, such that $(p,q)\neq (i,j)$. That is,
$\tau_1(e_p,e_q)\neq\tau_1(e_i,e_j)$ (cf. Discussion~\ref{disc}).  We
will use the notation $\tau_1(e_p,e_q)\prec \tau_1(e_i,e_j)$ and say
that ``$\tau_1(e_p,e_q)$ precedes $\tau_1(e_i,e_j)$'', or
alternatively ``$\tau_1(e_i,e_j)$ succeeds $\tau_1(e_p,e_q)$'',
whenever $p<i$, or whenever $p=i$ and $q<j$.
\end{remark}

\section{Exactness of the Cyc complex}

\subsection{Statement of the main result}
$\phantom{+}$\bigskip

The main result of the paper is the following.

\begin{theorem}\label{main}
Let $\mbg$ be a strongly connected digraph of $n$ vertices or,
equivalently, let $L$ be an $n\times n$ $\icb$ matrix. Let
$\mcc_{\mathbb{G}}$ be the associated $\cyc$ complex. Then
\begin{eqnarray*}
\mcc_{\mathbb{G}}:\; 0\leftarrow \mbk[x]/I(\mcl)\leftarrow
\mcc_0=\mbk[x]\xleftarrow{\sdif_1} \mcc_1=\mbk[x]^{r_1}
\xleftarrow{\sdif_2} \cdots 
\xleftarrow{\sdif_{n-1}}\mcc_{n-1}=\mbk[x]^{r_{n-1}}\xleftarrow{} 0
\end{eqnarray*}
is a free resolution of $\mbk[x]/I(\mcl)$.
\end{theorem}

\begin{remark}
Recall that we assume that $\mbg$ is a finite, weighted, directed
graph, without loops, sources or sinks (cf.
Assumption~\ref{digraph}). In particular, its Laplacian matrix $L$ is
a $\cb$ matrix and, since $\mbg$ is strongly connected, then $L$ is an
$\icb$ matrix (see Subsection~\ref{dictionary}). By
Proposition~\ref{chain-complex}, $\mcc_{\mathbb{G}}$ is a chain
complex of free $\mba$-modules. Moreover, we suppose in $\mba=\mbk[x]$
the $\mbn$-grading given by $\nu(L)$ (cf. Notation~\ref{vectornu} and
Assumption~\ref{grading}). Furthermore, we assume that $\mbg$ has a
$(\omega,\delta)$-enumeration or, equivalently, that $L$ is in
$\delta$-block echelon form (see \ref{echelon-omega}). Finally, we
suppose that $\mcc_0=\mbk[x]$ is endowed with the $\wrlo$
(Assumption~\ref{wrlo}).

Recall also that $\dif_0:\mcc_0=\mbk[x]\to \mbk[x]/I(\mcl)$ is defined
to be the natural projection onto the quotient ring (see
Definition~\ref{dif0}). Observe that we may interpret $\mcc_n=0$ and
$\dif_n=0$.
\end{remark}

Let us define a monomial ordering on each module $\mcc_{k}$ of the
$\cyc$ complex.

\begin{remark}\label{mon-ordering}
Recall that the elements of the enumerated free basis $\mcb_k$ of the
free module $\mcc_k$ are denoted by $e_{k,1},\ldots ,e_{k,r_k}$, where
$r_k=\rank(\mcc_k)=\lvert\cyc_{n,k+1}\rvert$. Their images under
$\dif_k$ are denoted by $f_{k-1,1},\ldots ,f_{k-1,r_k}$, with
$f_{k-1,j}:=\dif_k(e_{k,j})$ (Definition~\ref{cyccomplex}). Let
\begin{eqnarray*}
G_{k-1}:=\dif_k(\mcb_k)= \{\dif_{k}(e_{k,1}),\ldots
,\dif_{k}(e_{k,r_k})\}= \{f_{k-1,1},\ldots
,f_{k-1,r_k}\}\subset\mcc_{k-1}.
\end{eqnarray*}
By Assumption~\ref{wrlo}, we endow $\mcc_0=\mbk[x]$ with the
$\wrlo$. Using Definitions~\ref{secondsetofdefinitions}, where
$G_0=\dif_1(\mcb_1)$, with $\mcb_1$ a basis of $\mcc_1$ and
$\dif_1:\mcc_1\to\mcc_0$, we endow $\mcc_1$ with the monomial ordering
induced by the one in $\mcc_0$ and the subset $G_0$. Taking into
account that $G_{k-1}=\dif_k(\mcb_{k})$, with $\mcb_{k}$ a basis of
$\mcc_k$ and $\dif_k:\mcc_k\to\mcc_{k-1}$, we can proceed recursively,
and endow each $\mcc_k$ with the monomial ordering induced by the one
in $\mcc_{k-1}$ and the subset $G_{k-1}$.
\end{remark}

\begin{notation}\label{m-basis-elements}
With the notations as above in Remark~\ref{mon-ordering}, set
\begin{eqnarray*}
m^k_{j,i}:=\frac{\Lcm(\lm(\dif_{k}(e_{k,j})),\lm(\dif_k(e_{k,i})))}
{\lt(\dif_k(e_{k,i}))}=
\frac{\Lcm(\lm(f_{k-1,j}),\lm(f_{k-1,i}))}{\lt(f_{k-1,i})}.
\end{eqnarray*}
Similarly to Notation~\ref{module-quotients}, $M_i(G_{k-1})$ are the
monomial ideals of $\mba=\mbk[x]$ defined as follows:
$M_1(G_{k-1})=0$, and for each $i=2,\ldots,r_{k-1}$,
\begin{eqnarray*}
&& M_i(G_{k-1})=(\lt(f_{k-1,1}),\ldots,\lt(f_{k-1,i-1})):\lt(f_{k-1,i})=
  \\ && \langle\lt(f_{k-1,1})\rangle:\lt(f_{k-1,i})+\cdots+
  \langle\lt(f_{k-1,i-1})\rangle:\lt(f_{k-1,i})= (m^k_{1,i},\ldots
  ,m^k_{i-1,i}).
\end{eqnarray*}
Note that, if $\lm(f_{k-1,j})$ and $\lm(f_{k-1,i})$ are two monomials
of $\mcc_{k-1}$ which are multiple of different basis elements, then
$\langle\lt(f_{k-1,j})\rangle:\lt(f_{k-1,i})=0$. In such a case, we
understand $m_{j,i}^k=0$ (see Definition~\ref{firstsetofdefinitions})
\end{notation}
 
\begin{notation}\label{hyphk} 
We will prove Theorem~\ref{main} by induction on the degree of the
component of the complex. We specify now the statements to be shown.

\underline{Statement $\mbfh_0$}:
\begin{itemize}
\item[$(a)$] $G_{0}=\dif_{1}(\mcb_{1})$ is a Gr\"obner basis of
  $\ker(\dif_{0})\subset\mcc_{0}$.
\item[$(b)$] $\lt(\dif_{1}(C,\overline{C}))=
  (-1)^0x^{C\to\overline{C}}(C\cup\overline{C})$, for each
  $(C,\overline{C})\in\mcb_{1}$.
\end{itemize}

\underline{Statement $\mbfh_k$}, for $1\leq k\leq n-2$ (recall that
$n\geq 3$):
\begin{itemize}
\item[$(a)$] $G_{k}=\dif_{k+1}(\mcb_{k+1})$ is a Gr\"obner basis of
  $\ker(\dif_{k})\subset\mcc_{k}$ w.r.t. the monomial ordering defined
  as in Remark~\ref{mon-ordering}.
\item[$(b)$] $\lt(\dif_{k+1}(I_1,\ldots
  ,I_{k+2}))=(-1)^{k}x^{I_{k+1}\to I_{k+2}}(I_1,\ldots ,I_{k+1}\cup
  I_{k+2})$, for each $(I_1,\ldots ,I_{k+2})$ in $\mcb_{k+1}$.
\end{itemize}

\underline{Statement $\mbfh_{n-1}$}:
$\dif_{n-1}:\mcc_{n-1}\to\dif_{n-2}$ is injective.
\end{notation}

\begin{remark}\label{degree0}
We already know that $\mbfh_0$ holds. Indeed, by
Corollary~\ref{GBI(L)}, $G_0=\dif_1(\mcb_1)$ is a Gr\"obner basis of
${\rm Im}(\dif_1)=I(\mcl)=\ker(\dif_0)$. Moreover, by Lemma~\ref{lm},
\begin{eqnarray*}
\lt(\dif_1(C,\overline{C}))=\lt(f_C)=m_C=x^{C\to\overline{C}}=
(-1)^0x^{C\to\overline{C}}(C\cup\overline{C}),
\end{eqnarray*}
where, according to Definition~\ref{cyccomplex}, $(C\cup
\overline{C})\equiv ([n])$ is identified with the unit element of
$\mcc_0=\mbk[x]$.
\end{remark}

\begin{remark}\label{exactnessn-1}
If $\mbfh_{n-2}$ holds, then $\mbfh_{n-1}$ holds. Indeed, for each
$i=2,\ldots ,r_{n-1}$, the basis element $e_{n-1,i}=(I_1,\ldots
,I_n)\in\mcb_{n-1}$ is a cyclically ordered partition of $[n]$ into
$n$ blocks, with $n\in I_n$. Thus, $(I_1,\ldots ,I_{n})
=(\{i_1\},\ldots ,\{i_{n-1}\},\{n\})$, where $\{i_1,\ldots
,i_{n-1},n\}=[n]$. Any different basis element
$e_{n-1,j}\in\mcb_{n-1}$, with $1\leq j<i$, will be of the same form:
$e_{n-1,j}=(\{j_1\},\ldots ,\{j_{n-1}\},\{n\})\in\mcb_{n-1}$. By part
$(b)$ of $\mbfh_{n-2}$,
\begin{eqnarray*}
\lt(f_{n-1,i})= \lt(\dif_{n-1}(e_{n-1,i}))=
  (-1)^{n-2}x^{I_{n-1}\to\{n\}}(\{i_1\},\ldots
  ,\{i_{n-2}\},\{i_{n-1},n\})
\end{eqnarray*}
and, similarly,
\begin{eqnarray*}
\lt(f_{n-1,j})=\lt(\dif_{n-1}(e_{n-1,j}))=
(-1)^{n-2}x^{J_{n-1}\to\{n\}}(\{j_1\},\ldots
,\{j_{n-2}\},\{j_{n-1},n\}).
\end{eqnarray*}
One sees that, since $e_{n-1,i}$ and $e_{n-1,j}$ are different, then
$\lt(f_{n-1,i})$ and $\lt(f_{n-1},j)$ are multiple of different basis
elements. Thus
$\langle\lt(f_{n-1,j})\rangle:\lt(f_{n-1,i})=0$ and
$M_i(G_{n-1})=0$, for all $i=1,\ldots ,r_{n-1}$ (see
Notation~\ref{m-basis-elements}). By Theorem~\ref{algorithm}, we
deduce that $\ker(\dif_{n-1})=0$.
\end{remark}

\begin{remark}\label{exactness1ton-2}
If $\mbfh_k$ holds, for every $k=0,\ldots ,n-1$, then
$\mcc_{\mathbb{G}}$ is an exact complex. Indeed, by
Remark~\ref{degree0} above, ${\rm Im}(\dif_1)=\ker(\dif_0)$. For
$1\leq k\leq n-2$, by part $(a)$ of $\mbfh_k$,
\begin{eqnarray*}
{\rm Im}(\dif_{k+1})=\langle
\dif_{k+1}(\mcb_{k+1})\rangle=\langle \mbox{Gr\"obner basis of
}\ker(\dif_k)\rangle =\ker(\dif_{k}).
\end{eqnarray*}
Finally, by definition, $\dif_{n}=0$, so ${\rm Im}(\dif_{n})=0$. On
the other hand, $\dif_{n-1}$ is injective by $\mbfh_{n-1}$.
\end{remark}

\begin{purpose}\label{reduction}
In the light of Remarks~\ref{degree0}, \ref{exactnessn-1} and
\ref{exactness1ton-2}, the proof of the main theorem is reduced to
show ``$\mbfh_{k-1}\Rightarrow \mbfh_{k}$, for $k=1,\ldots ,n-2$''.
\end{purpose}

Before that, we need two important results. These are presented in the next
two sections.

\subsection{Identifying superfluous monomial generators of the 
module quotients}
$\phantom{+}$\bigskip

\noindent Given a basis element $e_{k,i}=(I_1,\ldots
,I_k,I_{k+1})\in\mcb_k$, we want to identify superfluous elements
$m^k_{j,i}$ in $M_{i}(G_{k-1})$, where $1\leq j<i\leq r_k$ (recall
Notation~\ref{m-basis-elements} and Remark~\ref{uptosign}).

\begin{lemma}\label{superfluous}
Fix an integer $k$, with $1\leq k\leq n-2$. Suppose that $\mbfh_{k-1}$
holds. Fix a basis element $e_{k,i}=(I_1,\ldots ,I_k,I_{k+1})$ in
$\mcb_k$. Let $e_{k,j}=(J_1,\ldots ,J_k,J_{k+1})\in\mcb_k$, for some
$1\leq j<i\leq r_k$.
\begin{itemize}
\item[$(a)$] If for some $1\leq s\leq k-1$, $J_s\neq I_s$, then
  $m^k_{j,i}=0$. 
\end{itemize}
Suppose, on the contrary, that $e_{k,j}=(I_1,\ldots
,I_{k-1},J_{k},J_{k+1})$, so that $J_k\cup J_{k+1}=I_{k}\cup I_{k+1}$.
\begin{itemize}
\item[$(b)$] Then, 
\begin{eqnarray*}
m^k_{j,i} =(-1)^{k-1}x^{(J_k\cap I_k\to J_{k+1},I_{k+1})^+}x^{J_k\cap
  I_{k+1}\to J_{k+1}}.
\end{eqnarray*}
 In particular, if $J_k\supsetneq I_k$ (and so $J_{k+1}\subsetneq
 I_{k+1}$), then
\begin{eqnarray}\label{mji}
m^k_{j,i}=(-1)^{k-1}x^{J_{k}\cap I_{k+1}\to J_{k+1}}.
\end{eqnarray}
\item[$(c)$] If $J_k\not\supset I_k$, let $1\leq l<j$ be such that
  $e_{k,l}=(I_1,\ldots ,I_{k-1},J_{k}\cup I_k,J_{k+1}\cap
  I_{k+1})$. Then $m^k_{l,i}\mid m^k_{j,i}$, so $m^k_{j,i}$ is
  superfluous.
\end{itemize}
\end{lemma}
\begin{proof} 
Note that $(a)$ is vacuous if $k=1$. So, suppose momentarily, that
$k\geq 2$ and that $J_s\neq I_s$, for some $1\leq s\leq k-1$.  By the
hypothesis that $\mbfh_{k-1}$ holds, we have that
\begin{eqnarray}\label{lm-j-i}
\lt(\dif_k(e_{k,j}))=(-1)^{k-1}x^{J_k\to J_{k+1}}(J_1,\ldots
,J_{k-1},J_k\cup J_{k+1})
\end{eqnarray}
and
\begin{eqnarray*}
\lt(\dif_k(e_{k,i}))=(-1)^{k-1}x^{I_k\to I_{k+1}}(I_1,\ldots ,I_{k-1},
,I_k\cup I_{k+1}),
\end{eqnarray*}
where $(J_1,\ldots ,J_k\cup J_{k+1})\neq (I_1,\ldots ,I_k\cup
I_{k+1})$. Then $(a)$ follows from the fact that the least common
multiple of two monomial terms is zero provided the two basis elements
involved are different (see Definition~\ref{firstsetofdefinitions},
$(2)$).

Suppose that $e_{k,j}=(I_1,\ldots ,I_{k-1},J_{k},J_{k+1})$, with
$J_k\cup J_{k+1}=I_{k}\cup I_{k+1}$. Note that this means that
$(J_k,J_{k+1})$ and $(I_k,I_{k+1})$ are each partitions of the same
set $J_k\cup J_{k+1}=I_{k}\cup I_{k+1}$. Now the two basis elements
involved in $\lm(\dif_{k}(e_{k,j}))$ and $\lm(\dif_k(e_{k,i}))$ are
the same (see \eqref{lm-j-i} above). Therefore, using
Definition~\ref{firstsetofdefinitions}, $(2)$, again:
\begin{eqnarray*}
m^k_{j,i} & = & \frac{\Lcm(x^{J_k\to J_{k+1}},x^{I_k\to
      I_{k+1}})(I_1,\ldots ,I_{k-1},I_{k}\cup
    I_{k+1})}{(-1)^{k-1}x^{I_k\to I_{k+1}}(I_1,\ldots
    ,I_{k-1},I_{k}\cup I_{k+1})}\\ & = & \frac{\Lcm(x^{J_k\to
      J_{k+1}},x^{I_k\to I_{k+1}})}{(-1)^{k-1}x^{I_k\to I_{k+1}}}.
\end{eqnarray*}
Now, using \eqref{disjoint}, we see that
\begin{eqnarray*}
&&x^{J_k\to J_{k+1}}=x^{J_k\cap I_{k}\to J_{k+1}}x^{J_{k}\cap I_{k+1}\to
  J_{k+1}}\mbox{ and }\\&&x^{I_k\to I_{k+1}}=x^{J_k\cap I_k\to
  I_{k+1}}x^{J_{k+1}\cap I_k\to I_{k+1}}.
\end{eqnarray*}
Thus
\begin{eqnarray*}
&&\Lcm(x^{J_k\to J_{k+1}},x^{I_k\to I_{k+1}})\\&& =\Lcm(x^{J_k\cap
    I_{k}\to J_{k+1}}x^{J_{k}\cap I_{k+1}\to J_{k+1}},x^{J_k\cap
    I_k\to I_{k+1}}x^{J_{k+1}\cap I_k\to I_{k+1}})\\ &&
 =\Lcm(x^{J_k\cap I_{k}\to J_{k+1}},x^{J_k\cap I_k\to I_{k+1}})
  x^{J_{k}\cap I_{k+1}\to J_{k+1}}x^{J_{k+1}\cap I_k\to I_{k+1}}.
\end{eqnarray*}
Recalling \eqref{lcm-equality} in Notation~\ref{lcm-notation}, we
deduce:
\begin{eqnarray*}
\frac{\Lcm(x^{J_k\to J_{k+1}},x^{I_k\to I_{k+1}})}{(-1)^{k-1}x^{I_k\to
    I_{k+1}}}= (-1)^{k-1}x^{(J_k\cap I_k\to J_{k+1},I_{k+1})^+}
x^{J_k\cap I_{k+1}\to J_{k+1}}.
\end{eqnarray*}
In particular, if $J_k\supsetneq I_k$, then $J_{k+1}\subsetneq
I_{k+1}$ and $x^{(J_k\cap I_k\to J_{k+1},I_{k+1})^+}=1$. Therefore
\begin{eqnarray*}
m^k_{j,i}=(-1)^{k-1}x^{J_k\cap I_{k+1}\to J_{k+1}}.
\end{eqnarray*}
This proves $(b)$. Suppose now that $J_k\not\supset I_k$, so
$J_{k+1}\not\subset I_{k+1}$ and $J_{k+1}\cap I_{k+1}\subsetneq
J_{k+1}$. Let $1\leq l<j<i$ be such that $e_{k,l}=(I_1,\ldots
,I_{k-1},J_k\cup I_k,J_{k+1}\cap I_{k+1})$. Using $(b)$, we have that:
\begin{eqnarray*}
&& m^k_{l,i}=(-1)^{k-1}x^{J_k\cap I_{k+1}\to J_{k+1}\cap I_{k+1}}\mbox{
  and }\\ && m^k_{j,i}=(-1)^{k-1}x^{(J_k\cap I_k\to J_{k+1},I_{k+1})^+}
x^{J_k\cap I_{k+1}\to J_{k+1}}.
\end{eqnarray*}
In particular, $m^k_{l,i}$ divides $m^k_{j,i}$, which proves $(c)$.
\end{proof}

The following is an easy remark but will prove to be very useful in
organizing the basis elements of $\mcb_k$.

\begin{remark}\label{bki}
Fix an integer $k$, with $1\leq k\leq n-2$. Let $e_{k,i}=(I_1,\ldots
,I_k,I_{k+1})$ be a basis element of $\mcb_k$, where $1\leq i\leq
r_{k}$. Set
\begin{eqnarray*}
\mcb_{k,i}:=\{(I_1,\ldots ,I_{k-1},J_k,J_{k+1})\in\mcb_k\mid
J_k\supsetneq I_k\}.
\end{eqnarray*}
The following conditions hold.
\begin{itemize}
\item[$(a)$] The set $\mcb_{k,i}$ is empty if and only if
  $I_{k+1}=\{n\}$. (For instance, for the first basis element
\begin{eqnarray*}  
e_{k,1}=(\{k,\ldots ,n-1\},\{k-1\},\ldots ,\{2\},\{1\},\{n\})
\end{eqnarray*} 
of $\mcb_k$, considering the $\srle$, we have $\mcb_{k,1}=\varnothing$.)
\item[$(b)$] The cardinality of $\mcb_{k,i}$ is $\lvert
  \mcb_{k,i}\rvert=2^{\lvert I_{k+1}\rvert -1}-1$. (This formula holds
  trivially if $I_{k+1}=\{n\}$.)
\item[$(c)$] If $\mcb_{k,i}\neq \varnothing$ and $e_{k,j}=(I_1,\ldots
  ,I_{k-1},J_k,J_{k+1})$ in $\mcb_{k,i}$, then $1\leq j<i$. Moreover
  $\mcb_{k,j}\subsetneq\mcb_{k,i}$.
\item[$(d)$] If $\mcb_{k,i}\neq \varnothing$, let
  $\varrho_{k,i}:\mcb_{k,i}\to\mcb_{k+1}$ be defined by
\begin{eqnarray*}
\varrho_{k,i}(e_{k,j}):=(I_1,\ldots ,I_k,J_k\setminus I_k,J_{k+1}),
\end{eqnarray*}
while if $\mcb_{k,i}=\varnothing$, we set $\varrho_{k,i}(\mcb_{k,i}):=\varnothing$.
Then $\varrho_{k,i}$ is well-defined and injective.
\item[$(e)$] If $i\neq j$, then $\varrho_{k,i}(\mcb_{k,i})$ and
  $\varrho_{k,j}(\mcb_{k,j})$ are two disjoint sets.
\item[$(f)$] Then $\mcb_{k+1}=\bigcup_{i=2}^{r_k}\varrho_{k,i}(\mcb_{k,i})$.
In particular, $r_{k+1}=\sum_{i=2}^{r_k}\lvert\mcb_{k,i}\rvert$.
\end{itemize}
\end{remark}
\begin{proof} 
Suppose that $I_{k+1}=\{n\}$ and let $(I_1,\ldots
,I_{k-1},J_k,J_{k+1})\in\mcb_k$, with $J_k\supseteq I_k$. Then
\begin{eqnarray*}
J_k\cup J_{k+1}=[n]\setminus \cup_{s=1}^{k-1}I_s=I_k\cup I_{k+1}=I_k\cup
\{n\}. 
\end{eqnarray*}
Since $n\in J_{k+1}$, it follows that $J_k=I_k$ and so there
does not exist any element in $\mcb_{k,i}$. On the other hand, if
$I_{k+1}\supsetneq \{n\}$ and $m\in I_{k+1}$ with $m\neq n$, just take
$J_{k}=I_{k}\cup \{m\}$ and $J_{k+1}=I_{k+1}\setminus \{m\}$. Then
$(I_1,\ldots ,I_{k-1},J_k,J_{k+1})$ is in $\mcb_{k,i}$. This proves
$(a)$.

Suppose that $\mcb_{k,i}\neq \varnothing$, i.e., $I_{k+1}\supsetneq
\{n\}$. So $I_k\subsetneq [n-1]\setminus \cup_{s=1}^{k-1}I_s$. Then
the elements $(I_1,\ldots ,I_{k-1},J_k,J_{k+1})$ of $\mcb_{k,i}$ are
in one-to-one correspondence with the non-empty subsets $J_k\setminus
I_k$ of the set $[n-1]\setminus \cup_{s=1}^kI_s=I_{k+1}\setminus
\{n\}$. This proves $(b)$.

Let $e_{k,j}=(I_1,\ldots ,I_{k-1},J_k,J_{k+1})\in\mcb_{k,i}$, that is,
$J_k\supsetneq I_k$. Then, according to the $\srle$, $e_{k,j}$
precedes $e_{k,i}$ and so $j<i$. Moreover, if $e_{k,h}=(I_1,\ldots
,I_{k-1},H_k,H_{k+1})$ is in $\mcb_{k,j}$, then $H_k\supsetneq
J_k\supsetneq I_k$, and so $e_{k,h}\in\mcb_{k,i}$. Furthermore,
$e_{k,j}\in\mcb_{k,i}$, but $e_{k,j}\not\in\mcb_{k,j}$. This proves
$(c)$.

Clearly, if $e_{k,j}=(I_1,\ldots ,I_{k-1},J_k,J_{k+1})\in\mcb_{k,i}$,
then $J_k\supsetneq I_k$ and $J_k\setminus I_k\neq \varnothing$. Thus
$\varrho_{k,i}(e_{k,j})$ is indeed in $\mcb_{k+1}$. Take
$e_{k,j}=(I_1,\ldots ,I_{k-1},J_k,J_{k+1})$ and $e_{k,h}=(I_1,\ldots
,I_{k-1},H_k,H_{k+1})$ in $\mcb_{k,i}$ such that
$\varrho_{k,i}(e_{k,j})=\varrho_{k,i}(e_{k,h})$. Then $J_{k+1}=H_{k+1}$
and, so $J_k=(J_{k}\cup J_{k+1})\setminus J_{k+1}= (H_k\cup
H_{k+1})\setminus H_{k+1}=H_k$, and $e_{k,j}=e_{k,h}$, which proves
$(d)$.

Next we prove $(e)$. Let $e_{k,i}=(I_1,\ldots ,I_k,I_{k+1})$ and
$e_{k,j}=(J_1,\ldots ,J_k,J_{k+1})$, with $\mcb_{k,i}\neq \varnothing$
and $\mcb_{k,j}\neq \varnothing$, respectively.  Suppose that
$e_{k,p}=(I_1,\ldots ,I_{k-1},P_k,P_{k+1})$ is in $\mcb_{k,i}$ and
$e_{k,q}=(J_1,\ldots ,J_{k-1},Q_k,Q_{k+1})$ is in
$\mcb_{k,j}$. Suppose also that
$\varrho_{k,i}(e_{k,p})=\varrho_{k,j}(e_{k,q})$. Then $I_s=J_s$, for
$s=1,\ldots ,k$. In particular, $I_{k+1}=[n]\setminus
\cup_{s=1}^kI_s=[n]\setminus\cup_{s=1}^kJ_s=J_{k+1}$. Thus
$e_{k,i}=e_{k,j}$. The rest follows by $(d)$.

Finally, we prove that $(f)$ holds. Consider $(H_1,\ldots
,H_{k+2})\in\mcb_{k+1}$. Let us find two basis elements
$e_{k,i}=(I_1,\ldots ,I_{k+1})\in\mcb_k$ and $e_{k,j}$ in
$\mcb_{k,i}$, so that $e_{k,j}=(I_1,\ldots ,I_{k-1},J_k,J_{k+1})$,
with $J_k\supsetneq I_k$, and such that
\begin{eqnarray*}
(H_1,\ldots ,H_{k+2})=\varrho_{k,i}(e_{k,j})=(I_1,\ldots
  ,I_k,J_k\setminus I_k,J_{k+1}).
\end{eqnarray*}
This forces $I_{s}=H_s$, for $s=1,\ldots ,k$, $J_{k}\setminus
I_k=H_{k+1}$ and $J_{k+1}=H_{k+2}$.  Therefore $I_{k+1}$ must be equal
to $[n]\setminus \cup_{s=1}^kI_s=[n]\setminus \cup_{s=1}^kH_s=H_{k+1}\cup
H_{k+2}$ and $J_k=I_k\cup H_{k+1}=H_k\cup H_{k+1}$.  Thus, both
\begin{eqnarray*}
&& e_{k,i}=(I_1,\ldots ,I_k,I_{k+1})=(H_1,\ldots ,H_k,H_{k+1}\cup
H_{k+2})\in\mcb_k\mbox{ and }\\ && e_{k,j}=(I_1,\ldots
,I_{k-1},J_k,J_{k+1})=(H_1,\ldots ,H_{k-1},H_{k}\cup
H_{k+1},H_{k+2})\in\mcb_{k,i},
\end{eqnarray*}
are uniquely determined by $(H_1,\ldots ,H_{k+2})\in\mcb_{k+1}$.
\end{proof}

Having in mind Notation~\ref{m-basis-elements},
Theorem~\ref{algorithm} and Remark~\ref{nonminimal}, then the
preceding Lemma~\ref{superfluous} and Remark~\ref{bki} can be
summarised as follows. Note that Lemma~\ref{superfluous} discards
elements $m^k_{j,i}$ which are known to be superfluous for
certain. However, it does not ensure that we have discarded all of
them. In fact, we will see that this is precisely characterised by the
digraph $\mbg$ being strongly complete (see
Corollary~\ref{min-resolution}).

\begin{proposition}\label{gbdesiredform}
Fix an integer $k$, with $1\leq k\leq n-2$. Suppose that $\mbfh_{k-1}$
holds. For each $i=2,\ldots ,r_k$, then
\begin{eqnarray*}
\{ m^k_{j,i}\mid e_{k,j}\in\mcb_{k,i}\}
\end{eqnarray*}
is a generating set for the module quotient $M_i(G_{k-1})$. Let
$\tau_k(e_{k,i},e_{k,j})$ stand for a $\tau$-syzygy associated to the
$S$-vector of $f_{k-1,i}=\dif_k(e_{k,i})$ and
$f_{k-1,j}=\dif_k(e_{k,j})$, where $e_{k,j}\in\mcb_{k,i}$. Then there
is a Gr\"obner basis of $\ker(\dif_k)\subset\mcc_k$ (w.r.t. the
monomial ordering on $\mcc_k$ induced by the monomial ordering on
$\mcc_{k-1}$ and the Gr\"obner basis $G_{k-1}$) of the form
\begin{eqnarray}\label{gbdf}
\bigcup_{i=2}^{r_k}\{\tau_{k}(e_{k,i},e_{k,j})\mid
e_{k,j}\in\mcb_{k,i}\}.
\end{eqnarray}
The cardinality of this Gr\"obner basis is
$\sum_{i=2}^{r_k}\lvert\mcb_{k,i}\rvert$, which coincides with
$r_{k+1}$, the rank of $\mcc_{k+1}$.
\end{proposition}
\begin{proof}
Indeed, Lemma~\ref{superfluous} ensures that in order to find a
generating set of $M_i(G_{k-1})$ we only need to consider the
$m^{k}_{j,i}$ with $e_{k,j}\in\mcb_{k,i}$. By Theorem~\ref{algorithm}
and Remark~\ref{nonminimal}, we deduce that indeed there is a Gr\"obner
basis of $\ker(\dif_k)$ of the form \eqref{gbdf}.

By Discussion~\ref{disc}, we know that if
$\tau_{k}(e_{k,i},e_{k,j})=\tau_k(e_{k,p},e_{k,q})$, then
$e_{k,i}=e_{k,p}$ and $e_{k,j}=e_{k,q}$. In particular, the union in
\eqref{gbdf} is of disjoint sets and
\begin{eqnarray*}
\lvert\{\tau_k(e_{k,i},e_{k,j})\mid
e_{k,j}\in\mcb_{k,i}\}\rvert=\lvert\mcb_{k,i}\rvert. 
\end{eqnarray*}
Thus, the cardinality of this Gr\"obner basis is
$\sum_{i=2}^{r_k}\lvert\mcb_{k,i}\rvert$. By Remark~\ref{bki}, $(f)$,
$\sum_{i=2}^{r_k}\lvert\mcb_{k,i}\rvert=r_{k+1}$.
\end{proof}

\begin{remark}\label{min-m}
With the notations above, if $\mbg$ is strongly complete or,
equivalently, if $L$ is a $\pcb$ matrix, then
\begin{eqnarray*}
\vert\mcb_{k,i}\rvert=
\lvert \{ m^k_{j,i}\mid e_{k,j}\in\mcb_{k,i}\}\rvert.
\end{eqnarray*}
For in the strongly complete case, the map $\mcb_{k,i}\to\mcc_{k-1}$
defined by $e_{k,j}\mapsto m^{k}_{j,i}$ is injective. Indeed, if
$e_{k,j}\in\mcb_{k,i}$, then $m^k_{j,i}=(-1)^{k-1}x^{J_{k}\cap
  I_{k+1}\to J_{k+1}}$ (see \eqref{mji} in Lemma~\ref{superfluous}),
where $e_{k,i}=(I_1,\ldots ,I_{k+1})$ and $e_{k,j}=(I_1,\ldots,
I_{k-1},J_k,J_{k+1})$, with $J_k\supsetneq I_k$. Since $a_{i,j}>0$,
for all $i,j$, it follows that $m^k_{j,i}$ determines $J_k\cap
I_{k+1}$. Hence $J_k=(J_k\cap I_k)\cup (J_k\cap I_{k+1})=I_k\cup
(J_k\cap I_{k+1})$, and so $J_{k+1}$ and hence $e_{k,j}$ are uniquely
determined by $m^k_{j,i}$.
\end{remark}

\subsection{The images  of basis elements 
are syzygies associated to {\em S}-vectors} $\phantom{+}$\bigskip

\noindent Let us show that the differential $\dif_{k+1}$ of the $\cyc$
complex maps each basis element of $\mcb_{k+1}$ to a $\tau$-syzygy
associated to an $S$-vector, up to a $\pm$ sign (see
Discussion~\ref{disc} ).

\begin{proposition}\label{tau}
Fix an integer $k$, with $1\leq k\leq n-2$. Suppose that $\mbfh_{k-1}$
holds. Fix a basis element $(I_1,\ldots ,I_{k+2})$ in $\mcb_{k+1}$.
Then
\begin{eqnarray*}
\dif_{k+1}(I_1,\ldots,I_{k+2})=-\tau_{k}(e_{k,i},e_{k,j}),
\end{eqnarray*}
where $e_{k,i}=(I_1,\ldots ,I_{k+1}\cup I_{k+2})$ and
$e_{k,j}=(I_1,\ldots,I_{k}\cup I_{k+1},I_{k+2})$. That is,
$-\dif_{k+1}(I_1,\ldots,I_{k+2})$ is a $\tau$-syzygy associated to the
$S$-vector of $f_{k-1,i}=\dif_{k}(e_{k,i})$ and
$f_{k-1,j}=\dif_{k}(e_{k,j})$.
\end{proposition}
\begin{proof}
Note that, according to the $\srle$, $e_{k,j}$ does indeed precede
$e_{k,i}$ (see Remark~\ref{srle}) and it follows from the definition
that $e_{k,j}\in\mcb_{k,i}$ (see Remark~\ref{bki}). By
Definition~\ref{cyccomplex}, we have:
\begin{eqnarray*}
\dif_{k+1}(I_1,\ldots,I_{k+2}) & = &(-1)^{k}x^{I_{k+1}\to
    I_{k+2}}(I_1,\ldots ,I_{k+1}\cup I_{k+2})\\ && +
  (-1)^{k-1}x^{I_{k}\to I_{k+1}}(I_1,\ldots ,I_{k}\cup
  I_{k+1},I_{k+2})\\ && + \sum_{s=1}^{k-1}(-1)^{s-1}x^{I_s\to
    I_{s+1}}(I_1,\ldots ,I_s\cup I_{s+1},\ldots
  ,I_{k+2})\\ && - x^{I_{k+2}\to I_1}(I_2,\ldots,I_{k+1},I_1\cup I_{k+2}).
\end{eqnarray*}
Set $e_{k,j_s}:=(I_1,\ldots ,I_s\cup I_{s+1},\ldots ,I_{k+2})$, for
some $j_s$, with $1\leq j_1<\cdots< j_s<\cdots<j_{k-1}<j<i$.
Moreover, set $e_{k,l}:=(I_2,\ldots,I_{k+1},I_1\cup I_{k+2})$, for
some $l\in \{1,\ldots ,r_{k}\}$. Then
\begin{eqnarray}\label{difbasisele}
\dif_{k+1}(I_1,\ldots,I_{k+2}) & = & (-1)^{k}x^{I_{k+1}\to
  I_{k+2}}e_{k,i} + (-1)^{k-1}x^{I_{k}\to
  I_{k+1}}e_{k,j}\\ && + \sum_{s=1}^{k-1}(-1)^{s-1}x^{I_s\to
  I_{s+1}}e_{k,j_s}-x^{I_{k+2}\to I_1}e_{k,l}. \notag
\end{eqnarray}
Since $\dif_{k}\circ\dif_{k+1}=0$, and $\dif_{k}(e_{k,u})=f_{k-1,u}$,
for each $u$, we obtain:
\begin{eqnarray}\label{s-vector}
&& (-1)^{k-1}x^{I_{k+1}\to I_{k+2}}f_{k-1,i}-(-1)^{k-1}x^{I_{k}\to
    I_{k+1}}f_{k-1,j}\\ && = \sum_{s=1}^{k-1}(-1)^{s-1}x^{I_s\to
    I_{s+1}}f_{k-1,j_s}-x^{I_{k+2}\to I_1}f_{k-1,l}. \notag
\end{eqnarray}
Let us see that the left hand side of \eqref{s-vector} coincides
with the $S$-vector of $f_{k-1,i}$ and $f_{k-1,j}$. 

Since $\mbfh_{k-1}$ holds, we have
\begin{eqnarray*}
\lt(f_{k-1,j}) & =
&\lt(\dif_{k}(e_{k,j}))=\lt(\dif_{k}(I_1,\ldots,I_{k}\cup
I_{k+1},I_{k+2}))\\ & = & (-1)^{k-1}x^{I_{k}\cup I_{k+1}\to
  I_{k+2}}(I_1,\ldots ,I_{k}\cup I_{k+1}\cup I_{k+2}).
\end{eqnarray*}
Similarly,
\begin{eqnarray*}
\lt(f_{k-1,i}) & = & \lt(\dif_{k}(e_{k,i}))=\lt(\dif_{k}(I_1,\ldots,
I_{k+1}\cup I_{k+2}))\\ & = & (-1)^{k-1}x^{I_{k}\to I_{k+1}\cup I_{k+2}}(I_1,\ldots
,I_{k}\cup I_{k+1}\cup I_{k+2}).
\end{eqnarray*}
Then 
\begin{eqnarray*}
&& \Lcm(\lm(f_{k-1,j}),\lm(f_{k-1,i}))\\ && = \Lcm(x^{I_{k}\cup I_{k+1}\to
  I_{k+2}},x^{I_{k}\to I_{k+1}\cup I_{k+2}})(I_1,\ldots ,I_{k}\cup
I_{k+1}\cup I_{k+2}).
\end{eqnarray*}
Using \eqref{disjoint} in Notation~\ref{lcm-notation},
\begin{eqnarray*}
x^{I_{k}\cup I_{k+1}\to I_{k+2}}=x^{I_{k}\to I_{k+2}}x^{I_{k+1}\to
  I_{k+2}}\mbox{ and } x^{I_{k}\to I_{k+1}\cup I_{k+2}}=x^{I_{k}\to
  I_{k+1}}x^{I_{k}\to I_{k+2}}.
\end{eqnarray*}
Therefore,
\begin{eqnarray*}
\Lcm(x^{I_{k}\cup I_{k+1}\to I_{k+2}},x^{I_{k}\to I_{k+1}\cup
  I_{k+2}})=x^{I_{k}\to I_{k+2}}x^{I_{k+1}\to I_{k+2}}x^{I_{k}\to I_{k+1}}.
\end{eqnarray*}
Hence
\begin{eqnarray*}
m^{k}_{j,i}= \frac{\Lcm(x^{I_{k}\cup I_{k+1}\to
    I_{k+2}},x^{I_{k}\to I_{k+1}\cup
    I_{k+2}})}{(-1)^{k-1}x^{I_{k}\to I_{k+1}\cup
    I_{k+2}}}=(-1)^{k-1}x^{I_{k+1}\to I_{k+2}},
\end{eqnarray*}
and, similarly, $m^{k}_{i,j}=(-1)^{k-1}x^{I_{k}\to
  I_{k+1}}$. Therefore, the $S$-vector of $f_{k-1,i}$ and $f_{k-1,j}$ is
equal to
\begin{eqnarray*}
S(f_{k-1,i},f_{k-1,j}) & = & m^{k}_{j,i}f_{k-1,i}-m^{k}_{i,j}f_{k-1,j} 
\\ & = & (-1)^{k-1}x^{I_{k+1}\to
  I_{k+2}}f_{k-1,i}-(-1)^{k-1}x^{I_{k}\to I_{k+1}}f_{k-1,j},
\end{eqnarray*}
which is the left hand side of \eqref{s-vector}.

Now, let us prove that the right hand side of \eqref{s-vector} is a
standard expression for the $S$-vector $S(f_{k-1,i},f_{k-1,j})$
w.r.t. the Gr\"obner basis $G_{k-1}=\{f_{k-1,1},\ldots
,f_{k-1,r_{k}}\}$, with remainder zero. Note that $G_{k-1}$ is indeed
a Gr\"obner basis of $\ker(\dif_{k-1})$ due to hypothesis
$\mbfh_{k-1}$.  First, observe that the only coincident terms among
the summands of
\begin{eqnarray*}
m^{k}_{j,i}f_{k-1,i}=(-1)^{k-1}x^{I_{k+1}\to I_{k+2}}f_{k-1,i}\;\mbox{ and }\;
m^{k}_{i,j}f_{k-1,j}=(-1)^{k-1}x^{I_{k}\to I_{k+1}}f_{k-1,j}
\end{eqnarray*}
are precisely the leading terms. Indeed: 
\begin{eqnarray}\label{fi}
m^{k}_{j,i}f_{k-1,i} & = & (-1)^{k-1}x^{I_{k+1}\to
  I_{k+2}}f_{k-1,i}=(-1)^{k-1}x^{I_{k+1}\to
  I_{k+2}}\dif_{k}(e_{k,i})\\ & = & (-1)^{k-1}x^{I_{k+1}\to
  I_{k+2}}\dif_{k}(I_1,\ldots ,I_{k+1}\cup
I_{k+2}) \notag \\ & = & (-1)^{k-1}x^{I_{k+1}\to
  I_{k+2}}\sum_{s=1}^{k-1}(-1)^{s-1}x^{I_s\to I_{s+1}}(I_1,\ldots
,I_s\cup I_{s+1},\ldots ,I_{k+1}\cup I_{k+2}) \notag
\\ && + \underline{(-1)^{k-1}x^{I_{k+1}\to I_{k+2}}(-1)^{k-1}x^{I_{k}\to
    I_{k+1}\cup I_{k+2}}(I_1,\ldots ,I_{k}\cup I_{k+1}\cup
  I_{k+2})} \notag \\ && -(-1)^{k-1}x^{I_{k+1}\to I_{k+2}}x^{I_{k+1}\cup I_{k+2}\to
  I_1}(I_2,\ldots,I_1\cup I_{k+1}\cup I_{k+2}), \notag
\end{eqnarray}
whereas 
\begin{eqnarray}\label{fj}
m^{k}_{i,j}f_{k-1,j} & = & (-1)^{k-1}x^{I_{k}\to I_{k+1}}f_{k-1,j}=
  (-1)^{k-1}x^{I_{k}\to I_{k+1}}\dif_{k}(e_{k,j})\\ & = &
  (-1)^{k-1}x^{I_{k}\to I_{k+1}}\dif_{k}(I_1,\ldots ,I_{k}\cup
  I_{k+1},I_{k+2})\notag \\ & = & (-1)^{k-1}x^{I_{k}\to
    I_{k+1}}\sum_{s=1}^{k-2}(-1)^{s-1}x^{I_s\to I_{s+1}}(I_1,\ldots
  ,I_s\cup I_{s+1},\ldots ,I_{k}\cup I_{k+1}, I_{k+2})\notag \\ && +
  (-1)^{k-1}x^{I_{k}\to I_{k+1}}(-1)^{k-2}x^{I_{k-1}\to I_{k}\cup
    I_{k+1}} (I_1,\ldots , I_{k-1}\cup I_{k}\cup I_{k+1}, I_{k+2}) 
  \notag \\ && + \underline{(-1)^{k-1}x^{I_{k}\to
      I_{k+1}}(-1)^{k-1}x^{I_{k}\cup I_{k+1}\to I_{k+2}}(I_1,\ldots
    ,I_{k}\cup I_{k+1}\cup I_{k+2})} \notag \\ &&
  -(-1)^{k-1}x^{I_{k}\to I_{k+1}}x^{I_{k+2}\to
    I_1}(I_2,\ldots,I_{k}\cup I_{k+1},I_1\cup I_{k+2}). \notag
\end{eqnarray}
Excluding the (underlined) coincident leading terms (note that
hypothesis $\mbfh_{k-1}$ holds), the basis elements involved in
\eqref{fi} end either in $I_{k+1}\cup I_{k+2}$ or else in $I_1\cup I_{k+1}\cup
I_{k+2}$, while the basis elements involved in \eqref{fj} end either
in $I_{k+2}$ or else in $I_1\cup I_{k+2}$. Therefore, the $S$-vector
of $f_{k-1,i}$ and $f_{k-1,j}$ cancels the leading terms of these
expressions, but no other summand of $m^{k}_{j,i}f_{k-1,i}$ is
cancelled by any other summand of $m^{k}_{i,j}f_{k-1,j}$.  In
particular,
\begin{eqnarray}\label{geq1}
\lm(S(f_{k-1,i},f_{k-1,j}))\geq
\lm(m^{k}_{j,i}f_{k-1,i}-\lt(m^{k}_{j,i}f_{k-1,i})).
\end{eqnarray}

Now, we need to identify the leading monomials of the summands on the
right hand side of \eqref{s-vector}. For $1\leq s\leq k-1$, since the
hypothesis $\mbfh_{k-1}$ holds,
\begin{eqnarray}\label{lmRHSs}
\lm(x^{I_s\to I_{s+1}}f_{k-1,j_s}) & = & x^{I_s\to
    I_{s+1}}\lm(\dif_{k}(e_{k,j_s}))\\ & = & x^{I_s\to
    I_{s+1}}\lm(\dif_{k}(I_1,\ldots ,I_s\cup I_{s+1},\ldots
  ,I_{k+2})) \notag \\ & = & x^{I_s\to I_{s+1}}x^{I_{k+1}\to I_{k+2}}
  (I_1,\ldots ,I_s\cup I_{s+1},\ldots ,I_{k+1}\cup I_{k+2}). \notag
\end{eqnarray}
Similarly, the leading monomial of the last summand on the right hand
side of \eqref{s-vector} is:
\begin{eqnarray}\label{lmRHSl}
\lm(x^{I_{k+2}\to I_{1}}f_{k-1,l}) & = & x^{I_{k+2}\to
    I_{1}}\lm(\dif_{k}(e_{k,l}))\\ & = & x^{I_{k+2}\to
    I_{l}}\lm(\dif_{k}(I_2,\ldots ,I_{k+1}, I_{1}\cup I_{k+2}))
  \notag \\ & = & x^{I_{k+2}\to I_{1}}x^{I_{k+1}\to I_1\cup I_{k+2}}
  (I_2,\ldots ,I_1\cup I_{k+1}\cup I_{k+2}). \notag
\end{eqnarray}
All the monomials terms in \eqref{lmRHSs} and in \eqref{lmRHSl} appear
in \eqref{fi} and are different from the underlined leading term
$\lt(m^{k}_{j,i}f_{k-1,i})$. (By \eqref{disjoint}, $x^{I_{k+2}\to
  I_1}x^{I_{k+1}\to I_1\cup I_{k+2}}=x^{I_{k+1}\to
  I_{k+2}}x^{I_{k+1}\cup I_{k+2}\to I_1}$.)  Thus
\begin{eqnarray}\label{geq2}
&& \lm(m^{k}_{j,i}f_{k-1,i}-\lt(m^{k}_{j,i}f_{k-1,i}))\geq \\ &&
\lm(x^{I_s\to I_{s+1}}f_{k-1,j_s}),\lm(x^{I_{k+2}\to I_{1}}f_{k-1,l}). \notag
\end{eqnarray}
Concatenating \eqref{geq1} and \eqref{geq2}, we deduce that the right
hand side of \eqref{s-vector} is indeed a standard expression for
$S(f_{k-1,i},f_{k-1,j})$ with remainder zero and w.r.t. the Gr\"obner
basis $G_{k-1}$.

Hence, the corresponding $\tau$-syzygy $\tau_{k}(e_{k,i},e_{k,j})$
associated to the $S$-vector of the elements
$f_{k-1,i}=\dif_{k}(e_{k,i})$ and $f_{k-1,j}=\dif_{k}(e_{k,j})$ is
\begin{eqnarray}\label{tau-vector}
\tau_{k}(e_{k,i},e_{k,j}) & = & (-1)^{k-1}x^{I_{k+1}\to
  I_{k+2}}e_{k,i} - (-1)^{k-1}x^{I_{k}\to
  I_{k+1}}e_{k,j}\\ &&  -\sum_{s=1}^{k-1}(-1)^{s-1}x^{I_s\to
  I_{s+1}}e_{k,j_s}+x^{I_{k+2}\to I_1}e_{k,l}. \notag
\end{eqnarray}
Comparing \eqref{tau-vector} with \eqref{difbasisele}, we get
$\dif_{k+1}(I_1,\ldots,I_{k+2})=-\tau_{k}(e_{k,i},e_{k,j})$, the
desired equality. 
\end{proof}

As an immediate consequence of Proposition~\ref{tau}, we get the
following result. The second part can be seen as a natural
generalisation of Lemma~\ref{fc=fd}.

\begin{corollary}\label{ltbij}
Fix an integer $k$, with $1\leq k\leq n-2$. Suppose that $\mbfh_{k-1}$
holds. Then 
\begin{itemize}
\item[$(a)$]
  $\lt(\dif_{k+1}(I_1,\ldots,I_{k+1},I_{k+2}))=(-1)^{k}x^{I_{k+1}\to
  I_{k+2}}(I_1,\ldots, I_{k},I_{k+1}\cup I_{k+2})$.
\item[$(b)$] Moreover, the restriction of the map
  $\dif_{k+1}:\mcc_{k+1}\to\mcc_{k}$ over the basis set $\mcb_{k+1}$
  is injective. In particular, since $G_k:=\dif_{k+1}(\mcb_{k+1})$,
  then 
\begin{eqnarray*}
\lvert
G_k\rvert=\lvert\dif_{k+1}(\mcb_{k+1})\rvert=
\lvert\mcb_{k+1}\rvert=r_{k+1}.
\end{eqnarray*}
\end{itemize}
\end{corollary}
\begin{proof}
By Proposition~\ref{tau}, and with the notations used there, and by
\eqref{inducedlm} in Discussion~\ref{disc},
\begin{eqnarray*}
\lt(\dif_{k+1}(I_1,\ldots,I_{k+2})) & = & -\lt(\tau_{k}(e_{k,i},e_{k,j}))
=-m^{k}_{j,i}e_{k,i} \\ & = & (-1)(-1)^{k-1}x^{I_{k+1}\to I_{k+2}}(I_1,\ldots,
I_{k+1}\cup I_{k+2})\\ & = & (-1)^{k}x^{I_{k+1}\to I_{k+2}}(I_1,\ldots,
I_{k+1}\cup I_{k+2}),
\end{eqnarray*}
which proves $(a)$.

Suppose that $\dif_{k+1}(I_1,\ldots ,I_{k+2})=\dif_{k+1}(J_1,\ldots
,J_{k+2})$. Hence their leading terms coincide. Using the first part
$(a)$, it follows that
\begin{eqnarray*}
&& (-1)^{k}x^{I_{k+1}\to I_{k+2}}(I_1,\ldots,I_{k},I_{k+1}\cup
  I_{k+2})= (-1)^{k}x^{J_{k+1}\to
    J_{k+2}}(J_1,\ldots,J_{k},J_{k+1}\cup J_{k+2}).
\end{eqnarray*}
Therefore, $I_s=J_s$, for every $1\leq s\leq k$, and moreover,
$I_{k+1}\cup I_{k+2}=J_{k+1}\cup J_{k+2}$. Consider the summand
\begin{eqnarray*}
(-1)^{k-1}x^{I_{k}\to I_{k+1}}(I_1,\ldots ,I_{k-1},I_{k}\cup
  I_{k+1},I_{k+2})
\end{eqnarray*} 
of $\dif_{k+1}(I_1,\ldots ,I_{k+2})$, preceding the leading term and
different from 
\begin{eqnarray*}
-x^{I_{k+2}\to I_1}(I_2,\ldots ,I_{k+1},I_1\cup I_{k+2})
\end{eqnarray*}
(see \eqref{differential} in Definition~\ref{cyccomplex} and
Remark~\ref{dif-and-enum}). On the one hand, this element coincides
with
\begin{eqnarray*}
(-1)^{k-1}x^{I_{k}\to I_{k+1}}(J_1,\ldots ,J_{k-1},I_{k}\cup
  I_{k+1},I_{k+2}). 
\end{eqnarray*}
On the other hand, this element must coincide with a summand of
$\dif_{k+1}(J_1,\ldots ,J_{k+2})$. It follows that either $I_{k}\cup
I_{k+1}=J_{k}\cup J_{k+1}$ and $I_{k+2}=J_{k+2}$, or else, $I_{k}\cup
I_{k+1}=J_{k}$ and $I_{k+2}=J_{k+1}\cup J_{k+2}$. However, the second
case cannot occur, because, since $I_{k}=J_{k}$, it would follow that
$I_{k+1}=\varnothing$, a contradiction. Thus the first case holds, which
clearly implies $I_{k+1}=J_{k+1}$ and $I_{k+2}=J_{k+2}$.
\end{proof}

\subsection{Proof of the main theorem}
$\phantom{+}$\bigskip

\noindent Now, we have all the ingredients to prove
Theorem~\ref{main}.

\begin{proof} Acording to Purpose~\ref{reduction},
we have to show ``$\mbfh_{k-1}\Rightarrow \mbfh_{k}$, for $k=1,\ldots
,n-2$''. So fix an integer $k$, with $1\leq k\leq n-2$ and suppose
that $\mbfh_{k-1}$ holds. Let us prove that $\mbfh_k$ holds. By
Proposition~\ref{gbdesiredform}, we know that there is a Gr\"obner
basis $\tilde{G}_k$, say, of $\ker(\dif_{k})$ of the form: 
\begin{eqnarray}\label{groebf}
\tilde{G}_k:=\bigcup_{i=2}^{r_{k}}\{ \tau_{k}(e_{k,i},e_{k,j})\mid
e_{k,j}\in\mcb_{k,i}\}.
\end{eqnarray}
This is a Gr\"obner basis of $\ker(\dif_{k})$ w.r.t. the monomial
ordering on $\mcc_k$ induced by the monomial ordering on $\mcc_{k-1}$
and the Gr\"obner basis $G_{k-1}$ (see also
Remark~\ref{mon-ordering}). Note that Proposition~\ref{gbdesiredform}
ensures that the cardinality of $\tilde{G}_k$ is $r_{k+1}$.

In Proposition~\ref{tau}, we have shown that, for any $(I_1,\ldots
,I_{k+2})\in\mcb_{k+1}$,
\begin{eqnarray*}
\dif_{k+1}(I_1,\ldots,I_{k+2})=-\tau_{k}(e_{k,i},e_{k,j}),
\end{eqnarray*}
where $e_{k,i}=(I_1,\ldots ,I_{k+1}\cup I_{k+2})$ and
$e_{k,j}=(I_1,\ldots,I_{k}\cup I_{k+1},I_{k+2})$. That is,
$-\dif_{k+1}(I_1,\ldots,I_{k+2})$ is a $\tau$-syzygy associated to the
$S$-vector of $f_{k-1,i}=\dif_{k}(e_{k,i})$ and
$f_{k-1,j}=\dif_{k}(e_{k,j})$. In other words, any element of
$-G_{k}=-\dif_{k+1}(\mcb_{k+1})$ is an element of
$\tilde{G}_k$. Observe that, by Corollary~\ref{ltbij}, $(b)$, the
cardinality of $G_k$ is also $r_{k+1}$. Since the set $-G_{k}$ is
included in the Gr\"obner basis $\tilde{G}_k$ of $\ker(\dif_k)$, and
both $G_k$ and $\tilde{G}_k$ have the same cardinality, it follows
that $G_k=\dif_{k+1}(\mcb_{k+1})$ is a Gr\"obner basis of
$\ker(\dif_{k})$.

(To avoid the preceding discussion using cardinalities, we can argue
as follows. Take $\tau_{k}(e_{k,i},e_{k,j})$, any element in
$\tilde{G}_k$: that is, a $\tau$-syzygy associated to the $S$-vector
of $f_{k-1,i}=\dif_k(e_{k,i})$ and $f_{k-1,j}=\dif_{k}(e_{k,j})$, with
$e_{k,j}\in\mcb_{k,i}$. Remark that, by Discussion~\ref{disc}, both
basis elements $e_{k,i}$ and $e_{k,j}$ are uniquely determined by the
syzygy $\tau_k(e_{k,i},e_{k,j})$. Write $e_{k,i}=(I_1,\ldots
,I_{k},I_{k+1})$ and $e_{k,j}=(I_1,\ldots ,I_{k-1},J_{k},J_{k+1})$,
with $J_{k}\supsetneq I_{k}$. Consider the map
$\varrho_{k,i}:\mcb_{k,i}\to \mcb_{k+1}$ introduced in
Remark~\ref{bki}. Take the element
\begin{eqnarray*}
\varrho_{k,i}(e_{k,j})=(I_1,\ldots,I_{k},J_{k}\setminus
I_{k},J_{k+1})=:(H_1,\ldots,H_k,H_{k+1},H_{k+2})\in\mcb_{k+1}.
\end{eqnarray*}
Observe that $H_s=I_s$, for all $s=1,\ldots ,k$, that
$H_{k+1}=J_{k}\setminus I_{k}$ and that $H_{k+2}=J_{k+1}$. Thus
$H_{k+1}\cup H_{k+2}=(J_{k}\setminus I_{k})\cup J_{k+1}=I_{k+1}$ and
$H_{k}\cup H_{k+1}=I_{k}\cup (J_{k}\setminus I_{k})=J_{k}$. Hence
\begin{eqnarray*}
&& (H_1,\ldots ,H_{k},H_{k+1}\cup H_{k+2})=(I_1,\ldots
  ,I_{k},I_{k+1})=e_{k,i}\mbox{ and }\\ && (H_1,\ldots
  ,H_{k-1},H_{k}\cup H_{k+1},H_{k+2})=(I_1,\ldots
  ,I_{k-1},J_{k},J_{k+1})=e_{k,j}.
\end{eqnarray*}
Then, by Proposition~\ref{tau} again,
$\dif_{k+1}(\varrho_{k,i}(e_{k,j}))=-\tau_{k}(e_{k,i},e_{k,j})$. That
is, any element of $\tilde{G}_k$ is an element of
$-G_{k}=-\dif_{k+1}(\mcb_{k+1})$. In conclusion,
$G_{k}=\dif_{k+1}(\mcb_{k+1})$ is a Gr\"obner basis of
$\ker(\dif_{k})$.)

Since $\mbfh_{k-1}$ holds, it follows that part $(a)$ of
$\mbfh_k$ also holds (recall Notation~\ref{hyphk}). Part $(b)$ of
$\mbfh_k$ follows from Corollary~\ref{ltbij}, $(a)$, and the
hypothesis that $\mbfh_{k-1}$ holds. This completes the proof.
\end{proof}

We characterize now when $\mcc_{\mathbb{G}}$ is a minimal resolution
of $\mbk[x]/I(\mcl)$. This result generalises the first main result in
\cite{ms}, namely Theorem 2 there. In a follow-up paper, we shall
investigate applications of our results, beginning with
generalisations of the remaining results of \cite{ms}, before
branching out more widely.

\begin{corollary}\label{min-resolution}
Let $\mbg$ be a strongly connected digraph or, equivalently, let $L$
be an $\icb$ matrix. Let $\mcc_{\mathbb{G}}$ be its associated $\cyc$
complex. Then
\begin{eqnarray*}
\mcc_{\mathbb{G}}:\phantom{+} 0\leftarrow \mbk[x]/I(\mcl)\leftarrow
\mcc_0=\mbk[x]\xleftarrow{\sdif_1} \mcc_1=\mbk[x]^{r_1}
\xleftarrow{\sdif_2} \cdots 
\xleftarrow{\sdif_{n-1}}\mcc_{n-1}=\mbk[x]^{r_{n-1}}\xleftarrow{} 0,
\end{eqnarray*}
is a minimal free resolution of $\mbk[x]/I(\mcl)$ if and only if
$\mbg$ is strongly complete.
\end{corollary}
\begin{proof}
Suppose that $\mbg$ is strongly complete. That $\mcc_{\mathbb{G}}$ is
a minimal free resolution of $\mbk[x]/I(\mcl)$ follows from
Theorem~\ref{main} and Remark~\ref{dif-m}.

Conversely, suppose that $\mcc_{\mathbb{G}}$ is a minimal free
resolution of $\mbk[x]/I(\mcl)$. Consider $i,j\in [n-1]$ with $i\neq
j$, together with a basis element of the form 
\begin{eqnarray*}
(\{i\},\{j\},I_3,\ldots ,I_{k+1})\in\mcb_k.
\end{eqnarray*}
Then $\pm x^{a_{i,j}}_i$ appears as a coefficient in the expression
for $\dif_k(\{i\},\{j\},I_3,\ldots ,I_{k+1})$. By hypothesis,
$x^{a_{i,j}}_i\in\mfm$, and it follows that $a_{i,j}>0$. Similarly,
$a_{j,i}>0$.

Next consider a basis element of the form $(I_1,\ldots
,I_{k-1},\{i\},\{n\})\in\mcb_k$. A similar argument shows that
$a_{i,n}>0$.

Finally, consider a basis element of the form $(\{i\},I_2,\ldots
,I_k,\{n\})\in\mcb_k$. Again, a similar argument shows that
$a_{n,i}>0$, and the result follows.
\end{proof}

\subsection{More on enumerations}
$\phantom{+}$\bigskip

\noindent Since $\dif_{k+1}:\mcb_{k+1}\to\mcc_{k}$ is injective, it is
natural to consider in $\dif_{k+1}(\mcb_{k+1})$ the enumeration
inherited by the $\srle$ enumeration of $\mcb_{k+1}$ (see
Corollary~\ref{ltbij}, $(b)$). This is not a minor matter as regards
this paper. Indeed, to choose a ``correct'' enumeration in $\mcb_{k}$,
and hence in $G_{k-1}=\dif_{k}(\mcb_{k})$, has been seen to be crucial
in the proof of Lemma~\ref{superfluous} and hence in the proof of the
main theorem, Theorem~\ref{main}. In the light of this comment, the
next remark (and Remark~\ref{nonsrle} below) may help to clarify
ideas.

\begin{remark}\label{enum}
Let us consider in $\dif_{k+1}(\mcb_{k+1})$ the enumeration inherited
by $\mcb_{k+1}$. On the other hand, let us also consider in
$\Upsilon:=\bigcup_{i=2}^{r_k}\{\tau_{k}(e_{k,i},e_{k,j})\mid
e_{k,j}\in\mcb_{k,i}\}$ the enumeration defined in
Remark~\ref{induced-srle} (where $\Upsilon$ is possibly larger than
the set of syzygies considered in Remark~\ref{induced-srle}). Then
both enumerations agree, considering that
$-\dif_{k+1}(\mcb_{k+1})=\Upsilon$ (see the proof of
Theorem~\ref{main}).
\end{remark}
\begin{proof}
Let $e_{k+1,l}=(I_1,\ldots ,I_{k+2})\in\mcb_{k+1}$ and
$\dif_{k+1}(e_{k+1,l})= -\tau_{k}(e_{k,i},e_{k,j})$, with
\begin{eqnarray*}
e_{k,i}=(I_1,\ldots ,I_{k},I_{k+1}\cup I_{k+2})\mbox{ and }
e_{k,j}=(I_1,\ldots ,I_{k-1},I_k\cup I_{k+1},I_{k+2}).
\end{eqnarray*}
Analogously, let $e_{k+1,h}=(H_1,\ldots ,H_{k+2})\in\mcb_{k+1}$ and
$\dif_{k+1}(e_{k+1,h})=-\tau_{k}(e_{k,p},e_{k,q})$, with
\begin{eqnarray*}
e_{k,p}=(H_1,\ldots ,H_k,H_{k+1}\cup H_{k+2})\mbox{ and
}e_{k,q}=(H_1,\ldots H_{k-1},H_k\cup H_{k+1},H_{k+2}).
\end{eqnarray*}
(See Proposition~\ref{tau}.) Suppose that $\dif_{k+1}(e_{k+1,h})$
precedes $\dif_{k+1}(e_{k+1,l})$ in the inherited enumeration by
$\mcb_{k+1}$, that is, $e_{k+1,h}$ precedes $e_{k+1,l}$. By
Remark~\ref{srle}, there exists an integer $s\in\{1,\ldots ,k+1\}$,
such that $H_1=I_1,\ldots ,H_{s-1}=I_{s-1}$ and $H_s$ precedes
$I_s$. If $s\leq k$, then $e_{k,p}$ precedes $e_{k,i}$. If $s=k+1$,
then $e_{k,p}=e_{k,i}$ and $e_{k,q}$ precedes $e_{k,j}$. Therefore,
$\tau_k(e_{k,p},e_{k,q})$ precedes $\tau_k(e_{k,i},e_{k,j})$,
according to the enumeration considered in Remark~\ref{induced-srle}.

Conversely, suppose that $\tau_k(e_{k,p},e_{k,q})$ precedes
$\tau_k(e_{k,i},e_{k,j})$. If $e_{k,p}$ precedes $e_{k,i}$, then there
exists an integer $s\in\{1,\ldots ,k\}$, such that
$H_1=I_1,\ldots,H_{s-1}=I_{s-1}$ and $H_{s}$ precedes $I_{s}$. This
clearly implies that $e_{k+1,h}$ precedes $e_{k+1,l}$. On the other
hand, if $e_{k,p}=e_{k,i}$ and $e_{k,q}$ precedes $e_{k,p}$, then,
since $H_1=I_1,\ldots,H_{k}=I_{k}$, necessarily $H_k\cup H_{k+1}$
precedes $I_k\cup I_{k+1}$. But, since $H_k=I_k$, it follows that
$H_{k+1}$ precedes $I_{k+1}$ and so $e_{k+1,h}$ precedes $e_{k+1,l}$,
and $\dif_{k+1}(e_{k+1,h})$ precedes $\dif_{k+1}(e_{k+1,l})$.
\end{proof}

\begin{remark}\label{nonsrle}
If we attempt to prove Theorem~\ref{main} by means of
Theorem~\ref{algorithm}, only now using an alternative but seemingly
natural enumeration of the basis elements of the successive free
modules in the resolution, we can find that the module quotients
$M_{i}(G_{k-1})$ have an `excessive' number of generators, and as a
result, the free resolution we obtain is no longer the $\cyc$ complex.

For example, consider the $n\times n$ $\pcb$ matrix $L$ considered in
\cite[Example~1, 3]{ms}:
\begin{eqnarray*}
L=\left(\begin{array}{rrrr}
3&-1&-1&-1\\
-1&3&-1&-1\\
-1&-1&3&-1\\
-1&-1&-1&3\\
\end{array}\right).
\end{eqnarray*}
Here $\nu(L)=(1,1,1,1)$. Thus each variable $x,y,z,t$ is given degree
$1$ and we endow $A=\mbk[x,y,z,t]$ with the (weighted) reverse
lexicographic ordering (see Assumptions~\ref{grading} and
\ref{wrlo}). Now, instead of enumerating according to the $\srle$, we
enumerate $\cyc_{4,2}$ as follows (see Example~\ref{srlen=4}):
\begin{eqnarray*}
\cyc_{4,2}&=&\{(1,234),(2,134),(3,124),(123,4),(12,34),(13,24),(23,14)\}.
\end{eqnarray*}
Note that in Corollary~\ref{GBI(L)}, the enumeration considered in
$\mcb_1$ is not essential for its proof. Thus it follows that
$G_0=\{f_{0,1},f_{0,2},f_{0,3},f_{0,4},f_{0,5},f_{0,5},f_{0,6},f_{0,7}\}$
is a Gr\"obner basis of $I(\mcl)$, where
\begin{eqnarray*}
&&f_{0,1}=\dif_1(e_{1,1})=\dif_1(1,234)=\underline{x^3}-yzt,\\
&&f_{0,2}=\dif_1(e_{1,2})=\dif_1(2,134)=\underline{y^3}-xzt,\\
&&f_{0,3}=\dif_1(e_{1,3})=\dif_1(3,124)=\underline{z^3}-xyt,\\ 
&&f_{0,4}=\dif_1(e_{1,4})=\dif_1(123,4)=\underline{xyz}-t^3,\\
&&f_{0,5}=\dif_1(e_{1,5})=\dif_1(12,34)=\underline{x^2y^2}-z^2t^2,\\ 
&&f_{0,6}=\dif_1(e_{1,6})=\dif_1(13,24)=\underline{x^2z^2}-y^2t^2\mbox{ and }\\
&&f_{0,7}=\dif_1(e_{1,7})=\dif_1(23,14)=\underline{y^2z^2}-x^2t^2.
\end{eqnarray*}
It is a simple check to see that ($M_1(G_0)=0$):
\begin{eqnarray*}
&&M_2(G_0)=(x^3)\mbox{ , }
M_3(G_0)=(x^3,y^3)\mbox{ , }
M_4(G_0)=(x^2,y^2,z^2)\mbox{ , }\\
&&M_5(G_0)=M_6(G_0)=M_7(G_0)=(x,y,z)=\mfm.
\end{eqnarray*}
That gives $15$ minimal generators $x^\alpha$ of the ideal quotients
$M_i(G_0)$, for $i=2,\ldots ,7$, in Theorem~\ref{algorithm}.  Thus the
Gr\"obner basis $G_1$ of the syzygies of $G_0$ would contain $15$
elements and the corresponding resolution would begin with
\begin{eqnarray*}
0\leftarrow \mbk[x]/I(\mcl)\leftarrow
\mbk[x]\leftarrow\mbk[x]^{7}\leftarrow \mbk[x]^{15}\leftarrow\cdots ,
\end{eqnarray*}
and so it would not coincide with the complex $\mcc_{L}$.
\end{remark}

\section{An example: the four-dimensional case}

In this section we write down the case $n=4$ just as an illustrative
example. This will enable the reader to visualise the details of the
proofs of our main results for this particular case.

Recall that, when $n=4$, the complex $\mcc$ has the following form
(cf. Example~\ref{cyc4} for more details).
\begin{eqnarray*}
\mcc&:&0\leftarrow
\mcc_0=\mbk[x]\xleftarrow{\sdif_{1}}\mcc_{1}=\mbk[x]^{7}
\xleftarrow{\sdif_{2}}\mcc_{2}=\mbk[x]^{12}\xleftarrow{\sdif_{3}}
\mcc_{3}=\mbk[x]^{6}\xleftarrow{} 0,
\end{eqnarray*}
where we write $\mbk[x]\equiv \mbk[x,y,z,t]$.  Moreover,
Theorem~\ref{main} says that this is a free resolution of
$\mbk[x]/I(\mcl)$. By Corollary~\ref{GBI(L)},
$G_0:=\dif_1(\mcb_1)\subset\mcc_0$ is a Gr\"obner basis of
$\dif_1(\mcc_1)=I(\mcl)=\ker(\dif_0)$. Using
Definition~\ref{cyccomplex} (see also Example~\ref{cyc4}), on
enumerating according to the $\srle$ and underlining leading terms,
the elements of $G_0$ are:
\begin{eqnarray*}
\begin{array}{l}
f_{0,1}=\dif_1(e_{1,1})=
\dif_1(123,4)=\underline{x^{a_{1,4}}y^{a_{2,4}}z^{a_{3,4}}}-t^{a_{4,4}},\\
f_{0,2}=\dif_1(e_{1,2})=
\dif_1(23,14)=\underline{y^{a_{2,1}+a_{2,4}}z^{a_{3,1}+a_{3,4}}}-
x^{a_{1,2}+a_{1,3}}t^{a_{4,2}+a_{4,3}},\\
f_{0,3}=\dif_1(e_{1,3})=
\dif_1(13,24)=\underline{x^{a_{1,2}+a_{1,4}}z^{a_{3,2}+a_{3,4}}}-
y^{a_{2,1}+a_{2,3}}t^{a_{4,1}+a_{4,3}},\\
f_{0,4}=\dif_1(e_{1,4})=
\dif_1(12,34)=\underline{x^{a_{1,3}+a_{1,4}}y^{a_{2,3}+a_{2,4}}}-
z^{a_{3,1}+a_{3,2}}t^{a_{4,1}+a_{4,2}},\\
f_{0,5}=\dif_1(e_{1,5})=
\dif_1(3,124)=\underline{z^{a_{3,3}}}-x^{a_{1,3}}y^{a_{2,3}}t^{a_{4,3}},\\
f_{0,6}=\dif_1(e_{1,6})=
\dif_1(2,134)=\underline{y^{a_{2,2}}}-x^{a_{1,2}}z^{a_{3,2}}t^{a_{4,2}},\\
f_{0,7}=\dif_1(e_{1,7})=
\dif_1(1,234)=\underline{x^{a_{1,1}}}-y^{a_{2,1}}z^{a_{3,1}}t^{a_{4,1}}.
\end{array}
\end{eqnarray*}
Let us calculate the ideal quotients $M_i(G_0)=(m^1_{1,i},\ldots
,m^1_{i-1,i})$ of Notation~\ref{m-basis-elements}.  To do that, we
need to calculate
$m^1_{j,i}=\Lcm(\lm(f_{0,j}),\lm(f_{0,i}))/\lt(f_{0,i})$, for $1\leq
j<i\leq 7$. By definition $M_1(G_0)=0$. The subsequent $M_2(G_0)$ is
equal to $(m^1_{1,2})=(x^{a_{1,4}})$. To calculate
$M_3(G_0)=(m^1_{1,3},m^1_{2,3})$, observe that
$m^1_{2,3}=y^{a_{2,1}+a_{2,4}}z^{(a_{3,1}-a_{3,2})^+}$ is a multiple
of $m^1_{1,3}=y^{a_{2,4}}$, so it follows that
$M_3(G_0)=(y^{a_{2,4}})$. Note that this is consistent with
Lemma~\ref{superfluous} and Proposition~\ref{gbdesiredform}. Indeed,
$e_{1,3}=(13,14)$ and
\begin{eqnarray*}
\mcb_{1,3}=\{(J_1,J_2)\in\mcb_1\mid J_1\supsetneq \{1,3\}\}.
\end{eqnarray*}
Since $\{1,2,3\}\supsetneq \{1,3\}$, then
$e_{1,1}=(123,4)\in\mcb_{1,3}$ and, since $\{2,3\}\not\supset
\{1,3\}$, then $e_{1,2}\not\in\mcb_{1,3}$. In other words, $m^1_{2,3}$
is superfluous. The next $m^1_{j,i}$ are presented in the following
tables:
\begin{eqnarray*}
\begin{array}{|c|c|c|}\hline 
m^1_{j,i}&i=4&i=5\\\hline
j=1&z^{a_{3,4}}&x^{a_{1,4}}y^{a_{2,4}}\\\hline
j=2&y^{(a_{2,1}-a_{2,3})^+}z^{a_{3,1}+a_{3,4}}&
y^{a_{2,1}+a_{2,4}}\\\hline
j=3&x^{(a_{1,2}-a_{1,3})^+}z^{a_{3,2}+a_{3,4}}&x^{a_{1,2}+a_{1,4}}\\\hline
j=4&&x^{a_{1,3}+a_{1,4}}y^{a_{2,3}+a_{2,4}}\\\hline
j=5&&\\\hline 
j=6&&\\\hline
\end{array}
\end{eqnarray*}
\begin{eqnarray*}
\begin{array}{|c|c|c|c|c|}\hline 
m^1_{j,i}&i=6&i=7\\\hline
j=1&x^{a_{1,4}}z^{a_{3,4}}&y^{a_{2,4}}z^{a_{3,4}}\\\hline
j=2&z^{a_{3,1}+a_{3,4}}&y^{a_{2,1}+a_{2,4}}z^{a_{3,1}+a_{3,4}}\\\hline
j=3&x^{a_{1,2}+a_{1,4}}z^{a_{3,2}+a_{3,4}}&z^{a_{3,2}+a_{3,4}}\\\hline
j=4&x^{a_{1,3}+a_{1,4}}&y^{a_{2,3}+a_{2,4}}\\\hline
j=5&z^{a_{3,3}}&z^{a_{3,3}}\\\hline 
j=6&&y^{a_{2,2}}\\\hline
\end{array}
\end{eqnarray*}$\phantom{+}$
Similarly, $M_4(G_0)=(m^1_{1,4},m^1_{2,4},m^1_{3,4})$ is generated by
the monomials in the column labelled ``$i=4$''. But $m^1_{2,4}$ and
$m^1_{3,4}$ are multiples of $m^1_{1,4}$. Therefore,
$M_4(G_0)=(z^{a_{3,4}})$. Again, this is consistent with
Proposition~\ref{gbdesiredform}, since $e_{1,2}$ and $e_{1,3}$ are not
in $\mcb_{1,4}$, so they are superfluous. As for $M_5(G_0)$, note that
$e_{1,1},e_{1,2},e_{1,3}$ are in $\mcb_{1,5}$, but $e_{1,4}$ is
not. Thus
$M_5(G_0)=(x^{a_{1,4}}y^{a_{2,4}},y^{a_{2,1}+a_{2,4}},x^{a_{1,2}+a_{1,4}})$. Similarly,
$M_6(G_0)=(x^{a_{1,4}}z^{a_{3,4}},z^{a_{3,1}+a_{3,4}},x^{a_{1,3}+a_{1,4}})$
and
$M_7(G_0)=(y^{a_{2,4}}z^{a_{3,4}},z^{a_{3,2}+a_{3,4}},y^{a_{2,3}+a_{2,4}})$.
Observe that there are exactly twelve monomial generators, while
$\mcc_2$ is precisely the free module of rank $12$, which is
consistent with the equality
$r_2=\sum_{i=2}^{r_1}\lvert\mcb_{1,i}\rvert$ in
Proposition~\ref{gbdesiredform}. Concretely, the set of $12$ monomial
generators of the ideal quotients $M_i(G_0)$ are:
\begin{eqnarray*} 
\bigcup_{i=2}^{7}\{m^1_{j,i}\mid e_{1,j}\in\mcb_{1,i}\}= \{
  m^1_{1,2},m^1_{1,3},m^1_{1,4},m^1_{1,5},m^1_{2,5},m^1_{3,5},m^1_{1,6},m^1_{2,6},
  m^1_{4,6},m^1_{1,7},m^1_{3,7},m^1_{4,7}\}.
\end{eqnarray*}
Observe that all these $m^1_{j,i}$ are different provided that $L$ is
a $\pcb$ matrix, which is coherent with Remark~\ref{min-m}. On the
other hand, if $L$ is an $\icb$ matrix, which is not a $\pcb$ matrix,
then there may be coincidences (e.g., if $a_{1,4}=0$ and $a_{2,4}=0$,
then $m^1_{1,2}$ and $m^1_{1,3}$ are equal to $1$).

Each monomial generator $m^1_{j,i}$ of $M_i(G_0)$ induces a
corresponding $\tau$-syzygy element $\tau_1(e_{1,i},e_{1,j})$
associated to the $S$-vector of $f_{0,i}$ and $f_{0,j}$. Then
Theorem~\ref{algorithm} states that
\begin{eqnarray*}
&&\{\tau_1(e_{1,2},e_{1,1}),\tau_1(e_{1,3},e_{1,1}),\tau_1(e_{1,4},e_{1,1}),
    \tau_1(e_{1,5},e_{1,1}),\tau_1(e_{1,5},e_{1,2}),\tau_1(e_{1,5},e_{1,3}),\\&&
    \tau_1(e_{1,6},e_{1,1}),\tau_1(e_{1,6},e_{1,2}),\tau_1(e_{1,6},e_{1,4}),
    \tau_1(e_{1,7},e_{1,1}),\tau_1(e_{1,7},e_{1,3}),\tau_1(e_{1,7},e_{1,4})\}
\end{eqnarray*}
is a Gr\"obner basis of $\ker(\dif_1)\subset\mcc_1$ w.r.t. the monomial
ordering on $\mcc_1$ induced by the $\wrlo$ on $\mcc_0=\mbk[x]$ and
the Gr\"obner basis $G_0$. Note that we have enumerated this Gr\"obner
basis according to Remark~\ref{induced-srle}. The proof of
Theorem~\ref{main} and Remark~\ref{enum} say that the following is an
equality of enumerated sets:
\begin{eqnarray*}
G_1:=\dif_2(\mcb_2)&=-\{& \tau_1(e_{1,2},e_{1,1}),
\tau_1(e_{1,3},e_{1,1}),\tau_1(e_{1,4},e_{1,1}),
\tau_1(e_{1,5},e_{1,1}), \\ &&
\tau_1(e_{1,5},e_{1,2}),\tau_1(e_{1,5},e_{1,3}),
\tau_1(e_{1,6},e_{1,1}),\tau_1(e_{1,6},e_{1,2}), \\ &&
\tau_1(e_{1,6},e_{1,4}), \tau_1(e_{1,7},e_{1,1}),
\tau_1(e_{1,7},e_{1,3}),\tau_1(e_{1,7},e_{1,4})\}.
\end{eqnarray*}
More concretely (recall Example~\ref{cyc4}), and underlining the
leading terms (using either \eqref{inducedlm} or Theorem~\ref{main}):
\begin{eqnarray*}
\begin{array}{lcl}
f_{1,1} & = & \dif_{2}(e_{2,1})=\dif_2(23,1,4)\\ & = &
y^{a_{2,1}}z^{a_{3,1}}e_{1,1}-\underline{x^{a_{1,4}}e_{1,2}}-
t^{a_{4,2}+a_{4,3}}e_{1,7}=-\tau_1(e_{1,2},e_{1,1}),\\
f_{1,2} & = & \dif_{2}(e_{2,2})=\dif_2(13,2,4)= \\ & = &
x^{a_{1,2}}z^{a_{3,2}}e_{1,1}-\underline{y^{a_{2,4}}e_{1,3}}-
t^{a_{4,1}+a_{4,3}}e_{1,6}=-\tau_1(e_{1,3},e_{1,1}),\\
f_{1,3} & = & \dif_{2}(e_{2,3})=\dif_2(12,3,4)= \\ & = &
x^{a_{1,3}}y^{a_{2,3}}e_{1,1}-\underline{z^{a_{3,4}}e_{1,4}}-
t^{a_{4,1}+a_{4,2}}e_{1,5}=-\tau_1(e_{1,4},e_{1,1}),\\
f_{1,4} & = & \dif_{2}(e_{2,4})=\dif_2(3,12,4)= \\ & = &
z^{a_{3,1}+a_{3,2}}e_{1,1}-\underline{x^{a_{1,4}}y^{a_{2,4}}e_{1,5}}-
t^{a_{4,3}}e_{1,4}=-\tau_1(e_{1,5},e_{1,1}),\\
f_{1,5} & = & \dif_{2}(e_{2,5})=\dif_2(3,2,14)=\\ & = &
z^{a_{3,2}}e_{1,2}-\underline{y^{a_{2,1}+a_{2,4}}e_{1,5}}-
x^{a_{1,3}}t^{a_{4,3}}e_{1,6}=-\tau_1(e_{1,5},e_{1,2}),\\
f_{1,6} & = & \dif_{2}(e_{2,6})=\dif_2(3,1,24)=\\ & = &
z^{a_{3,1}}e_{1,3}-\underline{x^{a_{1,2}+a_{1,4}}e_{1,5}}-
y^{a_{2,3}}t^{a_{4,3}}e_{1,7}=-\tau_1(e_{1,5},e_{1,3}),\\
f_{1,7} & = & \dif_{2}(e_{2,7})=\dif_2(2,13,4)= \\ & = &
y^{a_{2,1}+a_{2,3}}e_{1,1}-\underline{x^{a_{1,4}}z^{a_{3,4}}e_{1,6}}-
t^{a_{4,2}}e_{1,3}=-\tau_1(e_{1,6},e_{1,1}),\\ 
f_{1,8} & = & \dif_{2}(e_{2,8})=\dif_2(2,3,14)=\\ & = &
y^{a_{2,3}}e_{1,2}-\underline{z^{a_{3,1}+a_{3,4}}e_{1,6}}-
x^{a_{1,2}}t^{a_{4,2}}e_{1,5}=-\tau_1(e_{1,6},e_{1,2}),\\
f_{1,9} & = & \dif_{2}(e_{2,9})=\dif_2(2,1,34)=\\ & = &
y^{a_{2,1}}e_{1,4}-\underline{x^{a_{1,3}+a_{1,4}}e_{1,6}}-
z^{a_{3,2}}t^{a_{4,2}}e_{1,7}=-\tau_1(e_{1,6},e_{1,4}),\\
f_{1,10} & = & \dif_{2}(e_{2,10})=\dif_2(1,23,4)=\\ & = &
x^{a_{1,2}+a_{1,3}}e_{1,1}-\underline{y^{a_{2,4}}z^{a_{3,4}}e_{1,7}}-
t^{a_{4,1}}e_{1,2}=-\tau_1(e_{1,7},e_{1,1}),\\
f_{1,11} & = & \dif_{2}(e_{2,11})=\dif_2(1,3,24)=\\ & = &
x^{a_{1,3}}e_{1,3}-\underline{z^{a_{3,2}+a_{3,4}}e_{1,7}}-
y^{a_{2,1}}t^{a_{4,1}}e_{1,5}=-\tau_2(e_{1,7},e_{1,3}),\\
f_{1,12} & = & \dif_{2}(e_{2,12})=\dif_2(1,2,34)=\\ & = &
x^{a_{1,2}}e_{1,4}-\underline{y^{a_{2,3}+a_{2,4}}e_{1,7}}-
z^{a_{3,1}}t^{a_{4,1}}e_{1,6}=-\tau_1(e_{1,7},e_{1,4}).
\end{array}
\end{eqnarray*}
Let us focus on the first equality. Since $\cyc$ is a complex,
$\dif_1(f_{1,1})=\dif_1(\dif_2(e_{2,1}))=0$. Hence
\begin{eqnarray*}
x^{a_{1,4}}\dif_1(e_{1,2})-y^{a_{2,1}}z^{a_{3,1}}\dif_1(e_{1,1})=
-t^{a_{4,2}+a_{4,3}}\dif_1(e_{1,7}).
\end{eqnarray*}
In other words, 
\begin{eqnarray}\label{standard-exp21}
x^{a_{1,4}}f_{0,2}-y^{a_{2,1}}z^{a_{3,1}}f_{0,1}=-t^{a_{4,2}+a_{4,3}}f_{0,7}.
\end{eqnarray}
On one hand,
$x^{a_{1,4}}f_{0,2}-y^{a_{2,1}}z^{a_{3,1}}f_{0,1}=S(f_{0,2},f_{0,1})$
because $m^1_{1,2}=x^{a_{1,4}}$ and
$m^1_{2,1}=y^{a_{2,1}}z^{a_{3,1}}$. On the other hand,
\eqref{standard-exp21} is a standard expression of
$S(f_{0,2},f_{0,1})$ w.r.t. the $\wrlo$ on $\mcc_0$ and with respect
to $G_0$, and with remainder zero. Indeed, the $S$-vector of $f_{0,2}$
and $f_{0,1}$ is defined so as to cancel the leading terms of
$f_{0,2}$ and $f_{0,1}$. Thus on the left hand side we have just a
binomial, and hence similarly on the right hand side. Therefore, the
leading term of the right hand side must necessarily be equal to the
leading term of the left hand side. This proves that
\eqref{standard-exp21} is a standard expression. One deduces that
$\tau_1(e_{1,2},e_{1,1})=x^{a_{1,4}}e_{1,2}-
y^{a_{2,1}}z^{a_{3,1}}e_{1,1}+t^{a_{4,2}+a_{4,3}}e_{1,7}$ is a
$\tau$-syzygy associated to the $S$-vector of $f_{0,2}$ and $f_{0,1}$.
So $f_{1,1}=\dif_2(e_{2,1})=-\tau_1(e_{1,2},e_{1,1})$, which is
consistent with Proposition~\ref{tau} and Theorem~\ref{main}.

The remainder of the equalities can be treated in the same manner, so
we deduce that $G_1=\dif_2(\mcb_2)$ is a Gr\"obner basis of
$\ker(\dif_1)$ and
\begin{eqnarray*}
{\rm Im}(\dif_2)=\langle\dif_{2}(\mcb_{2})\rangle=\langle
\tau_1(e_{1,i},e_{1,j})\mid i=2,\ldots ,7\mbox{ and
  }e_{1,j}\in\mcb_{1,i} \rangle=\ker(\dif_{1}),
\end{eqnarray*}
and the complex $\mcc$ is exact in degree $1$, as shown in
Theorem~\ref{main}.

We next find the syzygies of $G_1=\dif_2(\mcb_2)=\{f_{1,1},\ldots,
f_{1,12}\}$. To do that, we compute the corresponding module quotients
of Notation~\ref{m-basis-elements}. By definition, $M_1(G_1)=0$. Since
the least common multiple of two monomial terms with different basis
elements is zero, it follows that $M_2(G_1)$, $M_3(G_1)$ and
$M_4(G_1)$ are zero. For the same reason, $M_7(G_1)=0$ and
$M_{10}(G_1)=0$.  Note that $M_5(G_1)$ is generated by
$m^2_{4,5}=-x^{a_{1,4}}$. (Remark this negative sign, which will be
the reason for the negative sign in the subsequent equality
$\dif_3(e_{3,1})=-\tau_2(e_{2,5},e_{2,4})$.) Next,
$M_6(G_1)=(m^2_{4,6},m^2_{5,6})$. Since $e_{2,5}\not\in\mcb_{2,6}$,
then $m^2_{5,6}$ is superfluous. Thus
$M_6(G_1)=(m^2_{4,6})=(y^{a_{2,4}})$. It is easy to check that the
nonzero module quotients $M_i(G_1)$ are:
\begin{eqnarray*}
&&M_5(G_1)=(m^2_{4,5})=(x^{a_{1,4}})\mbox{ , }
M_6(G_1)=(m^2_{4,6})=(y^{a_{2,4}})\mbox{ , }\\
&&M_8(G_1)=(m^2_{7,8})=(x^{a_{1,4}})\mbox{ , }
M_9(G_1)=(m^2_{7,9})=(z^{a_{3,4}})\mbox{ , }\\
&&M_{11}(G_1)=(m^2_{10,11})=(y^{a_{2,4}})\mbox{ and }
M_{12}(G_1)=(m^2_{10,12})=(z^{a_{3,4}}).
\end{eqnarray*}
Observe that there are exactly six monomial generators, while $\mcc_3$
is precisely the free module of rank $6$, which is consistent with the
equality $r_3=\sum_{i=2}^{r_2}\lvert\mcb_{2,i}\rvert$ in
Proposition~\ref{gbdesiredform}. Concretely, the set of $6$ monomial
generators of the ideal quotients $M_i(G_1)$ are:
\begin{eqnarray*}
\bigcup_{i=2}^{12}\{m^2_{j,i}\mid e_{2,j}\in\mcb_{2,i}\}=\{
m^2_{4,5},m^2_{4,6},m^2_{7,8},m^2_{7,9},m^2_{10,11},m^2_{10,12}\}.
\end{eqnarray*}
Each monomial generator $m^2_{j,i}$ of $M_i(G_1)$ produces a
corresponding $\tau$-syzygy element $\tau_2(e_{2,i},e_{2,j})$
associated to the $S$-vector of $f_{1,i}$ and $f_{1,j}$. Then
Theorem~\ref{algorithm} affirms that
\begin{eqnarray*}
\{\tau_2(e_{2,5},e_{2,4}),\tau_2(e_{2,6},e_{2,4}),\tau_2(e_{2,8},e_{2,7}),
\tau_2(e_{2,9},e_{2,7}),\tau_2(e_{2,11},e_{2,10}),\tau_2(e_{2,12},e_{2,10})\}
\end{eqnarray*}
is a Gr\"obner basis of $\ker(\dif_2)\subset\mcc_2$. Again, we have
enumerated this Gr\"obner basis according to Remark~\ref{induced-srle}.
Moreover, Theorem~\ref{main} and Remark~\ref{enum} say that the
following is an equality of enumerated sets:
\begin{eqnarray*}
G_2 & := & \dif_3(\mcb_3)\\ & = &
-\{\tau_2(e_{2,5},e_{2,4}),\tau_2(e_{2,6},e_{2,4}),\tau_2(e_{2,8},e_{2,7}),
\tau_2(e_{2,9},e_{2,7}),\tau_2(e_{2,11},e_{2,10}),\tau_2(e_{2,12},e_{2,10})\}.
\end{eqnarray*}

More concretely (recall Example~\ref{cyc4}), and underlining the
leading terms (using either \eqref{inducedlm} or Theorem~\ref{main}):

\begin{eqnarray*}
\begin{array}{lcl}
f_{2,1} & = & \dif_3(e_{3,1})=\dif_3(3,2,1,4)\\ & = &
z^{a_{3,2}}e_{2,1}-y^{a_{2,1}}e_{2,4}+
\underline{x^{a_{1,4}}e_{2,5}}-t^{a_{4,3}}e_{2,9}=-\tau_2(e_{2,5},e_{2,4}),\\ 
f_{2,2} & = & \dif_3(e_{3,2})=\dif_3(3,1,2,4)= \\ & = &
z^{a_{3,1}}e_{2,2}-x^{a_{1,2}}e_{2,4}+
\underline{y^{a_{2,4}}e_{2,6}}-t^{a_{4,3}}e_{2,12}
=-\tau_2(e_{2,6},e_{2,4}),\\ 
f_{2,3} & = & \dif_3(e_{3,3})=\dif_3(2,3,1,4)= \\ & = &
y^{a_{2,3}}e_{2,1}-z^{a_{3,1}}e_{2,7}+
\underline{x^{a_{1,4}}e_{2,8}}-t^{a_{4,2}}e_{2,6}
=-\tau_2(e_{2,8},e_{2,7}),
\end{array}
\end{eqnarray*}
\begin{eqnarray*}
\begin{array}{lcl} 
f_{2,4} & = & \dif_3(e_{3,4})=\dif_3(2,1,3,4)= \\ & = &
y^{a_{2,1}}e_{2,3}-x^{a_{1,3}}e_{2,7}
+\underline{z^{a_{3,4}}e_{2,9}}-t^{a_{4,2}}e_{2,11}
=-\tau_2(e_{2,9},e_{2,7}),\\
f_{2,5} & = & \dif_3(e_{3,5})=\dif_3(1,3,2,4)= \\ & = &
x^{a_{1,3}}e_{2,2}-z^{a_{3,2}}e_{2,10}+
\underline{y^{a_{2,4}}e_{2,11}}-t^{a_{4,1}}e_{2,5}
=-\tau_2(e_{2,11},e_{2,10}),\\
f_{2,6} & = & \dif_3(e_{3,6})=\dif_3(1,2,3,4)= \\ & = &
x^{a_{1,2}}e_{2,3}-y^{a_{2,3}}e_{2,10}+
\underline{z^{a_{3,4}}e_{2,12}}-t^{a_{4,1}}e_{2,8}
=-\tau_2(e_{2,12},e_{2,10}).
\end{array}
\end{eqnarray*}

Again, let us focus on the first equality. Since $\cyc$ is a complex,
$\dif_2(f_{2,1})=\dif_2(\dif_3(e_{3,1}))=0$. Hence
\begin{eqnarray*}
-x^{a_{1,4}}\dif_{2}(e_{2,5})+y^{a_{2,1}}\dif_{2}(e_{2,4})=
z^{a_{3,2}}\dif_2(e_{2,1})-t^{a_{4,3}}\dif_2(e_{2,9}).
\end{eqnarray*}
In other words, 
\begin{eqnarray}\label{standard-exp31}
-x^{a_{1,4}}f_{1,5}+y^{a_{2,1}}f_{1,4}=z^{a_{3,2}}f_{1,1}-t^{a_{4,3}}f_{1,9}.
\end{eqnarray}
Note that $-x^{a_{1,4}}f_{1,5}+y^{a_{2,1}}f_{1,4}=S(f_{1,5},f_{1,4})$,
because $m^2_{4,5}=-x^{a_{1,4}}$ and $m^2_{5,4}=-y^{a_{2,1}}$. On the
other hand, \eqref{standard-exp31} is a standard expression of
$S(f_{1,5},f_{1,4})$ with respect to $G_1$, with remainder
zero. Indeed, the $S$-vector of $f_{1,5}$ and $f_{1,4}$ is defined so
as to cancel the leading terms of $f_{1,5}$ and $f_{1,4}$. Recall that
\begin{eqnarray*}
&&f_{1,5}=z^{a_{3,2}}e_{1,2}-\underline{y^{a_{2,1}+a_{2,4}}e_{1,5}}-
x^{a_{1,3}}t^{a_{4,3}}e_{1,6},\\
&&f_{1,4}=z^{a_{3,1}+a_{3,2}}e_{1,1}-\underline{x^{a_{1,4}}y^{a_{2,4}}e_{1,5}}-
t^{a_{4,3}}e_{1,4},\\
&&f_{1,1}=y^{a_{2,1}}z^{a_{3,1}}e_{1,1}-\underline{x^{a_{1,4}}e_{1,2}}-
t^{a_{4,2}+a_{4,3}}e_{1,7},\\
&&f_{1,9}=y^{a_{2,1}}e_{1,4}-\underline{x^{a_{1,3}+a_{1,4}}e_{1,6}}-
z^{a_{3,2}}t^{a_{4,2}}e_{1,7}.
\end{eqnarray*}
Note that the leading term of $z^{a_{3,2}}f_{1,1}$ and the leading
term of $t^{a_{4,3}}f_{1,9}$ coincide with the two summands of
$-x^{a_{1,4}}f_{1,5}-\lt(-x^{a_{1,4}}f_{1,5})$, as stated in the proof
Proposition~\ref{tau}. It follows that \eqref{standard-exp31} is a
standard expression, so one deduces that
\begin{eqnarray*}
\tau_2(e_{2,5},e_{2,4})=-x^{a_{1,4}}e_{2,5}+y^{a_{2,1}}e_{2,4}-z^{a_{3,2}}e_{2,1}
+t^{a_{4,3}}e_{2,9}
\end{eqnarray*}
is a $\tau$-syzygy associated to the $S$-vector of $f_{1,5}$ and
$f_{1,4}$. Thus, $f_{2,1}=-\tau_2(e_{2,5},e_{2,4})$, which is
consistent with Proposition~\ref{tau} and Theorem~\ref{main}. 

The remainder of the equalities can be treated in the same manner, so
we deduce that $G_2=\dif_3(\mcb_3)$ is a Gr\"obner basis of
$\ker(\dif_2)$ and
\begin{eqnarray*}
{\rm Im}(\dif_3)=\langle\dif_{3}(\mcb_{3})\rangle=\langle
\tau_2(e_{2,i},e_{2,j})\mid i=2,\ldots ,12\mbox{ and
  }e_{2,j}\in\mcb_{2,i} \rangle=\ker(\dif_{2}),
\end{eqnarray*}
and the complex $\mcc$ is exact in degree $2$, as shown in
Theorem~\ref{main}.

Finally, observe that, since all the leading terms of $f_{2,i}$ have
distinct basis elements, then the module quotients $M_i(G_2)$ are all
zero (see Notation~\ref{m-basis-elements}). Therefore, there are no
syzygies among the elements of $G_2$ and $\dif_3:\mcc_3\to\mcc_2$ is
injective, from which follows the exactness of the $\mcc$ complex in
degree $3$, which is consistent with Remark~\ref{exactnessn-1}.

\bigskip
\bigskip

\textit{Acknowledgements.} In developing this research, we have
benefitted from the hospitality of Departament de Matem\`{a}tiques,
Universitat Polit\`{e}cnica de Catalunya; the School of Mathematics,
University of Edinburgh; and ICMS, Edinburgh. We gratefully
acknowledge financial support from the research grant MTM2015-69135-P
and also grants from ICMS, Edinburgh, under their `Research in Groups'
scheme.

In particular, we would like to thank F.O. Schreyer for the stimulus
provided by the several informative conversations we had at the June
$2015$ meeting on ``Minimal free resolutions, Betti numbers, and
combinatorics'', held at ICMS, Edinburgh.

We would also like to thank the referees for their helpful and
interesting reports which have improved the present work.

\small{

}


\vspace{1cm}

{\footnotesize\sc

\noindent School of Mathematics, University of Edinburgh, James Clerk
Maxwell Building, \newline Peter Guthrie Tait Road, Edinburgh EH9 3FD,
Scotland. 

\vspace{0.4cm}

\noindent Departament de Matem\`atiques, ETSEIB, Universitat
Polit\`ecnica de Catalunya. \newline Diagonal 647, ETSEIB, 08028
Barcelona, Catalunya.  {\em francesc.planas@upc.edu}
}

\end{document}